\documentclass[12pt]{article}
\usepackage{makeidx}
\usepackage{multirow}
\usepackage{multicol}
\usepackage[dvipsnames,svgnames,table]{xcolor}
\usepackage{graphicx}
\usepackage{epstopdf}
\usepackage{ulem}
\usepackage{hyperref}
\usepackage{amssymb}
\usepackage{amsmath,amsthm,cite}
\usepackage[affil-it]{authblk}

\title{}
\usepackage[paperwidth=612pt,paperheight=792pt,top=72pt,right=72pt,bottom=72pt,left=72pt]{geometry}

\makeatletter
	{\par\setlength{\parindent}{#3}
	\setlength{\leftmargin}{#1}       \setlength{\rightmargin}{#2}%
	\advance\linewidth -\leftmargin       \advance\linewidth -\rightmargin%
	\advance\@totalleftmargin\leftmargin  \@setpar{{\@@par}}%
	\parshape 1\@totalleftmargin \linewidth\ignorespaces}{\par}%
\makeatother 

\usepackage{calrsfs}

\usepackage{tikz}
\usetikzlibrary{arrows,chains,matrix,positioning,scopes}
\makeatletter
\tikzset{join/.code=\tikzset{after node path={%
\ifx\tikzchainprevious\pgfutil@empty\else(\tikzchainprevious)%
edge[every join]#1(\tikzchaincurrent)\fi}}}
\makeatother
\tikzset{>=stealth',every on chain/.append style={join},
         every join/.style={->}}

\newtheorem{theorem}{Theorem}[section]
\newtheorem{lemma}[theorem]{Lemma}

\theoremstyle{definition}

\newtheorem{Definition}[theorem]{Definition}

\newtheorem{proposition}[theorem]{Proposition}

\title{Strong Shape Theory of Continuous Maps  {\footnote{The research is partially supported by the institutional scientific research project of Batumi Shota Rustaveli State University (ISRP-2019). The paper is completed and written during A. Beridze's visit at the University of California, Santa Barbara (UCSB) from 15.09.2019 to 15.12.2019.}} }
\date{\vspace{-5ex}}

\author{V.Baladze, A. Beridze, R.Tsinaridze}
\affil{Department of Mathematics\\Batumi Shota Rustaveli  State University}

\numberwithin{equation}{section}
\begin{document}
\maketitle
\begin{abstract}
The work is motivated by the papers \cite{1}, \cite{2}, \cite{7}, \cite{11}, \cite{13} and \cite{14}. In particular, the strong homology groups of continuous maps were defined and studied in \cite{13} and \cite{14}. To show that the given groups are a homology  type functor, it was required to construct a corresponding shape category. In this paper, we study this very problem. In particular, using the methods developed in \cite{7}, \cite{19}, the strong shape theory of continuous maps of compact metric spaces,  the so-called strong fiber shape theory is constructed.
\end{abstract}
\section*{Introduction}
\

The idea of the expansion of a map into the inverse or direct system consisting of “good” maps has been successfully used by various mathematicians to solve various problems of general topology, geometric topology and algebraic topology \cite{7}, \cite{12}, \cite{15}, \cite{16}, \cite{17}, \cite{18}, \cite{19}, \cite{22}. Using the idea of the papers (\cite{1} -\cite{11}), \cite{15}, \cite{16}, \cite{17}, \cite{18} continuous maps are investigated from the point of view of homology and homotopy theories. Applying  the (co)shape properties of continuous maps, functors from the category of maps of topological spaces to the category of long exact sequences of groups were considered by V.Baladze \cite{11}. On the other hand, the projective and the strong homology groups of continuous maps of compact metric spaces were defined in the papers \cite{24}, \cite{13}. The connection between the spectral and strong homology groups of maps was studied in the paper \cite{14}. Our further purpose is axiomatic characterization of it, without relative strong homology groups, in the sense of Hu \cite{hu}. In this case, the theory of inverse systems plays an important role. Consequently, the main aim of this paper is to develop fiber strong shape theory of continuous maps.   

We construct the fiber strong shape classification of maps using the methods of inverse system theory. It consists of approximation of maps by maps of $\mathbf{ANR}$-spaces. There exist another approaches to fiber strong shape theory, which are analogous to the technique considered in \cite{n1}, \cite{n2}, \cite{n3}, \cite{n4}, \cite{n5}, \cite{n6}, \cite{n7}, \cite{n8} and lead to equivalent theories. In this paper we use the method of Marde\v{s}i\'{c}-Lisica \cite{19}, which is more geometric and is connected to construction of strong homology groups.

As it is known, in the process of constructing the general shape category of topological spaces, the main step is to show that any resolution of the space is an expansion of the given space. For constructing the strong shape theory of topological spaces, it is an important fact that any resolution of the space is a strong expansion and any strong expansion is a coherent expansion \cite{17}, \cite{18}, \cite{19}. In the paper \cite{7} the fiber resolution and fiber expansion of continuous maps is defined and it is shown that any fiber resolution is a fiber expansion. In this paper we will define a strong fiber expansion. We will modify some lemmas and theorems of \cite{7}, \cite{17}, \cite{18} and will show that any fiber resolution is a strong fiber expansion. Besides, we will prove an analogous lemma of the main lemma on strong expansions \cite{17}. Using the obtained results and methods of strong shape theory, we will construct a strong fiber shape category of maps of compact metric spaces.

\

In this paper we will use the following notations and notions.

\

Let $\mathbf{Top}$ be the category of topological spaces and continuous maps. Denote by $\mathbf{MOR_{Top}}$ the category of morphisms of the category $\mathbf{Top}$. Therefore, any continuous map $f:X \to X'$ is an object of the category $\mathbf{MOR_{Top}}$ and if $f:X \to X'$ and $g:Y \to Y'$ are two objects of this category, then a pair $(\varphi, \varphi ')$ of continuous maps $\varphi :X \to Y$ and $\varphi ' :X' \to Y'$ is a morphism of $\mathbf{MOR_{Top}}$  if the condition $f' \circ \varphi =\varphi ' \circ f$ is fulfilled. Let $\mathbf{CM}$ be the category of compact metric spaces and continuous maps. Denote by $\mathbf{MOR_{CM}}$ the full subcategory of the category $\mathbf{MOR_{Top}}$ objects of which are continuous maps of compact metric spaces. In the case of the category $\mathbf{M}$ of metric spaces, $\mathbf{MOR_M}$ denotes the corresponding full subcategory of the category $\mathbf{MOR_{Top}}$.

\

Two morphisms $(\varphi,\varphi '), (\psi.\psi'):f \to f'$ of the category $\mathbf{MOR_{Top}}$ are called homotopic if there is a morphism 
\begin{equation}
(\Theta,\Theta'):f \times 1_I \to f'
\end{equation}
such that $\Theta$ is a homotopy from $\varphi$ to $\psi $ and $\Theta'$ is a homotopy from $\varphi '$ to $\psi '$ \cite{7}.

\

A submap of a map $f:X \to X'$ is a map $g:A \to A'$, where $A \subset X$, $A' \subset X'$ and $f_{|A}=g$. The submap is called closed (open) if $A$ and $A'$ are closed (open) subspaces of  $X$ and  $X'$, respectively \cite{7}.

\

Let $g:A \to A'$ be a submap of a map $f:X \to X'$ and $(\varphi, \varphi '):f \to g$ be a morphism. The pair $(\varphi_{|A}, \varphi '_{|A'})$ of restrictions $\varphi_{|A}:A \to Y$ and $\varphi'_{|A'}:A' \to Y'$ is called a restriction of morphism $(\varphi, \varphi ')$ on the submap $g:A \to A'$ and this pair is denoted by $(\varphi, \varphi ')_{|g}$ \cite{7}.

\

A morphism $(i,i'):g \to f$ is said to be an embedding. If  $i:A \to X$ and $i':A' \to X'$ are both embeddings. If both $i:A \to X$ and $i':A' \to X'$ are closed maps, then $(i,i'):g \to f$ is said to be a closed embedding \cite{7}.

\

A morphism $(\varphi, \varphi '):f \to g$ is said to be constant if $\varphi : X \to Y$ and $\varphi ' : X' \to Y'$ are constant maps \cite{7}.

\

Let $g:A \to A'$ be a submap of a map $f:X \to X'$. A submap $f_U:U \to U'$ of a map $f$ is said to be a neighborhood of $g$ in $f$, if $U$ is an open neighborhood of $A$ in $X$ and  $U'$ is an open neighborhood of $A'$ in $X'$ \cite{7}.

\

Let $g:A \to A'$ be a submap of a map $f:X \to X'$ and $(\varphi,\varphi ' ):g \to h$ be a morphism. A morphism $(\tilde{\varphi}, \tilde{\varphi} '):f \to h$ is said to be an extension of $(\varphi, \varphi ')$, if $(\tilde{\varphi}, \tilde{\varphi} '  )_{|g}=(\varphi, \varphi ' )$ \cite{7}.

\

Let $(i,i'):g \to f$ be an embedding. The g is said to be a retract of  $f$ provided that there exists a morphism $(r,r'):f \to g$ such that $(r,r') \circ (i,i')=(1_X,1_{X'})$ \cite{7}.

\

A submap $g$ is said to be a neighborhood retract of  $f$, if it is a retract of some neighborhood $f_U$ of $g$ in $f$ \cite{7}.

\

A map $f:X \to X'$ of the category $\mathbf{MOR_M}$ is said to be an absolute retract for the category $\mathbf{MOR_M}$, if for each closed embedding $(i,i'):g \to f \in \mathbf{MOR_M}$ there exists a retraction $(r,r'):f \to (i,i')(g)$ \cite{7}.

\

A map $f:X \to X'$ of the category $\mathbf{MOR_M}$ is said to be an absolute neighborhood retract for the category $\mathbf{MOR_M}$, if for each closed embedding $(i,i'):g \to f \in \mathbf{MOR_M}$ there exists a neighborhood $f_U:U \to U'$ of  $(i,i')(g)$ in $f$ and a retraction $(r,r):f_U \to (i,i' )(g)$ \cite{7}.

\

Let $\mathbf{AR(MOR_M)}$ and $\mathbf{ANR(MOR_M)}$ be the category of all absolute retracts and all absolute neighborhood retracts for the category $\mathbf{MOR_M}$, respectively. Note that if $f \in \mathbf{A(N)R(MOR_M)}$, then $f$ is called an $\mathbf{A(N)R(MOR_M)}$-map. From now, we call an $\mathbf{A(N)R(MOR_M)}$-map an $\mathbf{A(N)R}$-map.

\section{Fiber~ resolution ~and~ strong fiber expansion of a continuous map}
\

Let $\mathbf{f}=\left\{f_{\lambda } ,\left(p_{\lambda ,\lambda '} ,p'_{\lambda ,\lambda '} \right)\; ,{\rm \Lambda }\right\}$ be an inverse system in the category $\mathbf{Mor_{Top} }$ of continuous maps of topological spaces. Let $f=\left\{f\right\}$ be a rudimentary system whose term is just only a map  $f:X\to X'$.
\begin{Definition} (see \cite{7}) A fiber resolution of a map $f$ is a morphism  $\mathbf{\left(p,p'\right)}=\left\{\left(p_{\lambda } ,p'_{\lambda } \right)\right\}:f\to \mathbf{f}$ of the category $\mathbf{pro-Mor_{Top}} $\textbf{ }which for any $\mathbf{ANR}$-map $t:P\to P'$ and a pair $\left(\alpha ,\alpha '\right)$ of coverings $\alpha \in Cov\left(P\right)$ and $\alpha '\in Cov\left(P'\right)$, satisfies the following two conditions:

FR1) for every morphism $\left(\varphi ,\varphi '\right):f\to t$ there exist $\lambda \in {\rm \Lambda }$ and a morphism $\left(\varphi _{\lambda } ,\varphi' _{\lambda } \right):f_{\lambda } \to t$ such that $\left(\varphi _{\lambda } ,\varphi ' _{\lambda } \right)\circ \left(p_{\lambda } ,p'_{\lambda } \right)$ and $\left(\varphi ,\varphi '\right)$ are $\left(\alpha ,\alpha '\right)$-near;

FR2) there exists a pair $\left(\beta ,\beta '\right)$ of coverings $\beta \in Cov\left(P\right)$ and $\beta '\in Cov\left(P'\right)$ with the following property: if $\lambda \in {\rm \Lambda }$ and $\left(\varphi _{\lambda } ,\varphi ' _{\lambda } \right),\left(\psi _{\lambda } ,\psi ' _{\lambda } \right):f_{\lambda } \to t$ are morphisms such that the morphisms $\left(\varphi _{\lambda } ,\varphi ' _{\lambda } \right)\circ \left(p_{\lambda } ,p ' _{\lambda } \right)$  and $\left(\psi _{\lambda } ,\psi ' _{\lambda } \right)\circ \left(p_{\lambda } ,p ' _{\lambda } \right)$ are $\left(\beta ,\beta '\right)$-near, then there exists a $\lambda '\ge \lambda $ such that $\left(\varphi _{\lambda } ,\varphi ' _{\lambda } \right)\circ \left(p_{\lambda \lambda '} ,p ' _{\lambda \lambda '} \right)$  and $\left(\psi _{\lambda } ,\psi ' _{\lambda } \right)\circ \left(p_{\lambda \lambda '} ,p'_{\lambda \lambda '} \right)$ are $\left(\alpha ,\alpha '\right)$-near.
\end{Definition}
If in a fiber resolution $\mathbf{ \left(p,p'\right)}:f\to \mathbf{f}$ each $f_{\lambda } $ is an $\mathbf{ANR}$-map, then this fiber resolution is called and $\mathbf{ANR}$-resolution.

In the paper \cite{7} it is shown that any continuous map admits an $\mathbf{ANR}$-fiber resolution (see theorem 3.2 of \cite{7}). In this section our aim is to define a strong fiber expansion of a continuous map and to prove that any $\mathbf{ANR}$-fiber resolution is a strong fiber expansion. For this aim we need some modification of lemma 3.4 of \cite{7} and analogous result of  lemma 1 of \cite{17}.

\begin{Definition} We will say that a morphism  $\mathbf{\left(p,p'\right)}=\left\{\left(p_{\lambda } ,p'_{\lambda } \right)\; \right\}:f\to \mathbf{f}$ of the category $\mathbf{pro-Mor_{Top}} $ is a strong fiber expansion of a continuous map $f:X\to X'$, if for every $\mathbf{ANR}$-map $t:P\to P'$ the following conditions are fulfilled:

SF1) for every morphism $\left(\varphi ,\varphi '\right):f\to t$ there exist $\lambda \in {\rm \Lambda }$ and a morphism $\left(\varphi _{\lambda } ,\varphi ' _{\lambda } \right):f_{\lambda } \to t$ such that
\begin{equation} \label{GrindEQ__2_1_} 
\left(\varphi ,\varphi '\right)\cong \left(\varphi _{\lambda } ,\varphi ' _{\lambda } \right) \circ \left(p_{\lambda } ,p ' _{\lambda } \right). 
\end{equation} 

SF2) if $\lambda \in {\rm \Lambda }$ and $\left(\varphi _{\lambda } ,\varphi ' _{\lambda } \right),\left(\psi _{\lambda } ,\psi ' _{\lambda } \right):f_{\lambda } \to t$ are morphisms such that the morphisms $\left(\varphi _{\lambda } ,\varphi ' _{\lambda } \right)\circ \left(p_{\lambda } ,p '_{\lambda } \right)$  and $\left(\psi _{\lambda } ,\psi ' _{\lambda } \right)\circ \left(p_{\lambda } ,p'_{\lambda } \right)$ are connected by a fiber homotopy $\left( \Theta , \Theta'\right):f\times 1_{I} \to t$, then there exist a $\lambda '\ge \lambda $ and fiber homotopies $\left({\rm \Delta },{\rm \Delta }'\right):f_{\lambda } \times 1_{I} \to t$ and $\left({\rm \Gamma },{\rm \Gamma }'\right):f\times 1_{I} \times 1_{I} \to t$ such that the homotopy  $\left({\rm \Delta },{\rm \Delta }'\right)$ connects  the morphisms $\left(\varphi _{\lambda } ,\varphi ' _{\lambda } \right)\circ \left(p_{\lambda \lambda '} ,p'_{\lambda \lambda '} \right)$  and $\left(\psi _{\lambda } ,\psi ' _{\lambda } \right)\circ \left(p_{\lambda \lambda '} ,p'_{\lambda \lambda '} \right)$ and the homotopy $\left({\rm \Gamma },{\rm \Gamma }'\right)$ connects $\left({\rm \Theta },{\rm \Theta }'\right)$ and $\left({\rm \Delta },{\rm \Delta }'\right)\circ \left(p_{\lambda } \times 1_{I} ,p'_{\lambda } \times 1_{I} \right)$ and  is fixed on the submap  $f\times 1_{\partial I} :X\times 1_{\partial I} \to X'\times 1_{\partial I} $.
\end{Definition}

Let $C\left(Z,P\right)$ be the space of all continuous functions from $Z$ to $P\; $ endowed with the compact-open topology. For any continuous map $t:P\to P'$ and any topological space $Z$ denote by $t^{\# } :C\left(Z,P\right)\to C\left(Z,P'\right)$ the map which is defined by the formula 
\begin{equation} \label{GrindEQ__2_2_} 
t^{\# } \left(f\right)=t\circ f,\; \; \; \; \; \; \forall \; f\in C\left(Z,P\right).                                  
\end{equation} 
\begin{proposition}
Let $t:P\to P'$ be an $\mathbf{ANR}$-map and let $Z$ be a compact metric space. Then the map $t^{\# } :C\left(Z,P\right)\to C\left(Z,P'\right)$ is an $\mathbf{ANR}$-map.
\end{proposition}
This proposition is proved in \cite{7}.

\begin{lemma}
 Let $f:X\to X'$ be a map of topological spaces, $t':P_{1} \to P_{1}^{'} $, $t:P\to P'$ be $\mathbf{ANR}$-maps. If $\left(\zeta ,\zeta '\right):f\to t'$,  $\left(\xi ,\xi '\right),\; \left(\eta ,\eta '\right):t'\to t$ are morphisms and $\left({\rm \Theta },{\rm \Theta }'\right):f\times 1_{I} \to t$ is a homotopy which connects the morphisms $\left(\xi ,\xi '\right)\circ \left(\zeta ,\zeta '\right)$  and $\left(\eta ,\eta '\right)\circ \left(\zeta ,\zeta '\right)$, then there exist an $\mathbf{ANR}$-map $t'':P_{2} \to P_{2}^{'} $, a morphism $\left(\sigma ,\sigma '\right):f\to t''$,  $\left(\kappa ,\kappa '\right):t''\to t'$ and a fiber homotopy $\left({\rm \Delta },{\rm \Delta }'\right):t''\times 1_{I} \to t$ such that 
\begin{equation} \label{GrindEQ__2_3_} 
\left(\zeta ,\zeta '\right)=\left(\kappa ,\kappa '\right)\circ \left(\sigma ,\sigma '\right),                                   
\end{equation} 
\begin{equation} \label{GrindEQ__2_4_} 
\left({\rm \Delta },{\rm \Delta }'\right):\left(\xi ,\xi '\right)\circ \left(\kappa ,\kappa '\right)\cong \left(\eta ,\eta '\right)\circ \left(\kappa ,\kappa '\right),                   
\end{equation} 
\begin{equation} \label{GrindEQ__2_5_} 
\left({\rm \Theta },{\rm \Theta }'\right)=\left({\rm \Delta },{\rm \Delta }'\right)\circ \left(\sigma \times 1_{I} ,\sigma _{1}^{'} \times 1_{I} \right).                        
\end{equation} 
\end{lemma}

Note that lemma 1.4 without \eqref{GrindEQ__2_5_} is proved in \cite{7}. The fiber homotopy $\left({\rm \Delta },{\rm \Delta }'\right):t''\times 1_{I} \to t$ is constructed there as well. So we just check that \eqref{GrindEQ__2_5_} is fulfilled.

\begin{proof}
Consider the $\mathbf{ANR}$-map $t^{\# } :C\left(I,P\right)\to C\left(I,P'\right)$ and define a morphism $\left(\mu ,\mu '\right):f\to t^{\# } $ by
\begin{equation} \label{GrindEQ__2_6_} 
\mu \left(x\right)\left(t\right)={\rm \Theta }\left(x,t\right),\; \; \; \; \; \; \; \; \; \; \; \; \forall \; x\in X,\; t\in I, 
\end{equation} 
\begin{equation} \label{GrindEQ__2_7_} 
\mu '\left(x'\right)\left(t\right)={\rm \Theta }'\left(x',t\right),\; \; \; \; \; \; \; \forall \; x'\in X',\; t\in I. 
\end{equation} 
In this case $t^{\# } \circ \mu =\mu '\circ f$  \cite{7}
. Now define a morphism $\left(\sigma ,\sigma '\right):f\to t'\times t^{\# } $ by
\begin{equation} \label{GrindEQ__2_8_} 
\sigma \left(x\right)=\left(\zeta \left(x\right),\mu \left(x\right)\right)\; \; \; \; \; \; \; \; \; \; \; \forall \; x\in X, 
\end{equation} 
\begin{equation} \label{GrindEQ__2_9_} 
\sigma '\left(x'\right)=\left(\zeta '\left(x'\right),\mu '\left(x'\right)\; \right)\; \; \; \; \; \; \; \; \; \; \; \forall \; x'\in X'. 
\end{equation} 
It is shown that $\left(t'\times t^{\# } \right)\circ \sigma =\sigma '\circ f$ \cite{7}.

Consider the projections $\kappa :P_{1} \times C\left(I,P\right)\to P_{1} $,  $\kappa ':P_{1}^{'} \times C\left(I,P'\right)\to P_{1}^{'} $ and the corresponding morphism $\left(\kappa ,{\rm \; }\kappa '\right):t'\times t^{\# } \to t'$  \cite{7}.

Let 
\begin{equation} \label{GrindEQ__2_10_} 
P_{2} =\left\{\left(y,\varphi \right)\in P_{1} \times C\left(I,P\right){\rm |}\varphi \left(0\right)=\xi \left(y\right),\; \varphi \left(1\right)=\xi \left(y\right)\right\}, 
\end{equation} 
\begin{equation} \label{GrindEQ__2_11_} 
P_{2}^{'} =\left\{\left(z,\psi \right)\in P_{1}^{'} \times C\left(I,P'\right){\rm |}\psi \left(0\right)=\xi '\left(y\right),\; \psi \left(1\right)=\xi '\left(y\right)\right\}. 
\end{equation} 
In the paper \cite{7} it is shown that $\left(t'\times t^{\# } \right)\left(P_{2} \right)\subset P_{2}^{'} $ and the sets $\sigma \left(X\right)$ and $\sigma '\left(X'\right)$ are the subsets of $P_{2} $ and $P_{2}^{'} $, respectively.  Consider the restrictions 
\begin{equation} \label{GrindEQ__2_12_} 
t"=\left(t'\times t^{\# } \right)_{\left|P_{2} \right. } :P_{2} \to P_{2}^{'} , 
\end{equation} 
\begin{equation} \label{GrindEQ__2_13_} 
\left(\kappa ,{\rm \; }\kappa '\right)=\left(\kappa ,{\rm \; }\kappa '\right)_{\left|t"\right. } :t"\to t'. 
\end{equation} 
Let $\left({\rm \Delta },{\rm \Delta }'\right)=t"\times 1_{I} \to t'$ be given by
\begin{equation} \label{GrindEQ__2_14_} 
{\rm \Delta }\left(\left(y,\varphi \right),t\right)=\varphi \left(t\right),\; \; \; \; \forall \; \left(y,\varphi \right)\in P_{2} , 
\end{equation} 
\begin{equation} \label{GrindEQ__2_15_} 
{\rm \Delta }'\left(\left(z,\psi \right),t\right)=\psi \left(t\right),\; \; \; \; \forall \; \left(z,\psi \right)\in P_{2}^{'} . 
\end{equation} 
Our aim is to show that \eqref{GrindEQ__2_5_} is fulfilled. Indeed,
\begin{equation} \label{GrindEQ__2_16_} 
{\rm \Delta }\circ \left(\sigma \times 1_{I} \right)\left(x,t\right)={\rm \Delta }\left(\sigma \left(x\right),t\right)={\rm \Delta }\left(\left(\zeta \left(x\right),\mu \left(x\right)\right),t\right)=\mu \left(x\right)\left(t\right)={\rm \Theta }\left(x,t\right), 
\end{equation} 
\begin{equation} \label{GrindEQ__2_17_} 
{\rm \Delta }'\circ \left(\sigma {'\; } \times 1_{I} \right)\left(x',t\right)={\rm \Delta }'\left(\sigma '\left(x'\right),t\right)={\rm \Delta }'\left(\left(\zeta '\left(x'\right),\mu '\left(x'\right)\right),t\right)=\mu '\left(x'\right)\left(t\right)={\rm \Theta }'\left(x',t\right). 
\end{equation} 
\end{proof}

\begin{theorem}
Let $t:P\to P'$ be an $\mathbf{ANR}$-map. Then every pair $\left(\alpha ,\alpha '\right)$ of coverings $\alpha $ and $\alpha '$ of  $P$ and $P'$, respectively, admits a pair $\left(\beta ,\beta '\right)$ of coverings of  $P$ and $P'$, respectively,  such that for any two $\left(\alpha ,\alpha '\right)$-near morphisms $\left(\varphi ,\varphi '\right),\left(\psi ,\psi '\right):f\to t$  from a map $f:X\to X'$ of arbitrary topological spaces into $t:P\to P'$, there exists a $\left(\beta ,\beta '\right)$-homotopy $\left({\rm \Theta },{\rm \Theta }'\right):\left(\varphi ,\varphi '\right)\cong \left(\psi ,\psi '\right)$. Moreover, if for a given point $x\in X$, $\varphi \left(x\right)=\psi \left(x\right),$ then
\begin{equation} \label{GrindEQ__2_18_} 
\left({\rm \Theta },{\rm \Theta }'\right)_{|f_{|\left\{x\right\}} \times 1_{I} } :f_{|\left\{x\right\}} \times 1_{I} \to t 
\end{equation} 
is constant.
\end{theorem}

This theorem is proved in \cite{7}.

Let $Z$ be a topological space and $\gamma $ be an open covering of it. For each $W\in \gamma $ let $J_{W} $ be an open covering of the unit interval $I$. In this case $\tilde{\gamma }=\left\{W\times J{\rm |}W\in \gamma ,\; J\in J_{W} \right\}$  is an open covering of $Z\times I$ which is called a stacked covering \cite{18}.

\begin{lemma}

 Let $Z$ be a normal space and let $\tilde{\gamma }=\left\{W\times J{\rm |}W\in \gamma ,\; J\in J_{W} \right\}$ be a stacked covering of  $Z\times I$, where $\gamma $ is locally finite and each  $J_{W} $, $W\in \gamma $ is finite. If for each $W\in \gamma $, consider a fixed real number $a_{W} >0$, then there exists a continuous function $\varphi :Z\to I$ such that every $z\in Z$ admits a $W\in \gamma $ such that 
\begin{equation} \label{GrindEQ__2_19_} 
z\in W,\; \; 0<\varphi \left(z\right)\le a_{W} .                               
\end{equation} 
\end{lemma}

This lemma is proved in \cite{18}.

Now according to \cite{18} we will formulate and prove the following:

\begin{lemma}
 Let $\mathbf{\left(p,p'\right)}=\left\{\left(p_{\lambda } ,p'_{\lambda } \right)\; \right\}:f\to \mathbf{f}$ be a fiber resolution and let $\lambda \in {\rm \Lambda }$,  $t:P\to P'$ be a continuous map,  $\left(\varphi _{\lambda } ,\varphi' _{\lambda } \right),\left(\psi _{\lambda } ,\psi' _{\lambda } \right):f_{\lambda } \to t$ be morphisms and $\left({\rm \Theta },{\rm \Theta }'\right):f\times 1_{I} \to t$ be a fiber homotopy between $\left(\varphi _{\lambda } ,\varphi' _{\lambda } \right)\circ \left(p_{\lambda } ,p'_{\lambda } \right)$ and $\left(\psi _{\lambda } ,\psi' _{\lambda } \right)\circ \left(p_{\lambda } ,p'_{\lambda } \right)$. Then for every pair $\left(\alpha ,\alpha '\right)$ of coverings $\alpha \in Cov\left(P\right)$ and $\alpha '\in Cov\left(P'\right)$, there exist $\lambda '\ge \lambda $ and a fiber homotopy  $\left({\rm \Delta },{\rm \Delta }'\right):f_{\lambda '} \times 1_{I} \to t$ such that
\begin{equation} \label{GrindEQ__2_20_} 
\left({\rm \Delta },{\rm \Delta }'\right):\left(\varphi _{\lambda } ,\varphi' _{\lambda } \right)\circ \left(p_{\lambda \lambda '} ,p'_{\lambda \lambda '} \right)\cong \left(\psi _{\lambda } ,\psi' _{\lambda } \right)\circ \left(p_{\lambda \lambda '} ,p'_{\lambda \lambda '} \right),      
\end{equation} 
\begin{equation} \label{GrindEQ__2_21_} 
\left(\left({\rm \Theta },{\rm \Theta }'\right),\left({\rm \Delta },{\rm \Delta }'\right)\circ \left(p_{\lambda } \times 1_{I} ,p'_{\lambda } \times 1_{I} \right)\right)\le \left(\alpha ,\alpha '\right).\;  
\end{equation} 
\end{lemma}

\begin{proof}
For the resolution $\mathbf{\left(p,p'\right)}=\left\{\left(p_{\lambda } ,p'_{\lambda } \right)\; \right\}:f\to \mathbf{f}$\textbf{ }and for the morphisms $\left(\varphi _{\lambda } ,\varphi' _{\lambda } \right),\left(\psi _{\lambda } ,\psi' _{\lambda } \right):f_{\lambda } \to t$ consider the corresponding homotopy $\left({\rm \Theta },{\rm \Theta }'\right):f\times 1_{I} \to t$. i.e.
\begin{equation} \label{GrindEQ__2_22_} 
{\rm \Theta }\left({\rm x},0\right)=\varphi _{\lambda } p_{\lambda } \left(x\right),\; \; \; \ \forall x\in X,                          
\end{equation} 
\begin{equation} \label{GrindEQ__2_23_} 
{\rm \Theta }\left({\rm x},1\right)=\psi _{\lambda } p_{\lambda } \left(x\right),\; \; \; \; \forall x\in X,                          
\end{equation} 
\begin{equation} \label{GrindEQ__2_24_} 
{\rm \Theta }'\left({\rm x}',0\right)=\varphi' _{\lambda } p'_{\lambda } \left(x'\right),\; \; \;  \forall x'\in X',                      
\end{equation} 
\begin{equation} \label{GrindEQ__2_25_} 
{\rm \Theta }\left({\rm x},1\right)=\psi' _{\lambda } p'_{\lambda } \left(x'\right),\; \; \; \; \forall x'\in X'. 
\end{equation} 
Let $\left(\alpha ,\alpha '\right)$ be a pair of coverings $\alpha $ and $\alpha '$ of $P$ and $P'$, respectively.  Let $\left(\alpha _{1} ,\alpha ' _{1} \right)$ be a star-refinement of $\left(\alpha ,\alpha '\right)$. Let use the theorem 1.5 for the $\mathbf{ANR}$-map $t:P\to P'$  and for the pair $\left(\alpha _{1} ,\alpha' _{1} \right)$, and choose a pair $\left(\beta ,\beta '\right)$ of coverings $\beta $ and $\beta '$ of $P$ and $P'$, respectively, such that $\left(\beta ,\beta '\right)$  is a star-refinement of $\left(\alpha _{1} ,\alpha ' _{1} \right).$ Now use the property FR2) for the pair $\left(\beta ,\beta '\right)$ and choose a pair $\left(\beta _{1} ,\beta ' _{1} \right)$  of coverings $\beta _{1} $ and $\beta' _{1}$ of $P$ and $P'$, respectively, such that $\left(\beta _{1} ,\beta' _{1} \right)$  is a star-refinement of $\left(\beta ,\beta '\right).$ Therefore, we have
\begin{equation} \label{GrindEQ__2_26_} 
\left(\alpha ,\alpha '\right)<^{*} \left(\alpha _{1} ,\alpha' _{1} \right)<^{*} \left(\beta ,\beta '\right)<^{*} \left(\beta _{1} ,\beta' _{1} \right).              
\end{equation} 

Let $t'$ be the map $t\times t:P\times P\to P'\times P'.$ Denote by $\left(\zeta ,\zeta '\right):f\to t'$ a morphism defined by
\begin{equation} \label{GrindEQ__2_27_} 
\zeta \left(x\right)=\left(\varphi _{\lambda } p_{\lambda } \left(x\right),\psi _{\lambda } p_{\lambda } \left(x\right)\right),\; \; \; \forall x\in X, 
\end{equation} 
\begin{equation} \label{GrindEQ__2_28_} 
\zeta '\left(x'\right)=\left(\varphi ' _{\lambda } p'_{\lambda } \left(x'\right),\psi' _{\lambda } p'_{\lambda } \left(x'\right)\right),\; \; \;  \forall x'\in X'. 
\end{equation} 
Let $\left(\eta ,\eta '\right),\left(\xi ,\xi '\right):t'\to t$ be the morphisms defined by
\begin{equation} \label{GrindEQ__2_29_} 
\eta \left(y,y_{1} \right)=y,\;  \xi \left(y,y_{1} \right)=y_{1} ,\; \; \; \;  \forall \; \left(y,y_{1} \right)\in P\times P, 
\end{equation} 
\begin{equation} \label{GrindEQ__2_30_} 
\eta '\left(y',y'_{1} \right)=y',\; \xi '\left(y',y'_{1} \right)=y'_{1} ,\; \; \; \;  \forall \; \left(y',y'_{1} \right)\in P'\times P'. 
\end{equation} 
In this case we have
\begin{equation} \label{GrindEQ__2_31_} 
\left(\eta ,\eta '\right)\circ \left(\zeta ,\zeta '\right)=\left(\varphi _{\lambda } ,\varphi' _{\lambda } \right)\circ \left(p_{\lambda } ,p'_{\lambda } \right),\;  
\end{equation} 
\begin{equation} \label{GrindEQ__2_32_} 
\left(\xi ,\xi '\right)\circ \left(\zeta ,\zeta '\right)=\left(\psi _{\lambda } ,\psi' _{\lambda } \right)\circ \left(p_{\lambda } ,p'_{\lambda } \right).\;  
\end{equation} 
Indeed,
\begin{equation} \label{GrindEQ__2_33_} 
\left(\eta \circ \zeta \right)\left(x\right)=\eta \left(\varphi _{\lambda } p_{\lambda } \left(x\right),\psi _{\lambda } p_{\lambda } \left(x\right)\right)=\varphi _{\lambda } p_{\lambda } \left(x\right),\;  
\end{equation} 
\begin{equation} \label{GrindEQ__2_34_} 
\left(\eta '\circ \zeta '\right)\left(x'\right)=\eta '\left(\varphi' _{\lambda } p'_{\lambda } \left(x'\right),\psi' _{\lambda } p'_{\lambda } \left(x'\right)\right)=\varphi ' _{\lambda } p'_{\lambda } \left(x'\right),\;  
\end{equation} 
\begin{equation} \label{GrindEQ__2_35_} 
\left(\xi \circ \zeta \right)\left(x\right)=\xi \left(\varphi _{\lambda } p_{\lambda } \left(x\right),\psi _{\lambda } p_{\lambda } \left(x\right)\right)=\psi _{\lambda } p_{\lambda } \left(x\right),\;  
\end{equation} 
\begin{equation} \label{GrindEQ__2_36_} 
\left(\xi '\circ \zeta '\right)\left(x'\right)=\xi '\left(\varphi ' _{\lambda } p'_{\lambda } \left(x'\right),\psi ' _{\lambda } p'_{\lambda } \left(x'\right)\right)=\psi' _{\lambda } p'_{\lambda } \left(x'\right).\;  
\end{equation} 
On the other hand, by \eqref{GrindEQ__2_22_}, \eqref{GrindEQ__2_23_}, \eqref{GrindEQ__2_24_} and \eqref{GrindEQ__2_25_} we have
\begin{equation} \label{GrindEQ__2_37_} 
{\rm \Theta }\left({\rm x},0\right)=\left(\eta \circ \zeta \right)\left(x\right),\; \; \; \;  \forall x\in X,                          
\end{equation} 
\begin{equation} \label{GrindEQ__2_38_} 
{\rm \Theta }\left({\rm x},1\right)=\left(\xi \circ \zeta \right)\left(x\right),\; \; \; \;  \forall x\in X,                          
\end{equation} 
\begin{equation} \label{GrindEQ__2_39_} 
{\rm \Theta }'\left({\rm x}',0\right)=\left(\eta '\circ \zeta '\right)\left(x'\right),\; \; \; \; \; \; \forall x'\in X',                      
\end{equation} 
\begin{equation} \label{GrindEQ__2_40_} 
{\rm \Theta }\left({\rm x},1\right)=\left(\xi ' \circ \zeta' \right)\left(x'\right),\; \; \; \; \; \; \forall x'\in X'. 
\end{equation} 
Therefore, $\left({\rm \Theta },{\rm \Theta }'\right):f\times 1_{I} \to t$ is a homotopy which connects the morphisms $\left(\eta ,\eta '\right)\circ \left(\zeta ,\zeta '\right)$ and $\left(\xi ,\xi '\right)\circ \left(\zeta ,\zeta '\right).$ The map $t'$ is an $\mathbf{ANR}$-map and so by lemma 1.3 there exist an  $\mathbf{ANR}$-map $t'':P_{2} \to P_{2}^{'} $, a morphism $\left(\sigma ,\sigma '\right):f\to t''$,  $\left(\kappa ,\kappa '\right):t''\to t'$ and a fiber homotopy $\left({\rm \Delta },{\rm \Delta }'\right):t''\times 1_{I} \to t$ such that 
\begin{equation} \label{GrindEQ__2_41_} 
\left(\zeta ,\zeta '\right)=\left(\kappa ,\kappa '\right)\circ \left(\sigma ,\sigma '\right),                                   
\end{equation} 
\begin{equation} \label{GrindEQ__2_42_} 
\left({\rm \Delta },{\rm \Delta }'\right):\left(\xi ,\xi '\right)\circ \left(\kappa ,\kappa '\right)\cong \left(\eta ,\eta '\right)\circ \left(\kappa ,\kappa '\right),                   
\end{equation} 
\begin{equation} \label{GrindEQ__2_43_} 
\left({\rm \Theta },{\rm \Theta }'\right)=\left({\rm \Delta },{\rm \Delta }'\right)\circ \left(\sigma \times 1_{I} ,\sigma ' _{1} \times 1_{I} \right).                        
\end{equation} 
Consider the pair $\left({\rm \Delta }^{-1} \left(\beta _{1} \right),{\rm \Delta }'^{-1} (\beta _{1}^{'} )\right)$ of coverings ${\rm \Delta }^{-1} \left(\beta _{1} \right)$ and ${\rm \Delta }'^{-1} (\beta _{1}^{'} )$ of spaces $P_{2} \times I$ and $P_{2}^{'} \times I$, respectivelly. For the coverings ${\rm \Delta }^{-1} \left(\beta _{1} \right)$ and ${\rm \Delta }'^{-1} (\beta _{1}^{'} )$ choose refinements $\tilde{\gamma }$ and $\tilde{\gamma }^{'} $, which are stacked coverings of $P_{2} \times I$ and $P_{2}^{'} \times I$, such that $\tilde{\gamma }<t''^{-1} \left(\tilde{\gamma }^{'} \right).$ 

By FR1) for the $\left(\sigma ,\sigma '\right):f\to t''$ and pair $\left(\gamma ,\gamma '\right)$, there exist a $\lambda ''\ge \lambda $ and a mapping $\left(\varepsilon ,\varepsilon '\right):f_{\lambda ''} \to t''$ such that
\begin{equation} \label{GrindEQ__2_44_} 
\left(\left(\sigma ,\sigma '\right),\left(\varepsilon ,\varepsilon '\right)\circ \left(p_{\lambda ''} ,p'_{\lambda ''} \right)\right)\le \left(\gamma ,\gamma '\right).                                
\end{equation} 
Note that for any $W\in \gamma $ and $W'\in \gamma '$, there exist $J\in J_{W} $ and $J'\in J_{W'} $ , such that $W\times \left\{0\right\}\subset W\times J$,  $W'\times \left\{0\right\}\subset W'\times J'$. On the other hand, for $W\times J$  and $W'\times J'$, there exist $V\in \beta _{1} $ and $V'\in \beta' _{1} $, such that $W\times J\subset {\rm \Delta }^{-1} \left(V\right)$ and $W'\times J'\subset {\rm \Delta }'^{-1} \left(V'\right)$. Therefore, by \eqref{GrindEQ__2_42_} we have 
\begin{equation} \label{GrindEQ__2_45_} 
\left(\eta \circ \kappa \right)\left(W\right)={\rm \Delta }\left(W\times \left\{0\right\}\right)\subset {\rm \Delta }\left(W\times J\right)\subseteq V, 
\end{equation} 
\begin{equation} \label{GrindEQ__2_46_} 
\left(\eta '\circ \kappa '\right)\left(W'\right)={\rm \Delta }'\left(W'\times \left\{0\right\}\right)\subset {\rm \Delta }'\left(W'\times J'\right)\subseteq V'. 
\end{equation} 
Consequently, by \eqref{GrindEQ__2_44_} we obtain
\begin{equation} \label{GrindEQ__2_47_} 
\left(\left(\eta ,\eta '\right)\circ \left(\kappa ,\kappa '\right)\circ \left(\sigma ,\sigma '\right),\left(\eta ,\eta '\right)\circ \left(\kappa ,\kappa '\right)\circ \left(\varepsilon ,\varepsilon '\right)\circ \left(p_{\lambda ''} ,p'_{\lambda ''} \right)\right)\le \left(\beta _{1} ,\beta' _{1} \right).  
\end{equation} 
By \eqref{GrindEQ__2_41_}  and \eqref{GrindEQ__2_31_} we have
\[\left(\eta ,\eta '\right)\circ \left(\kappa ,\kappa '\right)\circ \left(\sigma ,\sigma '\right)=\left(\eta ,\eta '\right)\circ \left(\zeta ,\zeta '\right)=\] 
\begin{equation} \label{GrindEQ__2_48_} 
\left(\varphi _{\lambda } ,\varphi' _{\lambda } \right)\circ \left(p_{\lambda } ,p'_{\lambda } \right)=\left(\varphi _{\lambda } ,\varphi' _{\lambda } \right)\circ \left(p_{\lambda \lambda ''} ,p'_{\lambda \lambda ''} \right)\circ \left(p_{\lambda ''} ,p'_{\lambda ''} \right).          
\end{equation} 
Therefore, \eqref{GrindEQ__2_47_} becomes 
\begin{equation} \label{GrindEQ__2_49_} 
\left(\left(\varphi _{\lambda } ,\varphi' _{\lambda } \right)\circ \left(p_{\lambda \lambda ''} ,p'_{\lambda \lambda ''} \right)\circ \left(p_{\lambda ''} ,p'_{\lambda ''} \right),\left(\eta ,\eta '\right)\circ \left(\kappa ,\kappa '\right)\circ \left(\varepsilon ,\varepsilon '\right)\circ \left(p_{\lambda ''} ,p'_{\lambda ''} \right)\right)\le \left(\beta _{1} ,\beta' _{1} \right).   
\end{equation} 
Analogously, we obtain
\begin{equation} \label{GrindEQ__2_50_} 
\left(\left(\psi _{\lambda } ,\psi' _{\lambda } \right)\circ \left(p_{\lambda \lambda ''} ,p'_{\lambda \lambda ''} \right)\circ \left(p_{\lambda ''} ,p'_{\lambda ''} \right),\left(\xi ,\xi '\right)\circ \left(\kappa ,\kappa '\right)\circ \left(\varepsilon ,\varepsilon '\right)\circ \left(p_{\lambda ''} ,p'_{\lambda ''} \right)\right)\le \left(\beta _{1} ,\beta' _{1} \right).   
\end{equation} 
By choose of $\left(\beta _{1} ,\beta _{1}^{'} \right)$ and the property FR2), there exists a $\lambda '\ge \lambda ''$ such that 
\begin{equation} \label{GrindEQ__2_51_} 
\left(\left(\varphi _{\lambda } ,\varphi' _{\lambda } \right)\circ \left(p_{\lambda \lambda '} ,p'_{\lambda \lambda '} \right),\left(\eta ,\eta '\right)\circ \left(\kappa ,\kappa '\right)\circ \left(\varepsilon ,\varepsilon '\right)\circ \left(p_{\lambda \lambda '} ,p'_{\lambda \lambda '} \right)\right)\le \left(\beta ,\beta '\right),      
\end{equation} 
\begin{equation} \label{GrindEQ__2_52_} 
\left(\left(\psi _{\lambda } ,\psi' _{\lambda } \right)\circ \left(p_{\lambda \lambda '} ,p'_{\lambda \lambda '} \right),\left(\xi ,\xi '\right)\circ \left(\kappa ,\kappa '\right)\circ \left(\varepsilon ,\varepsilon '\right)\circ \left(p_{\lambda \lambda '} ,p'_{\lambda \lambda '} \right)\right)\le \left(\beta ,\beta '\right).       
\end{equation} 
On the other hand, by choose of  $\left(\beta ,\beta '\right)$ and theorem 1.5, there exist $\left(\alpha _{1} ,\alpha _{1}^{'} \right)-$homotopies 
\begin{equation} \label{GrindEQ__2_53_} 
\left({\rm {\rm K} },{\rm {\rm K} }'\right):\left(\varphi _{\lambda } ,\varphi' _{\lambda } \right)\circ \left(p_{\lambda \lambda '} ,p'_{\lambda \lambda '} \right)\cong \left(\eta ,\eta '\right)\circ \left(\kappa ,\kappa '\right)\circ \left(\varepsilon ,\varepsilon '\right)\circ \left(p_{\lambda \lambda '} ,p'_{\lambda \lambda '} \right),     
\end{equation} 
\begin{equation} \label{GrindEQ__2_54_} 
\left({\rm \Lambda },{\rm \Lambda }'\right):\left(\psi _{\lambda } ,\psi' _{\lambda } \right)\circ \left(p_{\lambda \lambda '} ,p'_{\lambda \lambda '} \right)\cong \left(\xi ,\xi '\right)\circ \left(\kappa ,\kappa '\right)\circ \left(\varepsilon ,\varepsilon '\right)\circ \left(p_{\lambda \lambda '} ,p'_{\lambda \lambda '} \right). 
\end{equation} 
By \eqref{GrindEQ__2_44_}, for any $t\in I$ the points $\left(\sigma \left(x\right),t\right)$ and  $\left(\varepsilon p_{\lambda ''} \left(x\right),t\right)$ belong to some member of   $\tilde{\gamma }$ and therefore, they belong to ${\rm \Delta }^{-1} \left(V\right)$, for some $V\in \beta _{1} $. In the same way, $\left(\sigma '\left(x'\right),t\right)$ and  $\left(\varepsilon 'p'_{\lambda ''} \left(x'\right),t\right)$ belong to ${\rm \Delta }^{'\; -1} \left(V'\right)$, for some $V'\in \beta _{1}^{'} $. Consequently, $\left({\rm \Delta },{\rm \Delta }'\right)\circ \left(\sigma \times 1_{I} ,\sigma _{1}^{'} \times 1_{I} \right)$ and $\left({\rm \Delta },{\rm \Delta }'\right)\circ \left(\varepsilon \circ \left(p_{\lambda ''} \times 1_{I} \right),\varepsilon '\circ \left(p'_{\lambda ''} \times 1_{I} \right)\right)$ are $\left(\beta _{1} ,\beta _{1}^{'} \right)-$near and therefore, $\left(\beta ,\beta '\right)-$near. i.e. 
\begin{equation} \label{GrindEQ__2_55_} 
\left(\left({\rm \Delta },{\rm \Delta }'\right)\circ \left(\sigma \times 1_{I} ,\sigma' _{1} \times 1_{I} \right),\left({\rm \Delta },{\rm \Delta }'\right)\circ \left(\varepsilon \circ \left(p_{\lambda ''} \times 1_{I} \right),\varepsilon '\circ \left(p'_{\lambda ''} \times 1_{I} \right)\right)\; \right)\le \left(\beta ,\beta '\right).    
\end{equation} 
Our aim is to define a homotopy $\left({\rm {\rm H} },{\rm {\rm H} }'\right):f_{\lambda '} \times 1_{I} \to t$  using the homotopies $\left({\rm {\rm K} },{\rm {\rm K} }'\right),\; \left({\rm \Lambda },{\rm \Lambda }'\right):f_{\lambda '} \times 1_{I} \to t$  and $\left({\rm \Delta },{\rm \Delta }'\right):t''\times 1_{I} \to t.$ 

For each  $W'\in \gamma '$, choose a number $0<\alpha _{W'} <\frac{1}{3} $ , which is smaller than the Lebegue number of the covering $J_{W'} $. In this case, if $\left|t-t'\right|\le \alpha _{W'} ,\; \; t,t'\in I$, then there exists $J'\in J_{W'} $, such that $t,t'\in J'$. Consequently, if $z'\in W'$ and $\left|t-t'\right|\le \alpha _{W'} $, then $\left(z',t\right),\left(z',t'\right)\in W'\times J'\in \tilde{\gamma }^{'} $ and therefore, ${\rm \Delta }'\left(z',t\right)$ and ${\rm \Delta }'\left(z',t'\right)$ belong to some $U'\in \beta _{1}^{'} .$ Now apply lemma 2.6 and we obtain a continuous function $\varphi :P_{1}^{''} \to I$ such that for every $z\in P_{2}^{'} $ there is $W'\in \gamma '$ and 
\begin{equation} \label{GrindEQ__2_57_} 
z\in W',\; \; 0<\varphi \left(z\right)\le \alpha _{W'} \le \frac{1}{3} .                              
\end{equation} 
Let  ${\rm {\rm H} }:X_{\lambda '} \times I\to P$ and ${\rm {\rm H} }':X_{\lambda '}^{'} \times I\to P'$ are defined by
 
\begin{equation} \label{GrindEQ__2_58_} 
H(x,t)=\begin{cases} 
{\rm K} \left(  x, \frac{t}{\varphi t''\epsilon p_{\lambda '' \lambda '}(x)} \right), & 0\le t \le \varphi t'' \epsilon p_{\lambda '' \lambda '}(x) \\ 
\Delta \left( \epsilon p_{\lambda '' \lambda '}(x), \frac{1-\varphi t''\epsilon p_{\lambda '' \lambda '}(x)}{1-2\varphi t''\epsilon p_{\lambda '' \lambda '}(x)} \right), & \varphi t''\epsilon p_{\lambda '' \lambda '}(x) \le t \le 1- \varphi t'' \epsilon p_{\lambda '' \lambda '}(x) \\ 
\Lambda \left(  x, \frac{1-t}{\varphi t''\epsilon p_{\lambda '' \lambda '}(x)} \right), & 1- \varphi t''\epsilon p_{\lambda '' \lambda '}(x) \le t \le 1
\end{cases}                         
\end{equation}

\begin{equation} \label{GrindEQ__2_59_} 
H'(x',t)=\begin{cases} 
{\rm K'} \left(  x', \frac{t}{\varphi t''\epsilon ' p'_{\lambda '' \lambda '}(x')} \right), & 0\le t \le \varphi t'' \epsilon ' p'_{\lambda '' \lambda '}(x') \\ 
\Delta ' \left( \epsilon ' p'_{\lambda '' \lambda '}(x'), \frac{1-\varphi t''\epsilon ' p'_{\lambda '' \lambda '}(x')}{1-2\varphi t''\epsilon ' p'_{\lambda '' \lambda '}(x')} \right), & \varphi t''\epsilon' p'_{\lambda '' \lambda '}(x) \le t \le 1- \varphi t'' \epsilon' p'_{\lambda '' \lambda '}(x') \\ 
\Lambda ' \left(  x', \frac{1-t}{\varphi t''\epsilon ' p'_{\lambda '' \lambda '}(x)} \right), & 1- \varphi t''\epsilon ' p'_{\lambda '' \lambda '}(x') \le t \le 1
\end{cases}                             
\end{equation}   

Let show that the following diagram is commutative:
\begin{equation} \label{GrindEQ__2_60_}
\begin{tikzpicture}
\node[left] at (-1,0 ){$X_{\lambda}\times I$};
\node[right] at (1,0 ){$X'_{\lambda'}\times I$};
\node[left] at (-1.5,-2 ){$P$};
\node[right] at (1.5,-2 ){$P'.$};
\draw [thick, ->] (-1,0) to node[above] {$f_\lambda \times 1_I$} (1,0);
\draw [thick, ->] (-1,-2) to node[above] {$t$} (1,-2);
\draw [thick, ->] (-1.7,-0.5) to node[right] {$H$} (-1.7,-1.5);
\draw [thick, ->] (1.7,-0.5) to node[right] {$H'$} (1.7,-1.5);
\end{tikzpicture}
\end{equation} 
Consider three different cases:

\begin{enumerate}
\item  Let  $x\in X_{\lambda '} $  and $\; 0\le t\le \varphi t''\varepsilon p_{\lambda ''\lambda '} \left(x\right)$, then
\[\left({\rm t}\circ {\rm {\rm H} }\right)\left(x,t\right)=\left({\rm t}\circ {\rm {\rm K} }\right)\left(x,\frac{{\rm t}}{\varphi t''\varepsilon p_{\lambda ''\lambda '} \left(x\right)} \right)=\left({\rm {\rm K} }'\circ \left(f_{\lambda '} \times 1_{I} \right)\right)\left(x,\frac{{\rm t}}{\varphi t''\varepsilon p_{\lambda ''\lambda '} \left(x\right)} \right)=\] 
\begin{equation} \label{GrindEQ__2_61_} 
{\rm {\rm K} }'\left(f_{\lambda '} \left(x\right),\frac{{\rm t}}{\varphi t''\varepsilon p_{\lambda ''\lambda '} \left(x\right)} \right)={\rm {\rm H} }\left(f_{\lambda '} \left(x\right),{\rm t}\right)=\left({\rm {\rm H} }'\circ \left(f_{\lambda '} \times 1_{I} \right)\right)\left(x,t\right). 
\end{equation} 

\item  Let  $x\in X_{\lambda '} $  and $\; \varphi t''\varepsilon p_{\lambda ''\lambda '} \left(x\right)\le t\le 1-\varphi t''\varepsilon p_{\lambda ''\lambda '} \left(x\right)$, then
\[\left({\rm t}\circ {\rm {\rm H} }\right)\left(x,t\right)=\left({\rm t}\circ {\rm \Delta }\right)\left(\varepsilon p_{\lambda ''\lambda '} \left(x\right),\frac{1-\varphi t''\varepsilon p_{\lambda ''\lambda '} \left(x\right)}{1-2\varphi t''\varepsilon p_{\lambda ''\lambda '} \left(x\right)} \right)=\] 
\[\left({\rm \Delta }'\circ \left(t''\times 1_{I} \right)\right)\left(\varepsilon p_{\lambda ''\lambda '} \left(x\right),\frac{1-\varphi t''\varepsilon p_{\lambda ''\lambda '} \left(x\right)}{1-2\varphi t''\varepsilon p_{\lambda ''\lambda '} \left(x\right)} \right)=\] 
\[{\rm \Delta }'\left(t''\varepsilon p_{\lambda ''\lambda '} \left(x\right),\frac{1-\varphi t''\varepsilon p_{\lambda ''\lambda '} \left(x\right)}{1-2\varphi t''\varepsilon p_{\lambda ''\lambda '} \left(x\right)} \right)={\rm \Delta }'\left(\varepsilon 'p'_{\lambda ''\lambda '} f_{\lambda '} \left(x\right),\frac{1-\varphi t''\varepsilon p_{\lambda ''\lambda '} \left(x\right)}{1-2\varphi t''\varepsilon p_{\lambda ''\lambda '} \left(x\right)} \right)=\] 
\begin{equation} \label{GrindEQ__2_62_} 
={\rm H'} \left(f_{\lambda '} (x),t \right) =\left({\rm H'} \circ f_{\lambda ' \times 1_I} \right) (x,t).
\end{equation} 

\item  Let  $x\in X_{\lambda '} $  and $\; 1-\varphi \varepsilon 'p'_{\lambda ''\lambda '} \left(x'\right)\le t\le 1$, then
\[(t \circ H)(x,t)=\left(t \circ \Lambda \right) \left( x, \frac{1-t}{ \varphi t''\varepsilon p_{\lambda ''\lambda '}} \right) = \left(\Lambda' \circ \left(f_{\lambda'} \times 1_I\right) \right) \left( x, \frac{1-t}{ \varphi t''\varepsilon p_{\lambda ''\lambda '}} \right) =\] 
\begin{equation} \label{GrindEQ__2_63_} 
{\rm \Lambda }'\left(f_{\lambda '} \left(x\right),\frac{1-{\rm t}}{\varphi \varepsilon p_{\lambda ''\lambda '} \left(x\right)} \right)={\rm {\rm H} }'\left(f_{\lambda '} \left(x\right),{\rm t}\right)=\left({\rm {\rm H} }'\circ \left(f_{\lambda '} \times 1_{I} \right)\right)\left(x,t\right). 
\end{equation} 
\end{enumerate}
Therefore, the pair $\left({\rm {\rm H} },{\rm {\rm H} }'\right):f_{\lambda '} \times 1_{I} \to t$ is a well-defined homotopy between $\left(\varphi_{\lambda} , \varphi'_{\lambda}\right) \circ \left(\varphi_{\lambda \lambda ''} , \varphi'_{\lambda \lambda ''}\right)$ and $\left(\psi_{\lambda} , \psi'_{\lambda}\right) \circ \left(\psi_{\lambda \lambda ''} , \psi'_{\lambda \lambda ''}\right).$

By \cite{18} we have that 
\begin{equation} \label{GrindEQ__2_64_} 
\left({\rm \Theta }',{\rm {\rm H} }'\circ \left(p'_{\lambda '} \times 1_{I} \right)\right)\le \alpha '. 
\end{equation} 
Therefore, it remains to show that
\begin{equation} \label{GrindEQ__2_65_} 
\left({\rm \Theta }',{\rm {\rm H} }\circ \left(p_{\lambda '} \times 1_{I} \right)\right)\le \alpha . 
\end{equation} 
So we have to show that for every $\left(x,t\right)\in X\times I$, there exists a $U\in \alpha $ such that 
\begin{equation} \label{GrindEQ__2_66_} 
{\rm \Theta }'(x,t),{\rm {\rm H} }\left(p_{\lambda '}(x),t \right)) \in U.  
\end{equation} 
Let $\varphi t''\varepsilon p_{\lambda ''} \left(x\right)\le t\le 1-\varphi t''\varepsilon p_{\lambda ''} \left(x\right)$, then by \eqref{GrindEQ__2_58_} and \eqref{GrindEQ__2_43_} we have
\begin{equation} \label{GrindEQ__2_67_} 
{\rm \Theta }\left(x,t\right)={\rm \Delta }\left(\sigma \left(x\right),t\right). 
\end{equation} 
\[ {\rm H }(p_{\lambda'}(x),t)={\rm \Delta }\left(\varepsilon p_{\lambda '' \lambda '} \left(p_{\lambda '}(x)\right),\frac{1-\varphi t''\varepsilon p_{\lambda '' \lambda '} \left(p_{\lambda '}(x)\right)}{1-2\varphi t''\varepsilon p_{\lambda '' \lambda '} \left(p_{\lambda '}(x)\right)} \right)=\]
\begin{equation} \label{GrindEQ__2_68_} 
{\rm \Delta }\left(\varepsilon p_{\lambda ''} \left(x\right),\frac{1-\varphi t''\varepsilon p_{\lambda ''} \left(x\right)}{1-2\varphi t''\varepsilon p_{\lambda '' } \left(x\right)} \right). 
\end{equation} 
By \eqref{GrindEQ__2_57_}, for the $z'=t'' \varepsilon p_{\lambda ''} \left(x\right),$ there is a $W'\in \gamma '$ such that
\begin{equation} \label{GrindEQ__2_69_} 
z'\in W',\; \; 0<\varphi \left( t''\varepsilon p_{\lambda ''} \left(x\right) \right)\le \alpha _{W'}  \le \frac{1}{3}.                             
\end{equation} 
Note that $\varphi t''\varepsilon p_{\lambda ''} \left(x\right)\le t\le 1-\varphi t''\varepsilon p_{\lambda ''} \left(x\right)$ implies that  $\left|1-2t\right|\le 1-2\varphi t''\varepsilon p_{\lambda '} \left(y\right)$ and so
\begin{equation} \label{GrindEQ__2_70_} 
\left|t- \frac{1-\varphi t''\varepsilon p_{\lambda ''} \left(x\right)}{1-2\varphi t''\varepsilon p_{\lambda '' } \left(x\right)}\right|\le \varphi t''\varepsilon p_{\lambda ''}(x) \le \alpha_{W'}.                            
\end{equation} 
Therefore, there exists a $J'\in J_{W'} $ such that
\begin{equation} \label{GrindEQ__2_71_} 
\left(t''\varepsilon p_{\lambda ''} \left(x\right),t\right)\left(t''\varepsilon p_{\lambda ''} \left(x\right),\frac{1-\varphi t''\varepsilon p_{\lambda ''} \left(x\right)}{1-2\varphi t''\varepsilon p_{\lambda '' } \left(x\right)}\right)\in W'\times J'. 
\end{equation} 
By the choose of $\tilde{\gamma }^{'} $, for each $W' \times J' \in \tilde{\gamma }^{'} $ there exists a $W\times J\in \tilde{\gamma }$ such that $t''^{-1} \left(W'\times J'\right)\subset W\times J$ and so 
\begin{equation} \label{GrindEQ__2_72_} 
\left(\varepsilon p_{\lambda ''} \left(x\right),t\right),\left(\varepsilon p_{\lambda ''} \left(x\right),\frac{1-\varphi \varepsilon p_{\lambda ''} \left(x\right)}{1-2\varphi \varepsilon p_{\lambda '} \left(x\right)} \right)\in W \times J. 
\end{equation} 
On the other hand, for $W\times J\in \tilde{\gamma }$ there is a $V_{1} \in \beta _{1} $, such that $W \times J \in \Delta ^{-1}(V_1), $ and by \eqref{GrindEQ__2_72_} we have
\begin{equation} \label{GrindEQ__2_73_} 
{\rm \Delta }\left(\varepsilon p_{\lambda ''} \left(x\right),t\right),{\rm \Delta }\left(\varepsilon p_{\lambda ''} \left(x\right),\frac{1-\varphi \varepsilon p_{\lambda ''} \left(x\right)}{1-2\varphi \varepsilon p_{\lambda '} \left(x\right)} \right)\in V_{1} . 
\end{equation} 
Note that by \eqref{GrindEQ__2_44_} there exists a $W_1 \times J \in \gamma$ such that
\begin{equation} \label{GrindEQ__2_74_} 
\sigma \left(x\right),\varepsilon p_{\lambda ''} \left(x\right)\in W_{1} . 
\end{equation} 
Therefore, $\left(\sigma \left(x\right),t\right)$ and $\left( \varepsilon p_{\lambda''}(x),t\right)$ belong to some $W_{1} \times J\in \gamma $ and so there is a $V_{2} \in \beta _{1} $, such that 
\begin{equation} \label{GrindEQ__2_75_} 
{\rm \Delta }\left(\sigma \left(x\right),t\right),{\rm \Delta }\left(\varepsilon p_{\lambda ''} \left(x\right),t \right) \in V_2. 
\end{equation}
Note that $\beta _{1} $ is a star-refinement of $\beta$. On the other hand, $\beta $ is a star-refinement of $\alpha $ and so by \eqref{GrindEQ__2_72_} there is a $U \in \alpha$ such that
\begin{equation} \label{GrindEQ__2_76_} 
{\rm \Delta }\left(\sigma \left(x\right),t\right),{\rm \Delta }\left(\varepsilon p_{\lambda ''} \left(x\right),\frac{1-\varphi \varepsilon p_{\lambda ''} \left(x\right)}{1-2\varphi \varepsilon p_{\lambda '} \left(x\right)} \right)\in U. 
\end{equation} 
On the other hand, by \eqref{GrindEQ__2_67_} and \eqref{GrindEQ__2_68_} we have
\begin{equation} \label{GrindEQ__2_77_} 
{\rm \Theta }\left(x,t\right),{\rm {\rm H} }\left(p_{\lambda ''} \left(x\right),t\right)\in U. 
\end{equation} 
So, if $ \varphi t ''p_{\lambda''}(x) \le 1- \varphi t ''p_{\lambda''}(x)$ then \eqref{GrindEQ__2_65_} is fulfilled. 

Now consider the case when $0 \le t \le 1-\varphi t''p_{\lambda''}(x)$ and as before, for the $z'=\varepsilon p_{\lambda ''} \left(x\right)\in P_{1}^{''} $ choose $W' \in \gamma '$  such that
\begin{equation} \label{GrindEQ__2_78_} 
z'\in W',\; \; 0<\varphi t''\varepsilon p_{\lambda ''} \left(x\right)\le \alpha _{W'} \le \frac{1}{3}.                              
\end{equation} 
In this case $\left| t-0\right| \le \varphi t''p_{\lambda''}(x)$ and so there exists $J' \in J_W' $ such that
\begin{equation} \label{GrindEQ__2_79_} 
\left(t''\varepsilon p_{\lambda ''} \left(x\right),0\right)\left(t''\varepsilon p_{\lambda ''} \left(x\right),{\rm t}\right)\in W'\times J'. 
\end{equation} 
Since $t''^{-1}(\gamma) \ge \tilde{\gamma}$, there is a $W\times J\in \tilde{\gamma }$ such that
\begin{equation} \label{GrindEQ__2_80_} 
\left(\varepsilon p_{\lambda ''} \left(x\right),0\right)\left(\varepsilon p_{\lambda ''} \left(x\right),{\rm t}\right)\in W\times J.  
\end{equation} 
Therefore, there exists a  $V_{1} \in \beta _{1} $, such that 
\begin{equation} \label{GrindEQ__2_81_} 
{\rm \Delta }\left(\varepsilon p_{\lambda ''} \left(x\right),0\right),{\rm \Delta }\left(\varepsilon p_{\lambda ''} \left(x\right),{\rm t}\right)\in V_{1} . 
\end{equation} 
On the other hand, by \eqref{GrindEQ__2_74_} for each  $t \in I$ there is a $W_{1} \times J\in \tilde{\gamma }$ such that
\begin{equation} \label{GrindEQ__2_82_} 
{\rm \Delta }\left(\sigma \left(x\right),t \right),{\rm \Delta }\left(\varepsilon p_{\lambda ''} \left(x\right), t \right)\in V_{2} .
\end{equation} 
Let $t=0$, then \eqref{GrindEQ__2_82_} becomes 
\begin{equation} \label{GrindEQ__2_83_} 
{\rm \Delta }\left(\sigma \left(x\right),0 \right),{\rm \Delta }\left(\varepsilon p_{\lambda ''} \left(x\right), 0 \right)\in V'_{2} .  
\end{equation} 
In this case, in the same way as before, because  $\beta _{1} $ is a star-refinement of $\beta$ and  $\beta $ is a star-refinement of $\alpha,$  there is a $U_1 \in \alpha_1$ such that
\begin{equation} \label{GrindEQ__2_84_} 
{\rm \Delta }\left(x,0\right),{\rm \Delta }\left(x,{\rm t}\right)\in U_{1} . 
\end{equation} 
On the other hand, by \eqref{GrindEQ__2_67_} and \eqref{GrindEQ__2_68_} we have that ${\rm K}:X_{\lambda '} \times I \to P$ is an $\alpha -$homotopy and so for $y=p_{\lambda '}(x)$ there is a $U_{2} \in \alpha $ such that
\begin{equation} \label{GrindEQ__2_85_} 
{\rm K }\left(p_{\lambda '} \left(x\right),0\right),{\rm K }\left( p_{\lambda '} \left(x\right),\frac{t}{\varphi t'' \varepsilon p_{\lambda ''} \left(x\right)} \right)\in U_2. 
\end{equation} 
By \eqref{GrindEQ__2_53_}, \eqref{GrindEQ__2_34_}, \eqref{GrindEQ__2_41_}, \eqref{GrindEQ__2_43_} we have
 \[{\rm K }\left(p_{\lambda '} \left(x\right),0\right)=\left(\varphi_\lambda \circ p_{\lambda \lambda'}\right) \left(p_\lambda(x)\right)= \left(\varphi_\lambda \circ p_{\lambda}\right) \left((x)\right)=\left( \eta \circ \zeta\right)(x)=\]
\begin{equation} \label{GrindEQ__2_86_} 
\left(\eta \circ \kappa \circ \sigma\right)(x)=\left(\eta  \circ \kappa\right)(x)=\Delta \left(\sigma(x),0\right)=\Theta(x,0).
\end{equation}
Furthermore, by \eqref{GrindEQ__2_57_} 
\begin{equation} \label{GrindEQ__2_87_} 
{\rm K }\left( p_{\lambda '} \left(x\right),\frac{t}{\varphi t'' \varepsilon p_{\lambda ''} \left(x\right)} \right)={\rm H}\left(p_{\lambda'}(x),t\right).
\end{equation} 
Therefore, we have 
\begin{equation} \label{GrindEQ__2_88_} 
{\rm \Theta }\left(x,0\right),{\rm {\rm H} }\left(p_{\lambda '} \left(x\right),t\right)\in U_{2} . 
\end{equation} 
Since $\alpha_1$ is a star-refinement of $\alpha $, by \eqref{GrindEQ__2_84_} and \eqref{GrindEQ__2_88_}, there is a $U \in \alpha$ such that
\begin{equation} \label{GrindEQ__2_89_} 
{\rm \Theta }\left(x,0\right),{\rm {\rm H} }\left(p_{\lambda '} \left(x\right),t\right)\in U. 
\end{equation} 
Therefore, \eqref{GrindEQ__2_65_} is fulfilled.

The proof of the last case when $1-\varphi t'' \varepsilon p_{\lambda'' \lambda'}(x) \le t \le 1$ is analogous to the second case.  
\end{proof}

\begin{theorem}
 Every fiber resolution $\mathbf { \left(p,p'\right)}=\left\{\left(p_{\lambda } ,p'_{\lambda } \right)\; \right\}:f\to \mathbf f$\textbf{ }of a map $f:X\to X'$ is a strong fiber expansion of $f.$
\end{theorem}

\begin{proof}
 First, let prove the FS1) property. Let $t:P\to P'$ be an $\mathbf{ANR}$-map and let $\left(\varphi ,\varphi '\right):f \to t$ be a morphism. Consider a pair $(\alpha, \alpha')$ of $\alpha \in Cov\left(P\right)$ and $\alpha '\in Cov\left(P'\right)$, such that every $(\alpha, \alpha')$-near morphism into $t$ is homotopic. By the condition FR1) there exist a $\lambda \in \Lambda $ and a morphism $\left(\varphi _{\lambda } ,\varphi' _{\lambda } \right):f_{\lambda } \to t$ such that $\left(\varphi_\lambda  ,\varphi'_\lambda \right) \circ (p_\lambda , p'_\lambda )$ and $\left(\varphi ,\varphi '\right)$ are $\left(\alpha ,\alpha '\right)$-near and consequently, homotopic, i.e.
\begin{equation} \label{GrindEQ__2_90_} 
\left(\varphi ,\varphi '\right) \cong \left(\varphi_\lambda  ,\varphi'_\lambda \right) \circ (p_\lambda , p'_\lambda ).
\end{equation} 

Now we will prove the FS2). Let $t:P\to P'$ be an $\mathbf{ANR}$-map, $\lambda \in \Lambda $ and $\left(\varphi ,\varphi '\right), \left(\psi ,\psi '\right):f_\lambda \to t$ be morphisms such that $\left(\varphi _{\lambda } ,\varphi' _{\lambda } \right)\circ \left(p_{\lambda } ,p'_{\lambda } \right)$ and $\left(\psi _{\lambda } ,\psi' _{\lambda } \right)\circ \left(p_{\lambda } ,p'_{\lambda } \right)$ are connected by a fiber homotopy $\left({\rm \Theta },{\rm \Theta }'\right):f\times 1_{I} \to t$. Let $\left(\alpha ,\alpha '\right)$ be a pair of coverings $\alpha \in Cov\left(P\right)$ and $\alpha' \in Cov\left(P' \right)$. By theorem 1.2 there exist a pair $\left( \beta , \beta'\right)$ of coverings $\beta \in Cov\left(P\right)$ and $\beta '\in Cov\left(P'\right),$ such that each $\left(\beta ,\beta '\right)$-near morphism into $t:P\to P'$ is homotopic. If we use lemma 2.4 for the pair $\left(\beta ,\beta '\right)$, then we obtain that there exist a $\lambda '\ge \lambda $ and homotopy $\left({\rm {\rm H} },{\rm {\rm H} }'\right):f\times 1_{I} \to t$ such that
\begin{equation} \label{GrindEQ__2_91_} 
\left({\rm H },{\rm H ' }\right):\left(\varphi_\lambda  ,\varphi'_\lambda \right) \circ (p_\lambda , p'_\lambda ) \cong \left(\psi_\lambda  ,\psi'_\lambda \right) \circ (p_\lambda , p'_\lambda ) 
\end{equation} 
\begin{equation} \label{GrindEQ__2_92_} 
\left( \left(\Theta,\Theta \right),\left({\rm H}, {\rm H'}\right) \circ \left( p_\lambda \times 1_I, p'_\lambda \times 1_I \right) \right). 
\end{equation} 

Now, if we use theorem 1.2 and instead of map $f:X\to Y$ consider the map $t:P \to P' $ and instead of morphisms $\left(\varphi ^{1} ,\psi ^{1} \right),\left(\varphi ^{2} ,\psi ^{2} \right):g\to f$ consider the morphisms $\left({\rm \Theta },{\rm \Theta }'\right),\; \left({\rm {\rm H} },{\rm {\rm H} }'\right)\circ \left(p_{\lambda } \times 1_{I} ,p'_{\lambda } \times 1_{I} \right):f\times 1_{I} \to t,$  then we obtain that there exists a homotopy which connects given morphisms. Moreover, restriction of morphisms $\left({\rm \Theta },{\rm \Theta }'\right)$ and $\left({\rm {\rm H} },{\rm {\rm H} }'\right)\circ \left(p_{\lambda } \times 1_{I} ,p'_{\lambda } \times 1_{I} \right)$ on the submap  $f\times 1_{\partial I} :X\times \partial I\to X'\times \partial I$  coincides and so the obtained homotopy is a homotopy which is fixed on the submap  $f\times 1_{\partial I} $ .
\end{proof}

\section{Main theorem on strong fiber expansion}
\

Let $\left(\varphi _{\lambda } ,\varphi' _{\lambda } \right),\left(\psi _{\lambda } ,\psi' _{\lambda } \right):f_{\lambda } \to t$ be morphisms as in the SF2). Consider the new  morphism $\left(\sigma _{\lambda } ,\sigma' _{\lambda } \right):f_{\lambda } \times 1_{\partial I} \to t$ defined by the following formula:

\begin{equation} \label{GrindEQ__3_1_} 
\left( {{\sigma }_{\lambda }}_{|{{X}_{\lambda }}\times \left\{ 0 \right\}},\sigma' {{_{\lambda }}_{|{{X}_{\lambda }}'\times \left\{ 0 \right\}}} \right)=\left( {{\sigma }_{\lambda }},\sigma' _{\lambda } \right)\circ \left( {{i}_{{{X}_{\lambda }}\times \left\{ 0 \right\}}},{{i}_{{{X}_{\lambda }}'\times \left\{ 0 \right\}}} \right)=\left( {{\varphi }_{\lambda }},\varphi' _{\lambda } \right),          
\end{equation} 
\begin{equation} \label{GrindEQ__3_2_} 
\left( {{\sigma }_{\lambda }}_{|{{X}_{\lambda }}\times \left\{ 1 \right\}},\sigma' {{_{\lambda }}_{|{{X}_{\lambda }}'\times \left\{ 1 \right\}}} \right)=\left( {{\sigma }_{\lambda }},\sigma' _{\lambda } \right)\circ \left( {{i}_{{{X}_{\lambda }}\times \left\{ 1 \right\}}},{{i}_{{{X}_{\lambda }}'\times \left\{ 1 \right\}}} \right)=\left( {{\psi }_{\lambda }},\psi' _{\lambda } \right),          
\end{equation} 
where $\left(i_{X_{\lambda } \times \left\{0\right\}} ,i_{X_{\lambda } '\times \left\{0\right\}} \right):f_{\lambda } \times 1_{\left\{0\right\}} \to f_{\lambda } \times 1_{\partial I} $ is an inclusion. In this case, the condition SF2) can be formulated in the following way: 

If $\left(\sigma _{\lambda } ,\sigma' _{\lambda } \right):f_{\lambda } \times 1_{\partial I} \to t$ is a morphism and $\left({\rm \Theta },{\rm \Theta }'\right):f\times 1_{I} \to t$ is a fiber homotopy such that 
\[ \left( \Theta_{|X \times \partial I}, \Theta'_{|X' \times \partial I} \right)=\left( \Theta, \Theta' \right) \circ \left( i_{|X \times \partial I}, i'_{|X' \times \partial I}   \right)=\] 
\begin{equation} \label{GrindEQ__3_3_} 
\left(\sigma _{\lambda } ,\sigma' _{\lambda } \right) \circ \left(p_{\lambda } \times 1_{\partial I} ,p'_{\lambda } \times 1_{\partial I} \right),                                
\end{equation} 
then there exist $\lambda '\ge \lambda $ and fiber homotopies $\left({\rm \Delta },{\rm \Delta }'\right):f_{\lambda } \times 1_{I} \to t$ and $\left({\rm \Gamma },{\rm \Gamma }'\right):f_{\lambda } \times 1_{I} \times 1_{I} \to t$ such that
\[ \left( \Delta_{|X_\lambda' \times \partial I}, \Delta'_{|X'_{\lambda'} \times \partial I} \right)=\left( \Delta, \Delta' \right) \circ \left( i_{|X_\lambda \times \partial I}, i'_{|X'_{\lambda'} \times \partial I}   \right)=\] 
\begin{equation} \label{GrindEQ__3_4_} 
\left(\sigma _{\lambda } ,\sigma' _{\lambda } \right) \circ \left(p_{\lambda \lambda' } \times 1_{\partial I} ,p'_{\lambda \lambda' } \times 1_{\partial I} \right),                               
\end{equation} 
and the homotopy $\left({\rm \Gamma },{\rm \Gamma }'\right)$ connects $\left({\rm \Theta },{\rm \Theta }'\right)$ and $\left({\rm \Delta },{\rm \Delta }'\right)\circ \left(p_{\lambda } \times 1_{I} ,p'_{\lambda } \times 1_{I} \right)$ and is fixed on the submap  $f\times 1_{\partial I} :X\times 1_{\partial I} \to X'\times 1_{\partial I} $.

Let $t:P\to P'$ be a continuous map and $Y$ be a topological space. Consider the map
\begin{equation} \label{GrindEQ__3_5_} 
t^{\# } :C\left(I,P\right)\to C\left(I,P'\right) 
\end{equation} 
which is defined by the formula
\begin{equation} \label{GrindEQ__3_6_} 
t''(h)(y)=(t \circ h)(y). \; \; \; \forall h \in C(I,P), \; y \in Y.
\end{equation} 

Let $f:X\to X'$ be any continuous map and $\left(\varphi ,\varphi '\right):f\to t^{\# } $  be any morphism, then we can define the {morphism $\left(\bar{\varphi },\bar{\varphi }^{'} \right):f\times 1_{Y} \to t$ by the following
\begin{equation} \label{GrindEQ__3_7_} 
\left(\bar{\varphi }\left(x,y\right),\bar{\varphi }' \left(x',y\right)\right)=\left(\varphi \left(x\right)\left(y\right),\varphi '\left(x'\right)\left(y\right)\right),\; \; \; \forall \; x\in X,\; x'\in X',{\rm \; }y\in Y. 
\end{equation} 

Analogously, if  $\left(\psi ,\psi '\right):f\times 1_{Y} \to t$ is a morphism, then we can define a morphism $\left(\tilde{\psi },\tilde{\psi }^{'} \; \right):f\to t^{\# } $ by
\begin{equation} \label{GrindEQ__3_8_} 
\left(\tilde{\psi }\left(x\right)\left(y\right),\tilde{\psi }' \left(x'\right)\left(y\right)\right)=\left(\psi \left(x,y\right),\psi '\left(x',y\right)\right),\; \; \; \forall \; x\in X,\; x'\in X',{\rm \; }y\in Y. 
\end{equation} 

\begin{theorem}
 If $\mathbf{\left  (p,p'\right)}=\left\{\left(p_{\lambda } ,p'_{\lambda } \right)\; \right\}:f\to \mathbf f$ is a strong fiber expansion and  $Y$ is a compact Hausdorff space, then $\left(\mathbf{p}\times 1_{Y} ,\mathbf{p'}\times 1_{Y} \right)=\left\{\left(p_{\lambda } \times 1_{Y} ,p_{\lambda }^{'} \times 1_{Y} \right)\; \right\}:f\times 1_{Y} \to \mathbf{f}\times 1_{Y} $ is also a strong fiber\textbf{ }expansion.
\end{theorem}

\begin{proof}
First, let prove that $\left(\mathbf{p}\times 1_{Y} ,\mathbf{p'}\times 1_{Y} \right)$ has the property SF1). Consider any $\mathbf{ANR}$-map $t:P\to P'$ and a morphism $\left(\varphi ,\varphi '\right):f\to t$. Let $\left(\tilde{\varphi },\; \tilde{\varphi }' \; \right):f\to t^{\# } $ is a corresponding morphism. Since $t^{\# } $ is an $\mathbf{ANR}$-map and $\mathbf{\left(p,p'\right)}=\left\{\left(p_{\lambda } ,p'_{\lambda } \right)\; \right\}:f\to \mathbf f$ is a strong fiber expansion, and so has the property SF1), then there exist a $\lambda \in {\rm \Lambda }$ and a morphism $\left(\varphi _{\lambda } ,\varphi' _{\lambda } \right):f_{\lambda } \to t^{\# } $ and a fiber homotopy $\left({\rm \Theta },{\rm \Theta }'\right):f\times 1_{I} \to t^{\# } $ such that
\begin{equation} \label{GrindEQ__3_9_} 
{\rm \Theta }\left(x,0\right)=\tilde{\varphi }\left(x\right),\; {\rm \Theta }\left(x,1\right)=\left(\varphi _{\lambda } \circ \; p_{\lambda } \right)\left(x\right),\; \; \;  \forall x\in X\;  
\end{equation} 
\begin{equation} \label{GrindEQ__3_10_} 
{\rm \Theta }'\left(x',0\right)=\tilde{\varphi }' \left(x'\right),\; {\rm \Theta }'\left(x',1\right)=\left(\varphi' _{\lambda } \circ \; p'_{\lambda } \right)\left(x'\right),\; \; \; \forall x'\in X'.\;  
\end{equation} 
Let $\left(\bar{\varphi }_{\lambda } ,\bar{\varphi }_{\lambda }' \; \right):f_{\lambda } \times 1_{Y} \to t$ and $\left(\bar{{\rm \Theta }},\bar{{\rm \Theta }}^{'}\right):f\times 1_{Y} \times 1_{I} \to t^{\# } $ (note that instead of $f\times 1_{Y} \times 1_{I} $, we should have $f\times 1_{I} \times 1_{Y} $, but we use the above mentioned notation) be the morphisms induced by  $\left(\varphi _{\lambda } ,\varphi' _{\lambda } \right)$ and $\left({\rm \Theta },{\rm \Theta }'\right)$, respectively. In this case, we have
\begin{equation} \label{GrindEQ__3_11_} 
\bar{{\rm \Theta }}\left(x,y,0\right)=\bar{\varphi }\left(x,y\right),\; \; \bar{{\rm \Theta }}\left(x,y,1\right)=\left(\bar{\varphi }_{\lambda } \circ \left(p_{\lambda } \times 1_{Y} \right)\right)\left(x,y\right),\;  
\end{equation} 
\begin{equation} \label{GrindEQ__3_12_} 
\bar{{\rm \Theta }}^{'} \left(x',y,0\right)=\bar{\varphi }^{'} \left(x',y\right),\; \bar{{\rm \Theta }}^{'} \left(x',y,1\right)=\left(\bar{\varphi }'_{\lambda } \circ \left(p'_{\lambda } \times 1_{Y} \right)\right)\left(x',y\right).\;  
\end{equation} 
Indeed,
\begin{equation} \label{GrindEQ__3_13_} 
\bar{{\rm \Theta }}\left(x,y,0\right)={\rm \Theta }\left(x,0\right)\left(y\right)=\tilde{\varphi }\left(x\right)\left(y\right)=\bar{\varphi }\left(x,y\right), 
\end{equation} 

\[\bar{{\rm \Theta }}\left(x,y,1\right)={\rm \Theta }\left(x,1\right)\left(y\right)=\left(\varphi _{\lambda } \circ \; p_{\lambda } \right)\left(x\right)\left(y\right)=\] 
\begin{equation} \label{GrindEQ__3_14_} 
\varphi _{\lambda } \left(p_{\lambda } \left(x\right)\right)\left(y\right)=\bar{\varphi }\left(p_{\lambda } \left(x\right),y\right)=\left(\bar{\varphi }_{\lambda } \circ \left(p_{\lambda } \times 1_{Y} \right)\right)\left(x,y\right), 
\end{equation} 
\begin{equation} \label{GrindEQ__3_15_} 
\bar{{\rm \Theta }}^{'} \left(x',y,0\right)={\rm \Theta }'\left(x',0\right)\left(y\right)=\tilde{\varphi }' \left(x'\right)\left(y\right)=\bar{\varphi }' \left(x',y\right), 
\end{equation} 

\[\bar{{\rm \Theta }}^{'} \left(x',y,1\right)={\rm \Theta }'\left(x',1\right)\left(y\right)=\left(\varphi' _{\lambda } \circ \; p'_{\lambda } \right)\left(x'\right)\left(y\right)=\] 
\begin{equation} \label{GrindEQ__3_16_} 
\varphi' _{\lambda } \left(p'_{\lambda } \left(x'\right)\right)\left(y\right)=\bar{\varphi }' \left(p'_{\lambda } \left(x'\right),y\right)=\left(\bar{\varphi }'_{\lambda } \circ \left(p'_{\lambda } \times 1_{Y} \right)\right)\left(x',y\right). 
\end{equation} 
Therefore, $\left(\bar{{\rm \Theta }},\bar{{\rm \Theta }}^{\rm '}\right):f\times 1_{Y} \times 1_{I} \to t^{\# } $ is a fiber homotopy which connects the morphisms $\left(\bar{\varphi }_{\lambda } ,\bar{\varphi }'_{\lambda } \; \right):f_{\lambda } \times 1_{Y} \to t$ and $\left(\bar{\varphi }_{\lambda } ,\bar{\varphi }_{\lambda }' \right)\circ (\left(p_{\lambda } \times 1_{Y} ,p'_{\lambda } \times 1_{Y} \right):f_{\lambda } \times 1_{Y} \to t$ and the so condition SF1) is fulfilled. 

Now, we have to prove the SF2). Let $\left(\sigma _{\lambda } ,\sigma' _{\lambda } \right):\left(f_{\lambda } \times 1_{Y} \right)\times 1_{\partial I} \to t$ be a morphism and $\left({\rm \Theta },{\rm \Theta }^{'}\right):\left(f\times 1_{Y} \right)\times 1_{I} \to t$  be a fiber homotopy such that
\begin{equation} \label{GrindEQ__3_17_} 
\left({\rm \Theta }_{|X\times Y\times \partial I} ,{\rm \Theta }_{|X'\times Y\times \partial I}^{'} \right)=\left(\sigma _{\lambda } ,\sigma' _{\lambda } \right)\circ \left(p_{\lambda } \times 1_{Y} \times 1_{\partial I} ,p'_{\lambda } \times 1_{Y} \times 1_{\partial I} \right).    
\end{equation} 
Therefore, we have
\[{\rm \Theta }\left(x,y,s\right)=\left(\sigma _{\lambda } \circ p_{\lambda } \times 1_{Y} \times 1_{\partial I} \right)\left(x,y,s\right)=\] 
\begin{equation} \label{GrindEQ__3_18_} 
\sigma _{\lambda } \left(p_{\lambda } \left(x\right),y,s\right),\; \; \; \forall x\in X,\; y\in Y,\; s\in \partial I,                    
\end{equation} 

\[{\rm \Theta }^{'}\left(x',y,s\right)=\left(\sigma' _{\lambda } \circ p'_{\lambda } \times 1_{Y} \times 1_{\partial I} \right)\left(x',y,s\right)=\] 
\begin{equation} \label{GrindEQ__3_19_} 
\sigma' _{\lambda } \left(p'_{\lambda } \left(x'\right),y,s\right),\; \; \; \forall x'\in X',\; y\in Y,\; s\in \partial I. 
\end{equation} 
Consider the morphisms $\left(\tilde{\sigma }_{\lambda } ,\tilde{\sigma }'_{\lambda } \right):f_{\lambda } \times 1_{\partial I} \to t^{\# } $  and $\left(\tilde{{\rm \Theta }},\tilde{{\rm \Theta }}^{{\rm '}} \right):f\times 1_{I} \to t^{\# } $  induced by $\left(\sigma _{\lambda } ,\sigma' _{\lambda } \right)$ and $\left({\rm \Theta },{\rm \Theta }'\right)$. In this case, we have 
\begin{equation} \label{GrindEQ__3_20_} 
\left(\tilde{{\rm \Theta }}_{|X\times \partial I} ,\tilde{{\rm \Theta }}_{|X'\times \partial I}^{'} \right)=\left(\tilde{\sigma }_{\lambda } ,\tilde{\sigma }_{\lambda }' \right)\circ \left(p_{\lambda } \times 1_{\partial I} ,p_{\lambda }' \times 1_{\partial I} \right). 
\end{equation} 
Indeed,
\[\tilde{{\rm \Theta }}\left(x,s\right)\left(y\right)={\rm \Theta }\left(x,y,s\right)=\sigma _{\lambda } \left(p_{\lambda } \left(x\right),y,s\right)=\tilde{\sigma }_{\lambda } \left(p_{\lambda } \left(x\right),s\right)\left(y\right)=\] 
\begin{equation} \label{GrindEQ__3_21_} 
=\left(\tilde{\sigma }_{\lambda } \circ \left(p_{\lambda } \times 1_{\partial I} \right)\right)\left(x,s\right)\left(y\right),\; \; \; \forall x\in X,\; y\in Y,\; s\in \partial I,            
\end{equation} 

\[\tilde{{\rm \Theta }}^{'} \left(x',s\right)\left(y\right)={\rm \Theta }'\left(x',y,s\right)=\sigma _{\lambda }' \left(p_{\lambda }' \left(x'\right),y,s\right)=\tilde{\sigma }_{\lambda }' \left(p_{\lambda }' \left(x'\right),s\right)\left(y\right)=\]
 
\begin{equation} \label{GrindEQ__3_22_} 
=\left(\tilde{\sigma }_{\lambda }' \circ \left(p_{\lambda }' \times 1_{\partial I} \right)\right)\left(x',s\right)\left(y\right),\; \; \; \forall x'\in X',\; y\in Y,\; s\in \partial I. 
\end{equation} 
Therefore, by the SF2) for $\mathbf{\left(p,p'\right)}$, there exist $\lambda '\ge \lambda $ and a fiber homotopy $\left({\rm \Delta },{\rm \Delta }'\right):f_{\lambda '} \times 1_{I} \to t^{\# } $ such that
\begin{equation} \label{GrindEQ__3_23_} 
\left({\rm \Delta }_{|X\times \partial I} ,{\rm \Delta }_{|X'\times \partial I}^{{\rm '}} \right)=\left(\tilde{\sigma }_{\lambda '} ,\tilde{\sigma }_{\lambda '}' \right)\circ \left(p_{\lambda \lambda '} \times 1_{\partial I} ,p_{\lambda \lambda '}' \times 1_{\partial I} \right). 
\end{equation} 
Moreover, there is a fiber homotopy $\left({\rm \Gamma },{\rm \Gamma }'\right):f_{\lambda } \times 1_{I} \times 1_{I} \to t^{\# } $ which connects homotopies $\left(\tilde{{\rm \Theta }},\tilde{{\rm \Theta }}^{{\rm '}} \right)$,  $\left({\rm \Delta },{\rm \Delta }'\right)\circ \left(p_{\lambda '} \times 1_{I} ,p_{\lambda '}' \times 1_{I} \right)$ and $\left({\rm \Gamma },{\rm \Gamma }'\right)$ is fixed on the submap $f\times 1_{\partial I} :X\times \partial I\to X'\times \partial I$ , i.e. 
\begin{equation} \label{GrindEQ__3_24_} 
{\rm \Gamma }\left(x,t,0\right)=\tilde{{\rm \Theta }}\left(x,t\right),\; \; \; {\rm \Gamma }'\left(x',t,0\right)=\tilde{{\rm \Theta }}^{'} \left(x',t\right), 
\end{equation} 
\begin{equation} \label{GrindEQ__3_25_} 
{\rm \Gamma }\left(x,t,1\right)={\rm \Delta }\left(p_{\lambda '} \left(x\right),t\right),\; \; \; {\rm \Gamma }'\left(x',t,0\right)={\rm \Delta }'\left(p_{\lambda ''\; }' \left(x'\right),t\right), 
\end{equation} 
\begin{equation} \label{GrindEQ__3_26_} 
{\rm \Gamma }\left(x,t,s\right)={\rm \Gamma }\left(x,t,1\right)=\tilde{{\rm \Theta }}\left(x,t\right),\; \; \forall \; x\in X,\; t\in \partial I\; ,\; s\in I, 
\end{equation} 
\begin{equation} \label{GrindEQ__3_27_} 
{\rm \Gamma }'\left(x',t,s\right)={\rm \Gamma }'\left(x',t,1\right)=\tilde{{\rm \Theta }}^{'} \left(x',t\right),\; \; \forall \; x'\in X',\; t\in \partial I\; ,\; s\in I. 
\end{equation} 

Let $\left(\bar{{\rm \Delta }},\bar{{\rm \Delta }}'\right):f_{\lambda '} \times 1_{Y} \times 1_{I} \to t$  and $\left(\bar{{\rm \Gamma }},\bar{{\rm \Gamma }}'\right):f_{\lambda } \times 1_{Y} \times 1_{I} \times 1_{I} \to t$ be the morphisms induced by $\left({\rm \Delta },{\rm \Delta }'\right)$ and $\left({\rm \Gamma },{\rm \Gamma }'\right)$, respectively. In this case, we have
\begin{equation} \label{GrindEQ__3_28_} 
\left(\bar{{\rm \Delta }}_{|X\times Y\times \partial I} ,\bar{{\rm \Delta }}_{|X'\times Y\times \partial I}^{'} \right)=\left(\sigma _{\lambda } ,\sigma _{\lambda }' \right)\circ \left(p_{\lambda \lambda '} \times 1_{Y} \times 1_{\partial I} ,p_{\lambda \lambda '}' \times 1_{Y} \times 1_{\partial I} \right),   
\end{equation} 
\begin{equation} \label{GrindEQ__3_29_} 
\left(\bar{{\rm \Gamma }},\bar{{\rm \Gamma }}'\right):\left({\rm \Theta },{\rm \Theta }'\right)\cong \left(\bar{{\rm \Delta }},\bar{{\rm \Delta }}'\right)\circ \left(p_{\lambda } \times 1_{Y} \times 1_{I} ,p_{\lambda }' \times 1_{Y} \times 1_{I} \right)\left(rel\; f\times 1_{\partial I} \right). 
\end{equation} 
Indeed,
\[\bar{{\rm \Delta }}\left(x,y,s\right)={\rm \Delta }\left(x,{\rm s}\right)\left(y\right)=\left(\tilde{\sigma }_{\lambda '} \circ p_{\lambda \lambda '} \times 1_{\partial I} \right)\left(x,s\right)\left(y\right)=\] 

\[\tilde{\sigma }_{\lambda '} \left(p_{\lambda \lambda '} \left(x\right),s\right)\left(y\right)=\sigma _{\lambda '} \left(p_{\lambda \lambda '} \left(x\right),y,s\right)=\] 

\begin{equation} \label{GrindEQ__3_30_} 
\left(\sigma _{\lambda '} \circ p_{\lambda \lambda '} \times 1_{Y} \times 1_{\partial I} \right)\left(x,y,s\right),\; \; \; \; \forall \; x\in X,y\in Y,\; s\in \partial I, 
\end{equation} 

\[\bar{{\rm \Delta }}^{'} \left(x',y,s\right)={\rm \Delta }'\left(x',{\rm s}\right)\left(y\right)=\left(\tilde{\sigma }_{\lambda '}' \circ p_{\lambda \lambda '}' \times 1_{\partial I} \right)\left(x',s\right)\left(y\right)=\]
 
\[\tilde{\sigma }_{\lambda '}' \left(p_{\lambda \lambda '}' \left(x'\right),s\right)\left(y\right)=\sigma _{\lambda '}' \left(p_{\lambda \lambda '}' \left(x'\right),y,s\right)=\] 
\begin{equation} \label{GrindEQ__3_31_} 
\left(\sigma _{\lambda '}' \circ p_{\lambda \lambda '}' \times 1_{Y} \times 1_{\partial I} \right)\left(x',y,s\right),\; \; \; \; \forall \; x'\in X',y\in Y,\; s\in \partial I. 
\end{equation} 
Therefore, \eqref{GrindEQ__3_28_} is fulfilled. On the other hand, 
\begin{equation} \label{GrindEQ__3_32_} 
\bar{{\rm \Gamma }}\left(x,y,t,0\right)={\rm \Gamma }\left(x,t,0\right)\left(y\right)=\tilde{{\rm \Theta }}\left(x,t\right)\left(y\right)={\rm \Theta }\left(x,y,t\right), 
\end{equation} 
\begin{equation} \label{GrindEQ__3_33_} 
\bar{{\rm \Gamma }}^{'} \left(x',y,t,0\right)={\rm \Gamma }'\left(x',t,0\right)\left(y\right)=\tilde{{\rm \Theta }}^{'} \left(x',t\right)\left(y\right)={\rm \Theta }'\left(x',y,t\right), 
\end{equation} 
\begin{equation} \label{GrindEQ__3_34_} 
\bar{{\rm \Gamma }}\left(x,y,t,1\right)={\rm \Gamma }\left(x,t,1\right)\left(y\right)={\rm \Delta }\left(p_{\lambda '} \left(x\right),t\right)\left(y\right)=\bar{{\rm \Delta }}\left(p_{\lambda '} \left(x\right),y,t\right), 
\end{equation} 
\begin{equation} \label{GrindEQ__3_35_} 
\bar{{\rm \Gamma }}^{'} \left(x',y,t,1\right)={\rm \Gamma }'\left(x',t,1\right)\left(y\right)={\rm \Delta }'\left(p_{\lambda '}^{'} \left(x'\right),t\right)\left(y\right)=\bar{{\rm \Delta }}^{'} \left(p_{\lambda '}^{'} \left(x'\right),y,t\right). 
\end{equation} 
Therefore, it remains to show that
\begin{equation} \label{GrindEQ__3_36_} 
\bar{{\rm \Gamma }}\left(x,y,t,s\right)=\bar{{\rm \Gamma }}\left(x,y,t,1\right)={\rm \Theta }\left(x,y,t\right),\; \; \forall \; x\in X,y\in Y,\; t\in \partial I\; ,\; s\in I, 
\end{equation} 
\begin{equation} \label{GrindEQ__3_37_} 
\bar{{\rm \Gamma }}^{'} \left(x',y,t,s\right)=\bar{{\rm \Gamma }}^{'} \left(x',y,t,1\right)={\rm \Theta }'\left(x',y,t\right),\; \; \forall \; x'\in X',y\in Y,\; t\in \partial I,\; s\in I, 
\end{equation} 
which are followed by \eqref{GrindEQ__3_26_} and \eqref{GrindEQ__3_27_}. 
\end{proof}

\begin{Definition}
Let $\left(\varphi ,\varphi '\right):f\to g$ be a morphism. We will say that $\left(\varphi ,\varphi '\right)$ is a fiber cofibration if for each fiber homotopy $\left({\rm \Theta },{\rm \Theta }'\right){\rm \; }:f\times 1_{I} \to g$ and a morphism $\left(\psi ,\psi '\right):g\to t$ for which  
\begin{equation} \label{GrindEQ__3_38_} 
\left({\rm \Theta }_{|X\times \left\{0\right\}} ,{\rm \Theta }_{|X'\times \left\{0\right\}}^{'} \right)=\left({\rm \Theta },{\rm \Theta }'\right)\circ \left(i_{X\times \left\{0\right\}} ,i_{X'\times \left\{0\right\}} \right)=\left(\psi ,\psi '\right)\circ \left(\varphi ,\varphi '\right),    
\end{equation} 
there exists a fiber homotopy $\left({\rm \Delta },{\rm \Delta }'\right):g\times 1_{I} \to t$ such that
\begin{equation} \label{GrindEQ__3_39_} 
\left({\rm \Delta }_{|Y\times \left\{0\right\}} ,{\rm \Delta }_{|Y'\times \left\{0\right\}}^{'} \right)=\left({\rm \Delta },{\rm \Delta }'\right)\circ \left(i_{Y\times \left\{0\right\}} ,i_{Y'\times \left\{0\right\}} \right)=\left(\psi ,\psi '\right),            
\end{equation} 
\begin{equation} \label{GrindEQ__3_40_} 
\left({\rm \Delta },{\rm \Delta }'\right)\circ \left(\varphi \times 1_{I} ,\varphi '\times 1_{I} \right)=\left({\rm \Theta },{\rm \Theta }'\right).                      
\end{equation} 
\end{Definition}

\begin{lemma} If for a morphism $\left(\varphi ,\varphi '\right):f\to g$ the maps $\varphi :X\to Y$, $\varphi ':X'\to Y',$ $g:Y\to Y'$ and  $i_{B} :B\to Y'$, where  $B=\varphi '\left(X'\right)\mathop{\cup }\nolimits^{} g\left(Y\right)$ are cofibrations, then the morphism  $\left(\varphi ,\varphi '\right)$ is a fiber cofibration.
\end{lemma}

\begin{proof} Consider a morphism $\left(\psi ,\psi '\right):f\to t$ and a fiber homotopy $\left({\rm \Theta },{\rm \Theta }'\right):f\times 1_{I} \to t$  for which \eqref{GrindEQ__3_37_} is fulfilled. By our assumption, the continuous map $\varphi :X\to Y$ is a cofibration and ${\rm \Theta }:X\times I\to P$ is a homotopy such that
\begin{equation} \label{GrindEQ__3_41_} 
{\rm \Theta }_{|X\times \left\{0\right\}} =\psi \circ \varphi. 
\end{equation} 
Therefore, there exists a homotopy ${\rm \Delta }:Y\times I\to P$ such that
\begin{equation} \label{GrindEQ__3_42_} 
{\rm \Delta }_{|Y\times \left\{0\right\}} =\psi , 
\end{equation} 
\begin{equation} \label{GrindEQ__3_43_} 
{\rm \Delta }\circ \left(\varphi \times 1_{I} \right)={\rm \Theta }. 
\end{equation} 

Let ${\rm \Delta }_{1} :Y\times I\to P'$ be the map given by
\begin{equation} \label{GrindEQ__3_44_} 
{\rm \Delta }_{1} =t\circ {\rm \Delta }. 
\end{equation} 
In this case, we have
\begin{equation} \label{GrindEQ__3_45_} 
{{\text{ }\!\!\Delta\!\!\text{ }}_{1}}_{|Y\times \left\{ 0 \right\}}=t\circ \psi ={\psi }'\circ g.
\end{equation} 
On the other hand, $g:Y\to Y'$ is a cofibration and so there exists a homotopy ${\rm \Delta }_{1}^{'} :Y'\times I\to P'$ such that 
\begin{equation} \label{GrindEQ__3_46_} 
\text{ }\!\!\Delta\!\!\text{ }{{_{1}^{'}}_{|{Y}'\times \left\{ 0 \right\}}}={\psi }', 
\end{equation} 
\begin{equation} \label{GrindEQ__3_47_} 
{\rm \Delta }_{1}^{'} \circ \left(g\times 1_{I} \right)={\rm \Delta }_{1} . 
\end{equation} 
In the same way, for the maps $\varphi ':X'\to Y',$  $\psi ':Y'\to P'$ and the homotopy ${\rm \Delta }':X'\times I\to P'$ we have a homotopy ${\rm \Delta }_{1}^{''} \; :Y'\times I\to P'$ such that 
\begin{equation} \label{GrindEQ__3_48_} 
\text{ }\!\!\Delta\!\!\text{ }{{_{1}^{''}}_{|{Y}'\times \left\{ 0 \right\}}}={\psi }', 
\end{equation} 
\begin{equation} \label{GrindEQ__3_49_} 
{\rm \Delta }_{1}^{''} \circ \left(\varphi '\times 1_{I} \right)={\rm \Theta }'. 
\end{equation} 
Let $A=\left(\; \varphi '\circ f\right)\left(X\right)=\left(\; g\circ \varphi \right)\left(X\right)\subseteq Y'$ and $A\times I\subseteq Y'\times I$. Our aim is to show that $\text{ }\!\!\Delta\!\!\text{ }{{_{1}^{'}}_{|A\times I}}=\text{ }\!\!\Delta\!\!\text{ }{{_{1}^{''}}_{|A\times I}}$. Indeed, for each $a\in A$ there exists an $x\in X$ such that $g\left(\varphi \left(x\right)\right)=\varphi '\left(f\left(x\right)\right)=a$ and so
 
\[{\rm \Delta }_{1}^{'} \left(a,t\right)={\rm \Delta }_{1}^{'} \left(g\left(\varphi \left(x\right)\right),t\right)=\left({\rm \Delta }_{1}^{'} \circ \left(g\times 1_{I} \right)\right)\left(\varphi \left(x\right),t\right)={\rm \Delta }_{1} \left(\varphi \left(x\right),t\right)=\] 

\[\left(t\circ {\rm \Delta }\right)\left(\varphi \left(x\right),t\right)=\left(t\circ {\rm \Delta }\circ \left(\varphi \times 1_{I} \right)\right)\left(x,t\right)=t\left(\left({\rm \Delta }\circ \left(\varphi \times 1_{I} \right)\right)\left(x,t\right)\right)=\] 

\[t\left({\rm \Theta }\left(x,t\right)\right)=\left(t\circ {\rm \Theta }\right)\left(x,t\right)=\left({\rm \Theta }'\circ \left(f\times 1_{I} \right)\right)\left(x,t\right)={\rm \Theta }'\left(f\left(x\right),t\right)=\] 

\[\left({\rm \Delta }_{1}^{''} \circ \left(\varphi '\times 1_{I} \right)\right)\left(f\left(x\right),t\right)=\left({\rm \Delta }_{1}^{''} \circ \left(\varphi '\times 1_{I} \right)\circ \left(f\times 1_{I} \right)\right)\left(x,t\right)=\] 
\begin{equation} \label{GrindEQ__3_50.1_} 
{\rm \Delta }_{1}^{''} \left(\varphi '\left(f\left(x\right)\right),t\right)={\rm \Delta }_{1}^{''} \left(a,t\right). 
\end{equation} 
Let $B=\; \varphi '\left(X'\right)\mathop{\cup }\nolimits^{} g\left(Y\right)\subseteq Y'$ and ${\rm \Delta }_{2} :B\times I\to P'$ is given by
\begin{eqnarray} \label{GrindEQ__3_50_} 
{\rm \Delta }_{2}=\begin{cases} 
{{\rm \Delta }_{1}^{'} \left(b,t\right),\; if\; b\in g\left(Y\right)} \\ 
{{\rm \Delta }_{1}^{''} \left(b,t\right),\; if\; b\in \varphi '\left(X'\right).}
\end{cases}       
\end{eqnarray}
By \eqref{GrindEQ__3_49_} it is clear that ${\rm \Delta }_{2} $ is well defined. On the other hand, $\psi ':Y'\to P'$ is a continuous map and ${\rm \Delta }_{2} :B\times I\to P'$ is a homotopy such that
\begin{equation} \label{GrindEQ__3_51_} 
{{\text{ }\!\!\Delta\!\!\text{ }}_{2}}_{|B\times \left\{ 0 \right\}}={\psi }'\circ {{i}_{B}}. 
\end{equation} 
By our assumption, $i_{B} :B\to Y'$ is a cofibration and so there exists a ${\rm \Delta }':Y'\times I\to P'$ such that 
\begin{equation} \label{GrindEQ__3_52_} 
{\rm \Delta }_{|Y'\times \left\{0\right\}}^{{\rm '}} =\psi ', 
\end{equation} 
\begin{equation} \label{GrindEQ__3_53_} 
{\rm \Delta }'\circ \left(i_{B} \times 1_{I} \right)={\rm \; \Delta }_{2} .                                          
\end{equation} 
By \eqref{GrindEQ__3_48_}, \eqref{GrindEQ__3_49_}, \eqref{GrindEQ__3_50_} and \eqref{GrindEQ__3_53_} we have
\begin{equation} \label{GrindEQ__3_54_} 
{\rm \Delta }'\circ \left(\varphi '\times 1_{I} \right)={\rm \Theta }'. 
\end{equation} 

Now we will show that a pair $\left({\rm \Delta },{\rm \Delta }'\right)$ is a morphism from $g\times 1_{I} $ to $t$. For this aim we must show that the following diagram is commutative 
\begin{equation} \label{GrindEQ__3_55_}
\begin{tikzpicture}
\node[left] at (-1,0 ){$Y \times I$};
\node[right] at (1,0 ){$Yy' \times I$};
\node[left] at (-1.5,-2 ){$P$};
\node[right] at (1.5,-2 ){$P'.$};
\draw [thick, ->] (-1,0) to node[above] {$g \times 1_I$} (1,0);
\draw [thick, ->] (-1,-2) to node[above] {$t$} (1,-2);
\draw [thick, ->] (-1.7,-0.5) to node[right] {$\Delta$} (-1.7,-1.5);
\draw [thick, ->] (1.7,-0.5) to node[right] {$\Delta'$} (1.7,-1.5);
\end{tikzpicture}
\end{equation} 
By \eqref{GrindEQ__3_50_}, \eqref{GrindEQ__3_53_}, \eqref{GrindEQ__3_46_} and \eqref{GrindEQ__3_43_} we have
\[\left({\rm \Delta }'\circ \left(g\times 1_{I} \right)\right)\left(y,t\right)={\rm \Delta }'\left(g\left(y\right),t\right)=\left({\rm \Delta }'\circ \left(i_{B} \times 1_{I} \right)\right)\left(g\left(y\right),t\right)=\] 
\[{\rm \Delta }_{2} \left(g\left(y\right),t\right)={\rm \Delta }_{1}^{'} \left(g\left(y\right),t\right)=\left({\rm \Delta }_{1}^{'} \circ \left(g\times 1_{I} \right)\right)\left(y,t\right)=\] 
\begin{equation} \label{GrindEQ__3_56_} 
{\rm \Delta }_{1} \left(y,t\right)=\left(t\circ {\rm \Delta }\right)\left(y,t\right). 
\end{equation} 
So we obtain the fiber homotopy $\left({\rm \Delta },{\rm \Delta }'\right):g\times 1_{I} \to t$ for which \eqref{GrindEQ__3_38_} and \eqref{GrindEQ__3_39_} are fulfilled. 
\end{proof}

\begin{lemma}
 For any continuous map $\; f:X\to X'$ and any positive integer  $n\in \mathbb{N}$, the morphism $\left(1_{X} \times i_{\partial I^{n} } ,1_{X'} \times i_{\partial I^{n} } \right):f\times 1_{\partial I^{n} } \to f\times 1_{I^{n} } $ is a fiber cofibration.
\end{lemma}

\begin{proof}
Let $\left(\psi ,\psi '\right):f\times 1_{I^{n} } \to t$ be a morphism and $\left({\rm \Theta },{\rm \Theta }'\right):f\times 1_{\partial I^{n} } \times 1_{I^{2} } \to t$ be a fiber homotopy such that
\begin{equation} \label{GrindEQ__3_57_} 
\left({\rm \Theta }_{|X\times \partial I^{n} \times \left\{0\right\}} ,{\rm \Theta }_{|X'\times \partial I^{n} \times \left\{0\right\}}^{'} \right)=\left(\psi ,\psi '\right)\circ \left(1_{X} \times i_{\partial I^{n} } ,1_{X'} \times i_{\partial I^{n} } \right).       
\end{equation} 
Let $A=I^{n} \times \left\{o\right\}\mathop{\cup }\nolimits^{} \left(\partial I^{n} \times I\right),$   $A=I^{n} \times I$ and $r:B\to A$ be a corresponding retraction. Consider mappings ${\rm \Delta }_{1} :X\times A\to P$ and ${\rm \Delta }_{1}^{'} :X'\times A\to P'$ which are given by
\begin{equation} \label{GrindEQ__3_58_} 
{\rm \Delta }_{1} \left(x,e,t\right)=\left\{\begin{array}{c} {\Theta \left(x,e,t\right),\; if\; e\in \partial I^{n} } \\ {\psi \left(x,e,t\right),\; if\; t=0,} \end{array}\right.  
\end{equation} 
\begin{equation} \label{GrindEQ__3_59_} 
{\rm \Delta }_{1}^{'} \left(x',e,t\right)=\left\{\begin{array}{c} {{\rm \Theta }'\left(x',e,t\right),\; if\; e\in \partial I^{n} } \\ {\psi '\left(x',e,t\right),\; if\; t=0.} \end{array}\right.  
\end{equation} 
By \eqref{GrindEQ__3_57_} the mappings ${\rm \Delta }_{1} $ and ${\rm \Delta }_{1}^{'} $ are well defined. Let show that the pair $\left({\rm \Delta }_{1} ,{\rm \Delta }_{1}^{'} \right)$ is a morphism from $f\times 1_{A} $ to $t$. Consider the following cases:

\begin{enumerate}
\item  If $e\in \partial I^{n} $, then
\[\left(t\circ {\rm \Delta }_{1} \right)\left(x,e,t\right)=\left(t\circ {\rm \Theta }\right)\left(x,e,t\right)={\rm \Theta }'\left(f\left(x\right),e,t\right)=\] 
\begin{equation} \label{GrindEQ__3_60_} 
{\rm \Delta }_{1}^{'} \left(f\left(x\right),e,t\right)=\left({\rm \Delta }_{1}^{'} \circ \left(f\times 1_{A} \right)\right)\left(x,e,t\right). 
\end{equation} 

\item  If $t=0$, then
\[\left(t\circ {\rm \Delta }_{1} \right)\left(x,e,0\right)=\left(t\circ \psi \right)\left(x,e,0\right)=\psi '\left(f\left(x\right),e,0\right)=\] 
\begin{equation} \label{GrindEQ__3_61_} 
{\rm \Delta }_{1}^{'} \left(f\left(x\right),e,0\right)=\left({\rm \Delta }_{1}^{'} \circ \left(f\times 1_{A} \right)\right)\left(x,e,0\right). 
\end{equation} 
\end{enumerate}
Let  ${\rm \Delta }:X\times I^{n} \times I\to P$ and ${\rm \Delta }':X'\times I^{n} \times I\to P'$ be the mappings defined by 
\begin{equation} \label{GrindEQ__3_62_} 
{\rm \Delta }\left(x,e,t\right)={\rm \Delta }_{1} \left(x,r\left(e,t\right)\right), 
\end{equation} 
\begin{equation} \label{GrindEQ__3_63_} 
{\rm \Delta }'\left(x',e,t\right)={\rm \Delta }_{1}^{'} \left(x',r\left(e,t\right)\right). 
\end{equation} 
It is clear that $\left({\rm \Delta },{\rm \Delta }'\right)$ is a morphism from $f\times 1_{I} \times 1_{I} $ to $t$. So it remains to show that 
\begin{equation} \label{GrindEQ__3_64_} 
\left({\rm \Delta }_{|X\times I^{n} \times \left\{0\right\}} ,{\rm \Delta }_{|X'\times I^{n} \times \left\{0\right\}}^{'} \right)=\left(\psi ,\psi '\right), 
\end{equation} 
\begin{equation} \label{GrindEQ__3_65_} 
\left({\rm \Theta },{\rm \Theta }'\right)=\left({\rm \Delta },{\rm \Delta }'\right)\circ \left(1_{X} \times i_{\partial I^{n} } \times 1_{I} ,1_{X'} \times i_{\partial I^{n} } \times 1_{I} \right).       
\end{equation} 
Indeed, by \eqref{GrindEQ__3_62_}, \eqref{GrindEQ__3_58_}, \eqref{GrindEQ__3_63_} and \eqref{GrindEQ__3_59_} we have
\begin{equation} \label{GrindEQ__3_66_} 
{\rm \Delta }\left(x,e,0\right)={\rm \Delta }_{1} \left(x,r\left(e,0\right)\right)={\rm \Delta }_{1} \left(x,e,0\right)=\psi \left(x,e\right), 
\end{equation} 
\begin{equation} \label{GrindEQ__3_67_} 
{\rm \Delta }'\left(x',e,0\right)={\rm \Delta }_{1}^{'} \left(x',r\left(e,0\right)\right)={\rm \Delta }_{1}^{'} \left(x',e,0\right)=\psi '\left(x',e\right). 
\end{equation} 
Therefore, \eqref{GrindEQ__3_64_} is fulfilled.

To prove \eqref{GrindEQ__3_65_}, consider any  $\left(x,e,t\right)\in X\times \partial I^{n} \times I$ and $\left(x',e,t\right)\in X'\times \partial I^{n} \times I$. By \eqref{GrindEQ__3_62_}, \eqref{GrindEQ__3_58_}, \eqref{GrindEQ__3_63_} and \eqref{GrindEQ__3_59_} we have
\[\left({\rm \Delta }\circ \left(1_{X} \times i_{\partial I^{n} } \times 1_{I} \right)\right)\left(x,e,t\right)={\rm \Delta }\left(x,e,t\right)={\rm \Delta }_{1} \left(x,r\left(e,t\right)\right)=\] 
\begin{equation} \label{GrindEQ__3_68_} 
\Delta_1(x,e,t)=\Theta(x,e,t).                              
\end{equation} 

\[\left({\rm \Delta }'\circ \left(1_{X'} \times i_{\partial I^{n} } \times 1_{I} \right)\right)\left(x',e,t\right)={\rm \Delta }'\left(x',e,t\right)={\rm \Delta }_{1}^{'} \left(x',r\left(e,t\right)\right)=\] 
\begin{equation} \label{GrindEQ__3_69_} 
{\rm \Delta }_{1}^{'} \left(x',e,t\right)={\rm \Theta }'\left(x',e,t\right). 
\end{equation} 
\end{proof}

\begin{theorem}
 For every strong fiber expansion $\mathbf{\left(p,p'\right)}:f\to \mathbf f$, if for $\lambda \in {\rm \Lambda }$, $\left(\sigma _{\lambda } ,\sigma _{\lambda }' \right):f_{\lambda } \times 1_{\partial I^{2} } \to t$ is a morphism and $\left({\rm \Theta },{\rm \Theta }'\right):f\times 1_{I^{2} } \to t$ is a fiber homotopy such that  
\[\left({\rm \Theta }_{|X\times \partial I^{2} } ,{\rm \Theta }_{|X'\times \partial I^{2} }^{'} \right)=\left({\rm \Theta },{\rm \Theta }'\right)\circ \left(i_{X\times \partial I^{2} } ,i_{X'\times \partial I^{2} } \right)=\] 
\begin{equation} \label{GrindEQ__3_70_} 
\left(\sigma _{\lambda } ,\sigma _{\lambda }' \right)\circ \left(p_{\lambda } \times 1_{\partial I^{2} } ,p_{\lambda }' \times 1_{\partial I^{2} } \right) ,                                
\end{equation} 
then there exist a $\lambda '\ge \lambda $ and fiber homotopies $\left({\rm \Delta },{\rm \Delta }'\right):f_{\lambda } \times 1_{I^{2} } \to t$ and $\left({\rm \Gamma },{\rm \Gamma }'\right):f_{\lambda } \times 1_{I^{2} } \times 1_{I} \to t$ such that 
\[\left({\rm \Delta }_{|X_{\lambda '} \times \partial I^{2} } ,{\rm \Delta }_{|X_{\lambda '}^{'} \times \partial I^{2} }^{'} \right)=\left({\rm \Delta },{\rm \Delta }'\right)\circ \left(i_{X_{\lambda '} \times \partial I^{2} } ,i_{X_{\lambda '}^{'} \times \partial I^{2} } \right)=\] 
\begin{equation} \label{GrindEQ__3_71_} 
\left(\sigma _{\lambda } ,\sigma _{\lambda }' \right)\circ \left(p_{\lambda \lambda '} \times 1_{\partial I} ,p_{\lambda \lambda '}' \times 1_{\partial I} \right) 
\end{equation} 
and the homotopy $\left({\rm \Gamma },{\rm \Gamma }'\right)$ connects $\left({\rm \Theta },{\rm \Theta }'\right)$ and $\left({\rm \Delta },{\rm \Delta }'\right)\circ \left(p_{\lambda } \times 1_{I^{2} } ,p_{\lambda }' \times 1_{I^{2} } \right)$ and is fixed on the submap  $f\times 1_{\partial I^{2} } :X\times 1_{\partial I^{2} } \to X'\times 1_{\partial I^{2} } $.
\end{theorem}
\begin{proof}
 Consider any morphism $\left(\sigma _{\lambda } ,\sigma _{\lambda }' \right):f_{\lambda } \times 1_{\partial I^{2} } \to t$ and a fiber homotopy $\left({\rm \Theta },{\rm \Theta }'\right):f\times 1_{I^{2} } \to t$ such that \eqref{GrindEQ__3_70_} is fulfilled. By theorem 2.1 the $\left( \mathbf{P}\times 1_{I^{2} } :\mathbf{P'}\times 1_{I^{2} } \right):f\times 1_{I^{2} } \to \mathbf{f}\times 1_{I^{2} } $ is a strong fiber expansion and so by the SF1), there exist a $\lambda ''\ge \lambda $, a morphism $\left({\rm \Delta }_{1} ,{\rm \Delta }_{1}^{'} \right):f_{\lambda ''} \times 1_{\partial I^{2} } \to t$  and a fiber homotopy $\left({\rm \Gamma }_{1} ,{\rm \Gamma }_{1}^{'} \right):f_{\lambda } \times 1_{I^{2} } \times 1_{I} \to t$
\begin{equation} \label{GrindEQ__3_72_} 
\left( {{\text{ }\!\!\Gamma\!\!\text{ }}_{1}}_{|X\times {{I}^{2}}\times \left\{ 0 \right\}},\text{ }\!\!\Gamma\!\!\text{ }{{_{1}^{'}}_{|{X}'\times {{I}^{2}}\times \left\{ 0 \right\}}} \right)=\left( {{\text{ }\!\!\Delta\!\!\text{ }}_{1}},\text{ }\!\!\Delta\!\!\text{ }_{1}^{'} \right)\circ \left( {{p}_{{{\lambda }''}}}\times {{1}_{{{I}^{2}}}},p_{{{\lambda }''}}' \times {{1}_{{{I}^{2}}}} \right), 
\end{equation} 
\begin{equation} \label{GrindEQ__3_73_} 
\left( {{\text{ }\!\!\Gamma\!\!\text{ }}_{1}}_{|X\times {{I}^{2}}\times \left\{ 1 \right\}},\text{ }\!\!\Gamma\!\!\text{ }{{_{1}^{'}}_{|{X}'\times {{I}^{2}}\times \left\{ 1 \right\}}} \right)=\left( \text{ }\!\!\Theta\!\!\text{ },\text{ }\!\!\Theta\!\!\text{ }' \right).
\end{equation} 
Consider the morphism $\left(1_{X_{\lambda ''} } \times i_{\partial I^{2} } ,1_{X_{\lambda ''}^{'} } \times i_{\partial I^{2} } \right):f_{\lambda ''} \times 1_{\partial I^{2} } \to f_{\lambda ''} \times 1_{I^{2} } $. In this case the fiber homotopy
\begin{equation} \label{GrindEQ__3_74_} 
\left({\rm \Gamma }_{2} ,{\rm \Gamma }_{2}^{'} \right)=\left({\rm \Gamma }_{1} ,{\rm \Gamma }_{1}^{'} \right)\circ \left(1_{X} \times i_{\partial I^{2} } \times 1_{I} ,1_{X'} \times i_{\partial I^{2} } \times 1_{I} \right) 
\end{equation} 
connects the following morphisms
\begin{equation} \label{GrindEQ__3_75_} 
\left(\left({\rm \Delta }_{1} ,{\rm \Delta }_{1}^{'} \right)\circ \left(1_{X_{\lambda ''} } \times i_{\partial I^{2} } ,1_{X_{\lambda ''}^{'} } \times i_{\partial I^{2} } \right)\right)\circ \; \left(p_{\lambda ''} \times 1_{\partial I^{2} } ,p_{\lambda ''}^{'} \times 1_{\partial I^{2} } \right):f\times 1_{\partial I^{2} } \to t, 
\end{equation} 
\begin{equation} \label{GrindEQ__3_76_} 
\left(\left(\sigma _{\lambda } ,\sigma _{\lambda }' \right)\circ \left(p_{\lambda \lambda ^{''} } \times 1_{\partial I^{2} } ,p_{\lambda \lambda ^{''} }' \times 1_{\partial I^{2} } \right)\right)\circ \left(p_{\lambda ''} \times 1_{\partial I^{2} } ,p_{\lambda ''}' \times 1_{\partial I^{2} } \right)\; :f\times 1_{\partial I^{2} } \to t. 
\end{equation} 
On the other hand, by theorem 2.1 the morphism $\left(\mathbf{p}\times 1_{\partial I^{2} } ,\mathbf{p'}\times 1_{\partial I^{2} } \right)=\left\{\left(p_{\lambda } \times 1_{\partial I^{2} } ,p'_{\lambda } \times 1_{\partial I^{2} } \right) \right\}:f\times 1_{\partial I^{2} } \to \mathbf{f}\times 1_{\partial I^{2}} $ is an expansion and so by the SF2), there exist $\lambda '\ge \lambda ''$ and a fiber homotopy $\left({\rm {\rm K} },{\rm {\rm K} }'\right):f_{\lambda '} \times 1_{\partial I^{2} } \times 1_{I} \to t$ such that
\[ \left( {\rm K }_{|{X_{{\lambda'}  \times \partial {I^2} \times \left\{ 0 \right\}} }}, {\rm K' }_{|{{X'}_{{\lambda'}  \times \partial {I^2} \times \left\{ 0 \right\}} }}  \right)=\]
\begin{equation} \label{GrindEQ__3_78_} 
\left( \left(\Delta_1, \Delta'_1 \right) \circ \left({{1}_{{{X}_{{{\lambda }''}}}}}\times {{i}_{\partial {{I}^{2}}}}{{,1}_{X_{{{\lambda }''}}^{'}}}\times {{i}_{\partial {{I}^{2}}}}\right)   \right) \circ \left( {{p}_{{{\lambda }''}}}\times {{1}_{\partial {{I}^{2}}}},p_{{{\lambda }''}}^{'}\times {{1}_{\partial {{I}^{2}}}} \right),   
\end{equation} 
\[ \left( {\rm \Gamma }_{|{X_{{\lambda'}  \times \partial {I^2} \times \left\{ 1 \right\}} }}, {\rm \Gamma' }_{|{{X'}_{{\lambda'}  \times \partial {I^2} \times \left\{ 1 \right\}} }}  \right)=\]
\begin{equation} \label{GrindEQ__3_79_} 
\left( \left(\sigma_\lambda, \sigma'_\lambda \right) \circ \left( {{p}_{{\lambda \lambda'' }}}\times {{1}_{\partial {{I}^{2}}}},p'_{\lambda \lambda''}\times {{1}_{\partial {{I}^{2}}}} \right)   \right) \circ \left( {{p}_{{{\lambda }''}}}\times {{1}_{\partial {{I}^{2}}}},p_{{{\lambda }''}}'\times {{1}_{\partial {{I}^{2}}}} \right),   
\end{equation}
\begin{equation} \label{GrindEQ__3_80_} 
\left({\rm {\rm K} },{\rm {\rm K} }'\right)\circ \left(p_{\lambda '} \times 1_{\partial I^{2} } \times 1_{I} ,p'_{\lambda '} \times 1_{\partial I^{2} } \times 1_{I} \right)\varphi \cong \left({\rm \Gamma },{\rm \Gamma }'\right)\left(\; rel\; f\times 1_{\partial I^{2} } \times 1_{\partial I} \right). 
\end{equation} 
By lemma 2.4 the morphism $\left(1_{X_{\lambda '} } \times i_{\partial I^{2} } ,1_{X_{\lambda '}^{'} } \times i_{\partial I^{2} } \right):f_{\lambda '} \times 1_{\partial I^{2} } \to f_{\lambda ''} \times 1_{I^{2} } $ is a fiber cofibration. On the other hand, by \eqref{GrindEQ__3_79_} we have
\[\left({\rm {\rm K} }_{|X_{\lambda '} \times \partial I^{2} \times \left\{0\right\}} ,{\rm {\rm K} }_{|X_{\lambda '}^{'} \times \partial I^{2} \times \left\{0\right\}}^{{\rm '}} \right)=\] 
\[\left(\left({\rm \Delta }_{1} ,{\rm \Delta }_{1}^{'} \right)\circ \left(1_{X_{\lambda ''} } \times i_{\partial I^{2} } ,1_{X_{\lambda ''}^{'} } \times i_{\partial I^{2} } \right)\right)\circ \; \left(p_{\lambda ''} \times 1_{\partial I^{2} } ,p'_{\lambda ''} \times 1_{\partial I^{2} } \right)=\] 
\begin{equation} \label{GrindEQ__3_81_} 
\left(\left({\rm \Delta }_{1} ,{\rm \Delta }_{1}^{'} \right)\circ \left(p_{\lambda ''\lambda '} \times 1_{\partial I^{2} } ,p'_{\lambda ''\lambda '} \times 1_{\partial I^{2} } \right) \right) \circ \left(1_{X_{\lambda '} } \times i_{\partial I^{2} } ,1_{X_{\lambda '}^{'} } \times i_{\partial I^{2} } \right). 
\end{equation} 
Therefore, for the fiber homotopy $\left({\rm {\rm K} },{\rm {\rm K} }'\right):f_{\lambda '} \times 1_{\partial I^{2} } \times 1_{I} \to t$ there exists a fiber homotopy $\left({\rm {\rm K} }_{1} ,{\rm {\rm K} }_{1}^{'} \right):f_{\lambda '} \times 1_{I^{2} } \times 1_{I} \to t$ such that 
\begin{equation} \label{GrindEQ__3_82_} 
\left( {\rm K_1}_{|X_{\lambda'} \times I^2 \times \{0\} } , {\rm K'_1}_{|{X'}_{\lambda'} \times I^2 \times \{0\} }\right)=\left( {{\text{ }\!\!\Delta\!\!\text{ }}_{1}},\text{ }\!\!\Delta\!\!\text{ }_{1}^{'} \right)\circ \left( {{p}_{{\lambda }''{\lambda }'}}\times {{1}_{{{I}^{2}}}},p'_{{\lambda }''{\lambda }'}\times {{1}_{{{I}^{2}}}} \right), 
\end{equation} 
\begin{equation} \label{GrindEQ__3_83_} 
\left({\rm {\rm K} },{\rm {\rm K} }'\right)=\left({\rm {\rm K} }_{1} ,{\rm {\rm K} }_{1}^{'} \right)\circ \left(1_{X_{\lambda '} } \times i_{\partial I^{2} } \times 1_{I} ,1_{X_{\lambda '}^{'} } \times i_{\partial I^{2} } \times 1_{I} \right). 
\end{equation} 
Let $\left({\rm \Delta },{\rm \Delta }'\right):f_{\lambda '} \times 1_{I^{2} } \to t$ be a morphism defined by 
\begin{equation} \label{GrindEQ__3_84_} 
\left(\Delta,\Delta^{'} \right)=\left( {{\rm K }_{1}}_{|{{X}_{{{\lambda }'}}}\times {{I}^{2}}\times \left\{ 1 \right\}},\text{ }\!\!\rm K\!\!\text{ }{{_{1}^{'}}_{|X_{{{\lambda }'}}^{'}\times {{I}^{2}}\times \left\{ 1 \right\}}} \right). 
\end{equation} 
In this case, by \eqref{GrindEQ__3_80_}, \eqref{GrindEQ__3_84_} and \eqref{GrindEQ__3_85_} we have
\[\left(\Delta,\text{{ }}\!\!\Delta\!\!\text{ }^{'} \right)\circ \left( {{1}_{{{X}_{{{\lambda }^{'}}}}}}\times {{i}_{\partial {{I}^{2}}}}{{,1}_{X_{{{\lambda }^{'}}}^{'}}}\times {{i}_{\partial {{I}^{2}}}} \right)=\left( {{\text{ }\!\!\Delta\!\!\text{ }}_{|{{X}_{{{\lambda }'}}}\times \partial {{I}^{2}}}},\text{ }\!\!\Delta\!\!\text{ }_{\text{ }\!\!|\!\!\text{ }X_{{{\lambda }^{'}}}^{'}\times \partial {{I}^{2}}\text{ }\!\!\}\!\!\text{ }}^{\text{ }\!\!^{'}\!\!\text{ }} \right)=\]
\begin{equation} \label{GrindEQ__3_85_} 
\left( {{\text{ }\!\!\rm K\!\!\text{ }}_{1}}_{|{{X}_{{{\lambda }^{'}}}}\times {{I}^{2}}\times \left\{ 1 \right\}},\text{ }\!\!\rm K\!\!\text{ }{{_{1}^{'}}_{|X_{{{\lambda }^{'}}}^{'}\times {{I}^{2}}\times \left\{ 1 \right\}}} \right)=\left( {{\sigma }_{\lambda }},\sigma _{\lambda }' \right)\circ \left( {{p}_{\lambda {{\lambda }^{'}}}}\times {{1}_{\partial {{I}^{2}}}},p_{\lambda {{\lambda }^{'}}}'\times {{1}_{\partial {{I}^{2}}}} \right).
\end{equation} 
Therefore, \eqref{GrindEQ__3_71_} is fulfilled. 

By \eqref{GrindEQ__3_81_} there is a fiber homotopy $\left({\rm \Lambda },{\rm \Lambda }'\right):f\times 1_{\partial I^{2} } \times 1_{I} \times 1_{I} \to t$ such that 
\begin{equation} \label{GrindEQ__3_86_} 
\left({\rm \Lambda }_{|X\times \partial I^{2} \times I\times \left\{0\right\}} ,{\rm \Lambda }_{|X'\times \partial I^{2} \times I\times \left\{0\right\}}^{{\rm '}} \right)=\left({\rm {\rm K} },{\rm {\rm K} }'\right)\circ \left(p_{\lambda '} \times 1_{\partial I^{2} } \times 1_{I} ,p'_{\lambda '} \times 1_{\partial I^{2} } \times 1_{I} \right), 
\end{equation} 
\begin{equation} \label{GrindEQ__3_87_} 
\left({\rm \Lambda }_{|X\times \partial I^{2} \times I\times \left\{1\right\}} ,{\rm \Lambda }_{|X'\times \partial I^{2} \times I\times \left\{1\right\}}^{{\rm '}} \right)=\left({\rm \Gamma }_{2} ,{\rm \Gamma }_{2}^{'} \right), 
\end{equation} 
\begin{equation} \label{GrindEQ__3_88_} 
\left({\rm \Lambda }_{|X\times \partial I^{2} \times \partial I\times \left\{t\right\}} ,{\rm \Lambda }_{|X'\times \partial I^{2} \times \partial I\times \left\{t\right\}}^{{\rm '}} \right)=\left({\rm \Lambda }_{|X\times \partial I^{2} \times \partial I\times \left\{0\right\}} ,{\rm \Lambda }_{|X'\times \partial I^{2} \times \partial I\times \left\{0\right\}}^{{\rm '}} \right). 
\end{equation} 
Let $\left({\rm {\rm M} },{\rm {\rm M} }'\right):f\times 1_{\partial I^{3} } \times 1_{I} \to t$ be a morphism defined by
\begin{equation} \label{GrindEQ__3_89_} 
{\rm {\rm M} }\left(x,,u,v,s,t\right)=\left\{\begin{array}{c} {\Lambda \left(x,u,v,t,s\right),\; if\; \left(u,v\right)\in \partial I^{2} } \\ {{\rm {\rm K} }_{1} \left(p_{\lambda '} \left(x\right),u,v,t\right),\; \; \; \; \; \; \; if\; s=0} \\ {\Gamma \left(x,u,v,t\right),\; \; \; \; \; \; \; \; \; \; \; \; \; \; if\; s=1,} \end{array}\right.  
\end{equation} 
\begin{equation} \label{GrindEQ__3_90_} 
{\rm {\rm M} }'\left(x',,u,v,s,t\right)=\left\{\begin{array}{c} {{\rm \Lambda }'\left(x',u,v,t,s\right),\; if\; \left(u,v\right)\in \partial I^{2} } \\ {{\rm {\rm K} }_{1}^{'} \left(p_{\lambda '}' \left(x'\right),u,v,t\right),\; \; \; \; \; \; \; if\; s=0} \\ {{\rm \Gamma }'\left(x',u,v,t\right),\; \; \; \; \; \; \; \; \; \; \; \; \; \; if\; s=1.} \end{array}\right.  
\end{equation} 
Note that by \eqref{GrindEQ__3_75_}, \eqref{GrindEQ__3_87_} and \eqref{GrindEQ__3_88_} the mappings ${\rm {\rm M} }$ and ${\rm {\rm M} }'$ are well defined. Let $\left({\rm \Delta }_{2} ,{\rm \Delta }_{2}^{{\rm '}} \right):f\times 1_{I^{3} } \to t$ be a morphism defined by
\begin{equation} \label{GrindEQ__3_91_} 
{\rm \Delta }_{2} \left(x,u,v,s\right)=\; {\rm \Delta }_{1} \left(p_{\lambda ''} \left(x\right),\; u,v\right), 
\end{equation} 
\begin{equation} \label{GrindEQ__3_92_} 
{\rm \Delta }_{2}^{'} \left(x',u,v,s\right)=\; {\rm \Delta }_{1}^{'} \left(p_{\lambda ''}' \left(x'\right),\; u,v\right). 
\end{equation} 
On the other hand,
\begin{equation} \label{GrindEQ__3_93_} 
\left({\rm {\rm M} }_{|X\times \partial I^{3} \times \left\{0\right\}} ,{\rm {\rm M} }_{|X'\times \partial I^{3} \times \left\{0\right\}}^{{\rm '}} \right)=\left({\rm \Delta }_{2} ,{\rm \Delta }_{2}^{{\rm '}} \right)\circ \left(1_{X} \times 1_{\partial I^{3} } ,1_{X} \times 1_{\partial I^{3} } \right). 
\end{equation} 
Indeed, let $t=0$ and consider three cases:

\begin{enumerate}
\item  If  $\left(u,v\right)\in \partial I^{2} $, then 
\[\left({\rm {\rm M} }\left(x,,u,v,s,0\right),{\rm {\rm M} }'\left(x',,u,v,s,0\right)\right)=\left({\rm \Lambda }\left(x,u,v,0,s\right),{\rm \Lambda }'\left(x',u,v,0,s\right)\right)=\] 

\[\left({\rm \Lambda }\left(x,u,v,0,0\right),{\rm \Lambda }'\left(x',u,v,0,0\right)\right)=\left({\rm {\rm K} }\left(p_{\lambda '} \left(x\right),u,v,0\right),{\rm {\rm K} }'\left(p_{\lambda '}' \left(x'\right),u,v,0\right)\right)=\] 
\begin{equation} \label{GrindEQ__3_94_} 
\left({\rm \Delta }_{1} \left(p_{\lambda ''} \left(x\right),\; u,v\right),{\rm \Delta }_{1}^{'} \left(p_{\lambda ''}' \left(x'\right),\; u,v\right)\right)=\left({\rm \Delta }_{2} \left(x,u,v,s\right),{\rm \Delta }_{2}^{'} \left(x',u,v,s\right)\right). 
\end{equation} 

\item  If $s=0$,  then 
\[\left({\rm {\rm M} }\left(x,,u,v,0,0\right),{\rm {\rm M} }'\left(x',,u,v,0,0\right)\right)=\left({\rm {\rm K} }\left(p_{\lambda ''} \left(x\right),u,v,0\right),{\rm {\rm K} }'\left(p_{\lambda ''}' \left(x'\right),u,v,0\right)\right)=\] 
\begin{equation} \label{GrindEQ__3_95_} 
\left({\rm \Delta }_{1} \left(p_{\lambda ''} \left(x\right),\; u,v\right),{\rm \Delta }_{1}^{'} \left(p_{\lambda ''}^{'} \left(x'\right),\; u,v\right)\right)=\left({\rm \Delta }_{2} \left(x,u,v,0\right),{\rm \Delta }_{2}^{'} \left(x',u,v,0\right)\right). 
\end{equation} 

\item  If $s=1$,  then 
\[\left({\rm {\rm M} }\left(x,,u,v,1,0\right),{\rm {\rm M} }'\left(x',,u,v,1,0\right)\right)=\left({\rm \Delta }_{1} \left(x,\; u,v,0\right),{\rm \Delta }_{1}^{'} \left(x',\; u,v,0\right)\right)=\] 
\begin{equation} \label{GrindEQ__3_96_} 
\left({\rm \Delta }_{1} \left(p_{\lambda ''} \left(x\right),\; u,v\right),{\rm \Delta }_{1}^{'} \left(p_{\lambda ''}' \left(x'\right),\; u,v\right)\right)=\left({\rm \Delta }_{2} \left(x,u,v,1\right),{\rm \Delta }_{2}^{'} \left(x',u,v,1\right)\right). 
\end{equation} 
\end{enumerate}
By lemma 2.4 the morphism $\left(1_{X} \times i_{\partial I^{3} } ,1_{X'} \times i_{\partial I^{3} } \right):f\times 1_{\partial I^{3} } \to f\times 1_{I^{3} } $ is a fiber cofibration and so by \eqref{GrindEQ__3_94_} there exists a fiber homotopy $\left({\rm {\rm N} },{\rm {\rm N} }'\right):f\times 1_{I^{3} } \times 1_{I} \to t$ such that 
\begin{equation} \label{GrindEQ__3_97_} 
\left({\rm {\rm N} }_{|X\times \partial I^{3} \times \left\{0\right\}} ,{\rm {\rm N} }_{|X'\times \partial I^{3} \times \left\{0\right\}}^{{\rm '}} \right)=\left({\rm \Delta }_{2} ,{\rm \Delta }_{2}^{{\rm '}} \right), 
\end{equation} 
\begin{equation} \label{GrindEQ__3_98_} 
\left({\rm {\rm M} },{\rm {\rm M} }'\right)=\left({\rm {\rm N} },{\rm {\rm N} }'\right)\circ \left(1_{X} \times 1_{\partial I^{3} } \times 1_{I} ,1_{X} \times 1_{\partial I^{3} } \times 1_{I} \right). 
\end{equation} 
Let $\left({\rm \Gamma },{\rm \Gamma }'\right):f_{\lambda } \times 1_{I^{3} } \to t$ be morphisms defined by
\begin{equation} \label{GrindEQ__3_99_} 
{\rm \Gamma }\left(x,u,v,s\right)={\rm {\rm N} }\left(x,u,v,s,1\right), 
\end{equation} 
\begin{equation} \label{GrindEQ__3_100_} 
{\rm \Gamma }'\left(x',u,v,s\right)={\rm {\rm N'} }\left(x,u,v,s,1\right). 
\end{equation} 
Now we will show that for the fiber homotopy $\left({\rm \Gamma },{\rm \Gamma }'\right)$ the \eqref{GrindEQ__3_70_} and \eqref{GrindEQ__3_71_} are fulfilled. Indeed,
\[\left({\rm \Gamma }\left(x,,u,v,0\right),{\rm \Gamma }'\left(x',,u,v,0\right)\right)=\left({\rm {\rm N} }\left(x,u,v,0,1\right),{\rm {\rm N} }'\left(x',u,v,0,1\right)\right)=\] 
\[\left({\rm {\rm K} }_{1} \left(p_{\lambda ''} \left(x\right),\; u,v,1\right),{\rm {\rm K} }_{1}^{'} \left(p_{\lambda ''}' \left(x'\right),\; u,v,1\right)\right)=\left({\rm \Delta }\left(p_{\lambda '} \left(x\right),\; u,v\right),{\rm \Delta }'\left(p_{\lambda '}' \left(x'\right),\; u,v\right)\right)=\] 
\begin{equation} \label{GrindEQ__3_101_} 
\left(\left({\rm \Delta }\circ \left(p_{\lambda '} \times 1_{I^{3} } \right)\right)\left(x,\; u,v\right),\left({\rm \Delta }'\circ \left(p_{\lambda '}' \times 1_{I^{3} } \right)\right)\left(x',\; u,v\right)\right), 
\end{equation} 

\[\left({\rm \Gamma }\left(x,,u,v,1\right),{\rm \Gamma }'\left(x',,u,v,1\right)\right)=\left({\rm {\rm N} }\left(x,u,v,1,1\right),{\rm {\rm N} }'\left(x',u,v,1,1\right)\right)=\] 
\begin{equation} \label{GrindEQ__3_102_} 
\left({\rm {\rm M} }\left(x,,u,v,1,1\right),{\rm {\rm M} }'\left(x',,u,v,1,1\right)\right)=\left({\rm \Theta }\left(x,\; u,v\right),{\rm \Theta }'\left(x',\; u,v\right)\right). 
\end{equation} 
Moreover, if $\left(u,v\right)\in \partial I^{2} $, then
\[\left({\rm \Gamma }\left(x,,u,v,t\right),{\rm \Gamma }'\left(x',,u,v,t\right)\right)=\left({\rm {\rm N} }\left(x,u,v,t,1\right),{\rm {\rm N} }'\left(x',u,v,t,1\right)\right)=\] 
\begin{equation} \label{GrindEQ__3_103_} 
\left({\rm {\rm M} }\left(x,,u,v,t,1\right),{\rm {\rm M} }'\left(x',,u,v,t,1\right)\right)=\left({\rm \Lambda }\left(x,,u,v,1,0\right),{\rm \Lambda }'\left(x',,u,v,1,0\right)\right). 
\end{equation} 
Therefore, $\left({\rm \Gamma }\left(x,,u,v,t\right),{\rm \Gamma }'\left(x',,u,v,t\right)\right)$ do not depend on $t$ whenewer $\left(u,v\right)\in \partial I^{2} $.
\end{proof}

\section{Strong fiber shape category}
\

Let $\mathbf{Mor_{CM}} $ be the category of continuous maps of compact metric spaces. By theorem 3.11 of \cite{7}, for each $f\in \mathbf{Mor_{CM}} $ there exist the inverse sequences $\mathbf{X}=\left\{X_{n} ,p_{n,n+1} ,N\right\}$, $\mathbf{X'}=\left\{X_{n}^{'} ,p_{n,n+1}^{'} ,N\right\}$ and a system $\mathbf{f}=\left\{f_{n} ,1_{N} \right\}:\mathbf{X}\to \mathbf X'$\textbf{ }such that:

\begin{enumerate}
	\item  $X= \varprojlim  \mathbf{X} $, $X= \varprojlim \mathbf{X'} $, $f= \varprojlim \mathbf{f} $;
	
	\item  $X_{n} $ and $X_{n}^{'} $ are compact $\mathbf{ANR}$-spaces;
	
	\item  $\left\{p_{n}' \right\} \circ f=\left\{f_{n} ,1_{N} \right\} \circ \left\{p_{n} \right\}$, where $\mathbf{p}=\left\{p_{n} \right\}: X\to \mathbf{X}$ and $\mathbf{q}=\left\{q_{n} \right\}: X'\to \mathbf{X'}$ are the inverse limits of $\mathbf X$ and $\mathbf X'$, respectively.
\end{enumerate}

On the other hand, by \cite{7}, any inverse limit in the category $\mathbf{Mor_{CM}}$ is a resolution and so each continuous map  $f \in \mathbf{Mor_{CM}}$ admits a compact $\mathbf{ANR}$-resolution $\mathbf{{\left(p,p'\right)}}:f \to \mathbf f$, where $\mathbf f$ is an inverse sequence of compact $\mathbf{ANR}$-maps. By this fact, in this section we will construct the category $\mathbf{CH(tow-Mor_{CM})}$ of inverse sequences of continuous maps of compact metric spaces and coherent homotopy classes of coherent mappings of inverse sequences.

Let $\mathbf{f}=\left\{f_{n} ,(p_{n,n+1} ,p_{n,n+1}' )\; ,\mathbb N\right\}$ and $\mathbf{g}=\left\{g_{m} ,(q_{m,m+1} ,q_{n,n+1}' )\; , \mathbb N\right\}$ be inverse sequences in the category $\mathbf {Mor_{CM}} $. The system  $\boldsymbol{\left(\varphi,\varphi'\right)}=\left\{\left(\varphi _{m} ,\varphi _{m}' \right),\left(\varphi _{m,m+1} ,\varphi _{m,m+1}' \right),\varphi \right\}\; :\mathbf{f}\to \mathbf{g}$\textbf{ }is called a coherent morphism, if $\varphi :\mathbb N\to \mathbb N$ is an increasing function, $\left(\varphi _{m} ,\varphi _{m}' \right):f_{\varphi \left(m\right)} \to g_{m} $ and $\left(\varphi _{m,m+1} ,\varphi _{m,m+1}' \right):f_{\varphi \left(m+1\right)} \times 1_{I} \to g_{m} $ are  morphisms such that 
\begin{equation} \label{GrindEQ__4_1_} 
\left(\varphi _{m,m+1} \left(x,0\right),\varphi _{m,m+1}' \left(x',0\right)\right)=\left(\varphi _{m} \left(p_{\varphi \left(m\right),\varphi \left(m+1\right)} \left(x\right)\right),\varphi _{m}' \left(p_{\varphi \left(m\right),\varphi \left(m+1\right)}' \left(x'\right)\right)\right),       
\end{equation} 
\begin{equation} \label{GrindEQ__4_2_} 
\left(\varphi _{m,m+1} \left(x,1\right),\varphi _{m,m+1}' \left(x',1\right)\right)=\left(q_{m,m+1} \left(\varphi _{m+1} \left(x\right)\right),q_{m,m+1}' \left(\varphi _{m+1}' \left(x'\right)\right)\right) .            
\end{equation} 

Let $\boldsymbol{ \left(\varphi ,\varphi '\right)},\; \boldsymbol{\left(\psi ,\psi '\right)} :\mathbf{f}\to \mathbf{ g}$ be two coherent morphisms. We will say that they are coherent homotopic if there exists a coherent morphism
\begin{equation} \label{GrindEQ__4_2.1_} 
\boldsymbol{\left(\Theta ,\Theta '\right)}=\left\{\left({\rm \Theta }_{m} ,{\rm \Theta }_{m}^{'} \right),\left({\rm \Theta }_{m,m+1} ,{\rm \Theta }_{m,m+1}^{'} \right),{\rm \Theta }\right\}\; :\mathbf{f}\times 1_{I} \to \mathbf{g}            
\end{equation}
such that\textbf{ }${\rm \Theta }\left(m\right)\ge \varphi \left(m\right),\; \psi \left(m\right)$ for each $m\in \mathbb N$ and the following is fulfilled  
\begin{equation} \label{GrindEQ__4_3_} 
\left({\rm \Theta }_{m+1} \left(x,0\right),{\rm \Theta }_{m+1}^{'} \left(x',0\right)\right)=\left(\varphi _{m} \left(p_{\varphi \left(m\right),{\rm \Theta }\left(m\right)} \left(x\right)\right),\varphi _{m}' \left(p_{\varphi \left(m\right),{\rm \Theta }\left(m\right)}' \left(x'\right)\right)\right),                     
\end{equation} 
\begin{equation} \label{GrindEQ__4_4_} 
\left({\rm \Theta }_{m+1} \left(x,1\right),{\rm \Theta }_{m+1}^{'} \left(x',1\right)\right)=\left(\psi _{m} \left(p_{\varphi \left(m\right),{\rm \Theta }\left(m\right)} \left(x\right)\right),\psi _{m}' \left(p_{\varphi \left(m\right),{\rm \Theta }\left(m\right)}' \left(x'\right)\right)\right),                    
\end{equation} 
\[\label{GrindEQ__4_5_} 
\left({\rm \Theta }_{m,m+1} \left(x,s,0\right),{\rm \Theta }_{m,m+1}^{'} \left(x,s,0\right)\right)=\]
\begin{equation} \left({\rm \Theta }_{m} \left(p_{{\rm \Theta }\left(m\right),{\rm \Theta }\left(m+1\right)} \left(x\right),s\right),{\rm \Theta }_{m}^{'} \left(p_{{\rm \Theta }\left(m\right),{\rm \Theta }\left(m+1\right)}' \left(x'\right),s\right)\right),           
\end{equation} 

\[\left({\rm \Theta }_{m,m+1} \left(x,s,1\right),{\rm \Theta }_{m,m+1}^{'} \left(x,s,1\right)\right)=\]
\begin{equation} \label{GrindEQ__4_6_} 
\left(q_{m,m+1} \left({\rm \Theta }_{m+1} \left(x\right),s\right),q_{m,m+1}^{'} \left({\rm \Theta }_{m+1}^{'} \left(x'\right),s\right)\right),                  
\end{equation} 

\[\left({\rm \Theta }_{m,m+1} \left(x,0,t\right),{\rm \Theta }_{m,m+1}^{'} \left(x,0,t\right)\right)=\]
\begin{equation} \label{GrindEQ__4_7_} 
\left(\varphi _{m,m+1} \left(p_{\varphi \left(m+1\right),{\rm \Theta }\left(m+1\right)} \left(x\right),t\right),\varphi _{m,m+1}' \left(p_{\varphi \left(m+1\right),{\rm \Theta }\left(m+1\right)}' \left(x'\right),t\right)\right),      
\end{equation} 

\[\left({\rm \Theta }_{m,m+1} \left(x,1,t\right),{\rm \Theta }_{m,m+1}^{'} \left(x,1,t\right)\right)=\]
\begin{equation} \label{GrindEQ__4_8_} 
\left(\psi _{m,m+1} \left(p_{\psi \left(m+1\right),{\rm \Theta }\left(m+1\right)} \left(x\right),t\right),\psi _{m,m+1}' \left(p_{\psi \left(m+1\right),{\rm \Theta }\left(m+1\right)}' \left(x'\right),t\right)\right).      
\end{equation} 
In this case, we will use the notation $\boldsymbol{ \left(\Theta ,\Theta '\right)}: \boldsymbol{\left(\varphi ,\varphi '\right)}\cong \boldsymbol{\left(\psi ,\psi '\right)}$.

Let $\boldsymbol{\left[\left(\varphi ,\varphi '\right)\right]}$ be the ~equivalent ~class ~of ~the ~coherent~ morphism~ $\boldsymbol{\left(\varphi ,\varphi '\right)}$. Denote the category of all inverse sequences of continuous maps of compact metric spaces and coherent homotopy classes $\boldsymbol{\left[\left(\varphi ,\varphi '\right)\right]}$ of a coherent morphisms $\boldsymbol{\left(\varphi ,\varphi '\right)}$ by ~$\mathbf{CH\left(tow-Mor_{CM} \right)}$. Let $\mathbf{CH\left(tow-Mor_{ANR} \right)}$ be the full subcategory of the category $\mathbf{CH\left(tow-Mor_{CM} \right)}$, the objects of which are inverse sequences of $\mathbf{ANR}$-maps. 

\begin{theorem}
If $\mathbf{\left(p,p'\right)}=\left\{\left(p_{\lambda } ,p_{\lambda }' \right)\; \right\}:f\to \mathbf{f}$ is a strong fiber $\mathbf{ANR}$-expansion of a continuous map  $f:X\to X'$ of compact metric spaces, then for each coherent morphism $\boldsymbol{\left(\varphi ,\varphi '\right)}:f\to \mathbf g$, where  $\mathbf{g}\in \mathbf{CH\left(tow-Mor_{ANR} \right)}$, there exists a coherent morphism $\boldsymbol{ \left(\Psi ,\Psi '\right)}:\mathbf{f} \to \mathbf g$ such that $\boldsymbol{\left(\varphi ,\varphi '\right)}$ is a coherent homotopic to $\boldsymbol{ \left(\Psi ,\Psi '\right)} \circ \mathbf{\left(p,p'\right)}.$
\end{theorem}

\begin{proof}
Consider any index $m\in \mathbb N$ and the corresponding morphisms $\left(\varphi _{m} ,\varphi _{m}' \right):f\to g_{m} .$ By the property SF1) of  $\mathbf{\left(p,p'\right)}:f\to \mathbf f$ there exist $\tilde{\psi }\left(m\right)\in \mathbb N$ and morphisms $\left(\tilde{\psi }_{m} ,\tilde{\psi }_{m}' \right):f_{\tilde{\psi }\left(m\right)} \to g_{m} $ such that $\left(\varphi _{m} ,\varphi _{m}' \right)\cong \left(\tilde{\psi }_{m} ,\tilde{\psi }_{m}^{'} \right)\circ \left(p_{\tilde{\psi }\left(m\right)} ,p_{\tilde{\psi }\left(m\right)}' \right)$. Let $\left({\rm \Theta }_{m} ,{\rm \Theta }_{m}^{'} \; \right):f\times 1_{I} \to g_{m} $ be a corresponding homotopy, i.e. 
\begin{equation} \label{GrindEQ__4_9_} 
\left({\rm \Theta }_{m} \left(x,0\right),{\rm \Theta }_{m}^{'} \left(x',0\right)\right)=\left(\varphi _{m} \left(x\right),\varphi _{m}' \left(x'\right)\right), 
\end{equation} 
\begin{equation} \label{GrindEQ__4_10_} 
\left({\rm \Theta }_{m} \left(x,1\right),{\rm \Theta }_{m}^{'} \left(x',1\right)\right)=\left(\tilde{\psi }_{m} \left(p_{\tilde{\psi }\left(m\right)} \right)\left(x\right),\tilde{\psi }_{m}^{'} \left(\; p_{\tilde{\psi }\left(m\right)}' \right)\left(x'\right)\right). 
\end{equation} 
Consider the homotopies $\left({\rm \Theta }_{m} ,{\rm \Theta }_{m}^{'} \; \right),\left({\rm \Theta }_{m+1} ,{\rm \Theta }_{m+1}^{'} \; \right),\left(\varphi _{m,m+1} ,\varphi _{m,m+1}' \right):f\times 1_{I} \to g_{m} $. Let $\left({\rm \Gamma }_{m,m+1} ,{\rm \Gamma }_{m,m+1}^{'} \right):f\times 1_{\partial I} \times 1_{I} \to g_{m} $ be a morphism defined by
\begin{equation} \label{GrindEQ__4_11_} 
\left({\rm \Gamma }_{m,m+1}(x,s,t) ,{\rm \Gamma }_{m,m+1}^{'}(x',s,t) \right)=\begin{cases}
\left({\rm \Theta }_{m}(x,t) ,{\rm \Theta }_{m}^{'}(x',t)\right), \; if s=0  \\ 
\left(q_{m,m+1}{\rm \Theta }_{m+1}(x,t) ,q'_{m,m+1}{\rm \Theta }_{m+1}^{'}(x',t)\right), \; if s=1. 
\end{cases}
\end{equation} 
In this case, by \eqref{GrindEQ__4_9_}, \eqref{GrindEQ__4_10_}, \eqref{GrindEQ__4_1_} and \eqref{GrindEQ__4_2_} we have 
\begin{equation} \label{GrindEQ__4_12_} 
\left( {{\text{ }\!\!\Gamma\!\!\text{ }}_{m,m+1}}_{|X\times \partial I\times \left\{ 0 \right\}}~~,\text{ }\!\!\Gamma\!\!\text{ }{{_{m,m+1}^{'}}_{|{X}'\times \partial I\times \left\{ 0 \right\}}} \right)=\left( {{\varphi }_{m,m+1}}_{|X\times \partial I\times \left\{ 0 \right\}},\varphi {{_{m,m+1}'}_{|{X}'\times \partial I\times \left\{ 0 \right\}}} \right).           
\end{equation} 
On the other hand, $\left(1_{X} \times i_{\partial I} ,1_{X'} \times i_{\partial I} \right):f\times 1_{\partial I} \to f\times 1_{I} $ is a fiber cofibration and so there exists a fiber homotopy $\left(\tilde{{\rm \Gamma }}_{m,m+1} ,\tilde{{\rm \Gamma }}_{m,m+1}^{'} \right):f\times 1_{I} \times 1_{I} \to g_{m} $ such that 
\begin{equation} \label{GrindEQ__4_13_} 
\left(\tilde{{\rm \Gamma }}_{m,m+1 {|X\times \partial I\times \left\{ 0 \right\}}} ,\tilde{{\rm \Gamma }}_{m,m+1 {|X'\times \partial I\times \left\{ 0 \right\}}}^{'} \right)=\left( {{\varphi }_{m,m+1}},\varphi _{m,m+1}' \right)
\end{equation} 
\begin{equation} \label{GrindEQ__4_14_} 
\left({\rm \Gamma }_{m,m+1} ,{\rm \Gamma }_{m,m+1}^{'} \; \right)=\left(\tilde{{\rm \Gamma }}_{m,m+1} ,\tilde{{\rm \Gamma }}_{m,m+1}^{'} \; \right)\circ \left(1_{X} \times i_{\partial I} ,1_{X'} \times i_{\partial I} \right).                        
\end{equation} 
Note that by \eqref{GrindEQ__4_14_}, \eqref{GrindEQ__4_11_} and \eqref{GrindEQ__4_10_}, 
\[\left(\tilde{{\rm \Gamma }}_{m,m+1} \left(x,0,1\right),\tilde{{\rm \Gamma }}_{m,m+1}^{'} \left(x',0,1\right)\right)=\left({\rm \Gamma }_{m,m+1} \left(x,0,1\right),{\rm \Gamma }_{m,m+1}^{'} \; \left(x',0,1\right)\right)=\] 
\begin{equation} \label{GrindEQ__4_15_} 
\left({\rm \Theta }_{m} \left(x,1\right),{\rm \Theta }_{m}^{'} \; \left(x',1\right)\right)=\left(\tilde{\psi }_{m} \left(p_{\tilde{\psi }\left(m\right)} \right)\left(x\right),\tilde{\psi }_{m}^{'} \left(\; p_{\tilde{\psi }\left(m\right)}' \right)\left(x'\right)\right), 
\end{equation} 
\[\left(\tilde{{\rm \Gamma }}_{m,m+1} \left(x,1,1\right),\tilde{{\rm \Gamma }}_{m,m+1}^{'} \left(x',1,1\right)\right)=\left({\rm \Gamma }_{m,m+1} \left(x,1,1\right),{\rm \Gamma }_{m,m+1}^{'} \; \left(x',1,1\right)\right)=\] 
\[\left(q_{m,m+1} {\rm \Theta }_{m+1} \left(x,1\right),q_{m,m+1}' {\rm \Theta }_{m+1}^{'} \; \left(x',1\right)\right)=\] 
\begin{equation} \label{GrindEQ__4_16_} 
\left(\left(q_{m,m+1} \circ \tilde{\psi }_{m+1} \circ p_{\tilde{\psi }\left(m\right)} \right)\left(x\right),\left(q_{m,m+1}' \circ \tilde{\psi }_{m+1}^{'} \circ p_{\tilde{\psi }\left(m\right)}' \right)\left(x'\right)\right). 
\end{equation} 
Let $\left(\tilde{\varphi }_{m,m+1} ,\tilde{\varphi }_{m,m+1}' \right):f\times 1_{I} \to g_{m} $ be a morphism given by 
\begin{equation} \label{GrindEQ__4_17_} 
\left(\tilde{\varphi }_{m,m+1} \left(x,t\right),\tilde{\varphi }_{m,m+1}' \left(x',t\right)\; \right)=\left(\tilde{{\rm \Gamma }}_{m,m+1} \left(x,t,1\right),\tilde{{\rm \Gamma }}_{m,m+1}^{'} \left(x',t,1\right)\right).\;  
\end{equation} 
Consider $n=\max \left(\tilde{\psi }\left(m\right),\tilde{\psi }\left(m+1\right)\right)$ and define a morphism $\left(\sigma _{n} ,\sigma _{n}' \right):f_{n} \times 1_{I} \to g_{m} $ by the following formula
\[\left(\sigma _{n} \left(x,s,t\right),\sigma _{n}' \left(x',s,t\right)\right)=\]
\begin{equation} \label{GrindEQ__4_18_} 
=\begin{cases}
\left(\left(\tilde{\psi}_m \circ p_{\tilde{\psi}_m(m)}(x)\right), \left(\tilde{\psi'}_m \circ p'_{\tilde{\psi'}_m(m)} (x')\right) \right), \; \; \; if \; t=0\\
\left(\left(q_{m,m+1} \circ \tilde{\psi}_{m+1} \circ p_{\tilde{\psi}_m(m)}(x)\right), \left(q'_{m,m+1} \circ \tilde{\psi'}_{m+1} \circ p'_{\tilde{\psi}^{'}_m(m)} (x')\right) \right), \; \; \; if \; t=1.
\end{cases}
\end{equation} 
In this case, by \eqref{GrindEQ__4_17_}, \eqref{GrindEQ__4_18_}, \eqref{GrindEQ__4_15_} and \eqref{GrindEQ__4_16_}, it is clear that 
\[\left( {\tilde{\varphi}}_{m,m+1|X\times \partial I},{\tilde{\varphi}'}_{m,m+1|X'\times \partial I}  \right)=\]
\[\left( {{{\tilde{\varphi }}}_{m,m+1}},\tilde{\varphi }_{m,m+1}' \right)\circ \left( {{1}_{X}}\times {{i}_{\partial I}}{{,1}_{{{X}^{'}}}}\times {{i}_{\partial I}} \right)=\]

\begin{equation} \label{GrindEQ__4_19_} 
\left( {{\sigma }_{n}},\sigma _{n}' \right)\cdot \left( {{p}_{n}}\times {{i}_{\partial I}},p_{n}'\times {{i}_{\partial I}} \right).
\end{equation} 
Therefore, by the property SF2) of $\mathbf{\left(p,p'\right)}:f\to \mathbf f$, there exist $n'\in \mathbb N$, which will be denoted by $\tilde{\psi }\left(m,m+1\right)$, and $\left(\tilde{\psi }_{m,m+1} ,\tilde{\psi }_{m,m+1}^{'} \right):f_{\tilde{\psi }\left(m,m+1\right){\rm \; }} \times 1_{I} \to g_{m} $ such that
\[\left(\tilde{\psi }_{m,m+1} ,\tilde{\psi }_{m,m+1}^{'} \right)\circ \left(1_{X_{\tilde{\psi }\left(m,m+1\right)} } \times i_{\partial I} ,1_{X_{\tilde{\psi }\left(m,m+1\right)}^{'} } \times i_{\partial I} \right)=\] 
\begin{equation} \label{GrindEQ__4_20_} 
\left(\sigma _{n} ,\sigma _{n}' \right)\circ \left(p_{n,\tilde{\psi }\left(m,m+1\right)} \times i_{\partial I} ,p_{n,\tilde{\psi }\left(m,m+1\right)}' \times i_{\partial I} \right), 
\end{equation} 
\begin{equation} \label{GrindEQ__4_21_} 
\left(\tilde{\varphi }_{m,m+1} ,\tilde{\varphi }_{m,m+1}^{'} \right)\cong \left(\tilde{\psi }_{m,m+1} ,\tilde{\psi }_{m,m+1}^{'} \right)\cdot \left(p_{\tilde{\psi }\left(m,m+1\right)} ,p_{\tilde{\psi }\left(m,m+1\right)}' \right)\left(rel\left\{f\times 1_{\partial I} \right\}\right). 
\end{equation} 

Let $\psi \left(m,m+1\right)=\mathop{\max }\limits_{n\le m} \left\{\tilde{\psi }\left(n,n+1\right)\right\}$, then $\psi : \mathbb N\to \mathbb N$ will be an increasing function. Define the morphisms $\left(\psi _{m} ,\psi _{m}^{'} \right):f_{\psi \left(m\right)} \times 1_{I} \to g_{m} $ and $\left(\psi _{m,m+1} ,\psi _{m,m+1}^{'} \right):f_{\psi \left(m+1\right)} \times 1_{I} \to g_{m} $ by the following way
\begin{equation} \label{GrindEQ__4_22_} 
\left(\psi _{m} ,\psi _{m}^{'} \right)=\left(\tilde{\psi }_{m} ,\tilde{\psi }_{m}^{'} \right)\circ \left(p_{\tilde{\psi }\left(m\right),\psi \left(m\right)} ,p_{\tilde{\psi }\left(m\right),\psi \left(m\right)}' \right),                           
\end{equation} 
\begin{equation} \label{GrindEQ__4_23_} 
\left(\psi _{m,m+1} ,\psi _{m,m+1}^{'} \right)=\left(\tilde{\psi }_{m,m+1} ,\tilde{\psi }_{m,m+1}^{'} \right)\circ \left(p_{\tilde{\psi }\left(m,m+1\right),\psi \left(m+1\right)} \times 1_{I} ,p_{\tilde{\psi }\left(m,m+1\right),\psi \left(m+1\right)}' \times 1_{I} \right). 
\end{equation} 
By \eqref{GrindEQ__4_22_}, \eqref{GrindEQ__4_20_}, \eqref{GrindEQ__4_18_} and \eqref{GrindEQ__4_22_},
\[\left(\psi _{m,m+1} \left(x,0\right),\psi _{m,m+1}^{'} \left(x',0\right)\right)=\] 
\[\left(\left(\tilde{\psi }_{m,m+1} \circ \left(p_{\tilde{\psi }\left(m,m+1\right),\psi \left(m+1\right)} \times 1_{I} \right)\right)\left(x,0\right),\left(\tilde{\psi }_{m,m+1}^{'} \circ \left(p_{\tilde{\psi }\left(m,m+1\right),\psi \left(m+1\right)}' \times 1_{I} \right)\right)\left(x',0\right)\right)=\] 
{\tiny\[\left( \left(\sigma _{n} \circ \left(p_{n,\tilde{\psi }\left(m,m+1\right)} \times 1_{I} \right)\circ \left(p_{\tilde{\psi }\left(m,m+1\right), \psi \left(m+1\right)} \times 1_{I} \right)\right)(x,0),\left(\sigma _{n}' \circ \left(p_{n,\tilde{\psi }\left(m,m+1\right)}' \times 1_{I} \right)\circ \left(p_{\tilde{\psi }\left(m,m+1\right),\psi \left(m+1\right)}' \times 1_{I} \right)(x',0)\right)\right)=\]}
\[\left(\sigma _{n} \left(p_{n,\psi \left(m+1\right)} \left(x\right),0\right),\sigma _{n}' \left(p_{n,\psi \left(m+1\right)}' \left(x'\right),0\right)\; \right)=\] 
\[\left(\left(\tilde{\psi }_{m} \circ p_{\tilde{\psi }\left(m\right),n} \right)\left(p_{n,\psi \left(m+1\right)} \left(x\right)\right),\left(\tilde{\psi }_{m}^{'} \circ p_{\tilde{\psi }\left(m\right),n}^{'} \right)\left(p_{n,\psi \left(m+1\right)}' \left(x'\right)\right)\; \right)=\] 
\[\left(\left(\tilde{\psi }_{m} \circ p_{\tilde{\psi }\left(m\right),\psi \left(m+1\right)} \right)\left(x\right),\left(\tilde{\psi }_{m}^{'} \circ p_{\tilde{\psi }\left(m\right),\psi \left(m+1\right)}' \right)\left(x'\right)\; \right)=\] 
\[\left(\left(\tilde{\psi }_{m} \circ p_{\tilde{\psi }\left(m\right),\psi \left(m\right)} \right)\left(p_{\psi \left(m\right),\psi \left(m+1\right)} \left(x\right)\right),\left(\tilde{\psi }_{m}' \circ p_{\tilde{\psi }\left(m\right),\psi \left(m\right)}' \right)\left(p_{\psi \left(m\right),\psi \left(m+1\right)}' \left(x'\right)\right)\right)=\] 
\begin{equation} \label{GrindEQ__4_24_} 
\left(\psi _{m} \left(p_{\psi \left(m\right),\psi \left(m+1\right)} \left(x\right)\right),\psi _{m}^{'} \left(p_{\psi \left(m\right),\psi \left(m+1\right)}' \left(x'\right)\right)\; \right). 
\end{equation} 

\[\left(\psi _{m,m+1} \left(x,1\right),\psi _{m,m+1}^{'} \left(x',1\right)\right)=\] 
\[\left(\left(\tilde{\psi }_{m,m+1} \circ \left(p_{\tilde{\psi }\left(m,m+1\right),\psi \left(m+1\right)} \times 1_{I} \right)\right)\left(x,1\right),\left(\tilde{\psi }_{m,m+1}^{'} \circ \left(p_{\tilde{\psi }\left(m,m+1\right),\psi \left(m+1\right)}' \times 1_{I} \right)\right)\left(x',1\right)\right)=\] 
{\tiny \[\left(\left(\sigma _{n} \circ \left(p_{n,\tilde{\psi }\left(m,m+1\right)} \times 1_{I} \right)\circ \left(p_{\tilde{\psi }\left(m,m+1\right),\psi \left(m+1\right)} \times 1_{I} \right)\right)\left(x,1\right),\left(\sigma _{n}' \circ \left(p_{n,\tilde{\psi }\left(m,m+1\right)}' \times 1_{I} \right)\circ \left(p_{\tilde{\psi }\left(m,m+1\right),\psi \left(m+1\right)}' \times 1_{I} \right)\right)\left(x',1\right)\right)=\] }
\[\left(\sigma _{n} \left(p_{n,\psi \left(m+1\right)} \left(x\right),1\right),\sigma _{n}' \left(p_{n,\psi \left(m+1\right)}' \left(x'\right),1\right)\; \right)=\] 
\[\left(\left(q_{m,m+1} \circ \tilde{\psi }_{m+1} \circ p_{\tilde{\psi }\left(m+1\right),n} \right)\left(p_{n,\psi \left(m+1\right)} \left(x\right)\right),\left(q_{m,m+1}' \circ \tilde{\psi }_{m+1}^{'} \circ p_{\tilde{\psi }\left(m+1\right),n}' \right)\left(p_{n,\psi \left(m+1\right)}' \left(x'\right)\right)\; \right)=\] 
\[\left(\left(q_{m,m+1} \circ \tilde{\psi }_{m+1} \circ p_{\tilde{\psi }\left(m\right),\psi \left(m+1\right)} \right)\left(x,0\right),\left(q_{m,m+1}' \circ \tilde{\psi }_{m+1}^{'} \circ p_{\tilde{\psi }\left(m\right),\psi \left(m+1\right)}' \right)\left(x',0\right)\; \right)=\] 
\begin{equation} \label{GrindEQ__4_25_} 
\left(\left(q_{m,m+1} \circ \psi _{m+1} \right)\left(x,0\right),\left(q_{m,m+1}' \circ \psi _{m+1}^{'} \right)\left(x',0\right)\; \right). 
\end{equation} 
Therefore, we construct the system $\boldsymbol{\left(\Psi ,\Psi '\right)}=\left\{\left(\psi _{m} ,\psi _{m}^{'} \right),\left(\psi _{m,m+1,} \psi _{m,m+1,}^{'} \right)\right\}:\mathbf{f}\to \mathbf{g}$\textbf{ }which satisfies the following conditions:
\begin{equation} \label{GrindEQ__4_26_} 
\left(\psi _{m,m+1} \left(x,0\right),\psi _{m,m+1}' \left(x',0\right)\right)=\left(\left(\psi _{m} \circ p_{\psi \left(m\right),\psi \left(m+1\right)} \right)\left(x\right),\left(\psi _{m}' \cdot p_{\psi \left(m\right),\psi \left(m+1\right)}' \right)\left(x'\right)\; \right), 
\end{equation} 
\begin{equation} \label{GrindEQ__4_27_} 
\left(\psi _{m,m+1} \left(x,1\right),\psi _{m,m+1}' \left(x',1\right)\right)=\left(\left(q_{m,m+1} \circ \psi _{m+1} \right)\left(x\right),\left(q_{m,m+1}' \cdot \psi _{m+1}' \right)\left(x'\right)\; \right). 
\end{equation} 
So $\boldsymbol{\left(\Psi ,\Psi '\right)}$ is a coherent morphism.

Let show that the morphism $\boldsymbol{\left(\varphi ,\varphi '\right)}$ is a coherent homotopic to $\boldsymbol{\left(\psi ,\psi '\right)} \circ \mathbf{\left(p,p'\right)}.$ For this we should construct a system
\[\boldsymbol{\left(\Theta ,\Theta '\right)}=\left\{\left({\rm \Theta }_{m} ,{\rm \Theta }_{m}^{'} \right),\left({\rm \Theta }_{m,m+1} ,{\rm \Theta }_{m,m+1}^{'} \right),{\rm \Theta }\right\} :\mathbf{f}\times 1_{I} \to \mathbf{g}\] 
such that\textbf{ }
\begin{equation} \label{GrindEQ__4_28_} 
\left({\rm \Theta }_{m} \left(x,0\right),{\rm \Theta }_{m}^{'} \left(x',0\right)\right)=\left(\varphi _{m} \left(x\right),\varphi _{m}' \left(x'\right)\right),                                           
\end{equation} 
\begin{equation} \label{GrindEQ__4_29_} 
\left({\rm \Theta }_{m} \left(x,1\right),{\rm \Theta }_{m}^{'} \left(x',1\right)\right)=\left(\psi _{m} \left(p_{\psi \left(m\right)} \left(x\right)\right),\psi _{m}' \left(p_{\psi \left(m\right)}' \left(x'\right)\right)\right),                         
\end{equation} 
\begin{equation} \label{GrindEQ__4_30_} 
\left({\rm \Theta }_{m,m+1} \left(x,s,0\right),{\rm \Theta }_{m,m+1}^{'} \left(x,s,0\right)\right)=\left({\rm \Theta }_{m} \left(x,s\right),{\rm \Theta }_{m}^{'} \left(x',s\right)\right),                          
\end{equation} 
\[\left({\rm \Theta }_{m,m+1} \left(x,s,1\right),{\rm \Theta }_{m,m+1}^{'} \left(x,s,1\right)\right)=\]
\begin{equation} \label{GrindEQ__4_31_} 
\left(q_{m,m+1} \left({\rm \Theta }_{m+1} \left(x\right),s\right),q_{m,m+1}' \left({\rm \Theta }_{m+1}^{'} \left(x'\right),s\right)\right),      
\end{equation} 
\begin{equation} \label{GrindEQ__4_32_} 
\left({\rm \Theta }_{m,m+1} \left(x,0,t\right),{\rm \Theta }_{m,m+1}^{'} \left(x,0,t\right)\right)=\left(\varphi _{m,m+1} \left(x,t\right),\varphi _{m,m+1}' \left(x',t\right)\right),                
\end{equation} 
\[\left({\rm \Theta }_{m,m+1} \left(x,1,t\right),{\rm \Theta }_{m,m+1}^{'} \left(x,1,t\right)\right)=\]
\begin{equation} \label{GrindEQ__4_33_} 
\left(\psi _{m,m+1} \left(p_{\psi \left(m+1\right)} \left(x\right),t\right),\psi _{m,m+1}' \left(p_{\psi \left(m+1\right)}' \left(x'\right),t\right)\right).     
\end{equation} 
Note that for each $n\in \mathbb N$ we have defined a morphism $\left({\rm \Theta }_{m} ,{\rm \Theta }_{m}^{'} \; \right):f\times 1_{I} \to g_{m} $ and by \eqref{GrindEQ__4_9_}, \eqref{GrindEQ__4_10_} and \eqref{GrindEQ__4_22_},
\begin{equation} \label{GrindEQ__4_34_} 
\left({\rm \Theta }_{m} \left(x,0\right),{\rm \Theta }_{m}^{'} \left(x',0\right)\right)=\left(\varphi _{m} \left(x\right),\varphi _{m}' \left(x'\right)\right),                                           
\end{equation} 
\[\left({\rm \Theta }_{m} \left(x,1\right),{\rm \Theta }_{m}^{'} \left(x',1\right)\right)=\left(\left(\tilde{\psi }_{m} \cdot p_{\tilde{\psi }\left(m\right)} \right)\left(x\right),\left(\tilde{\psi }_{m}'\cdot p_{\tilde{\psi }\left(m\right)}' \right)\left(x'\right)\right)=\] 
\[\left(\left(\tilde{\psi }_{m} \cdot p_{\tilde{\psi }\left(m\right),\psi \left(m\right)} \right)\left(p_{\psi \left(m\right)} \left(x\right)\right),\left(\tilde{\psi }_{m}' \cdot p_{\tilde{\psi }\left(m\right),\psi \left(m\right)}' \right)\left(p_{\psi \left(m\right)}' \left(x'\right)\right)\right)=\] 
\begin{equation} \label{GrindEQ__4_35_} 
\left(\left(\psi _{m} \cdot p_{\psi \left(m\right)} \right)\left(x\right),\left(\psi _{m}' \cdot p_{\psi \left(m\right)}' \right)\left(x'\right)\right).\;  
\end{equation} 
So it remains to construct a morphism $\left({\rm \Theta }_{m+1} ,{\rm \Theta }_{m+1}^{'} \; \right):f\times 1_{I} \times 1_{I} \to g_{m} $. Consider a homotopy $\left({\rm \Delta }_{m+1} ,{\rm \Delta }_{m+1}^{'} \; \right):f\times 1_{I} \times 1_{I} \to g_{m} $ which realizes \eqref{GrindEQ__4_21_}, i.e
\begin{equation} \label{GrindEQ__4_36_} 
\left({\rm \Delta }_{m} \left(x,s,0\right),{\rm \Delta }_{m}^{'} \left(x',s,0\right)\right)=\left(\tilde{\varphi }_{m,m+1} \left(x,s\right),\tilde{\varphi }_{m,m+1}' \left(x',s\right)\right),                        
\end{equation} 
\[\left({\rm \Delta }_{m} \left(x,s,1\right),{\rm \Delta }_{m}^{'} \left(x',s,1\right)\right)=\]
\begin{equation} \label{GrindEQ__4_37_} 
\left(\tilde{\psi }_{m,m+1} \left(p_{\tilde{\psi }\left(m,m+1\right)} \left(x\right),s\right),\tilde{\psi }_{m,m+1}' \left(p_{\tilde{\psi }\left(m,m+1\right)}' \left(x'\right),s\right)\right), 
\end{equation} 
\begin{equation} \label{GrindEQ__4_38_} 
\left({\rm \Delta }_{m} \left(x,s,t\right),{\rm \Delta }_{m}^{'} \left(x',s,t\right)\right)=\left({\rm \Delta }_{m} \left(x,s,0\right),{\rm \Delta }_{m}^{'} \left(x',s,0\right)\right),\; \; \; \; \; \; \; \forall \; \; s\in \partial I. 
\end{equation} 
Let $\left(\tilde{{\rm \Theta }}_{m+1} ,\tilde{{\rm \Theta }}_{m+1}^{'} \; \right):f\times 1_{I} \times 1_{I} \to g_{m} $ is given by the formula:
\[\left(\tilde{{\rm \Theta }}_{m+1} \left(x,s,t\right),\tilde{{\rm \Theta }}_{m+1}^{'} \left(x',s,t\right)\right)=\]
\begin{equation} \label{GrindEQ__4_39_} 
=\begin{cases}
{\left(\tilde{{\rm \Gamma }}_{m,m+1} \left(x,t,2s\right),\tilde{{\rm \Gamma }}'_{m,m+1} \left(x',t,2s\right)\right),\; \; \; \; \; \; \; \; \; \; \; \; \; \; \; \; \; if\; 0\le 2s\le \frac{1}{2} \; } \\ {\left({\rm \Delta }_{m,m+1} \left(x,t,2s-1\right),{\rm \Delta }_{m,m+1} \left(x',t,2s-1\right)\right),\; \; \; if\; \; \; \frac{1}{2} \le 2s\le 1.} 
\end{cases} 
\end{equation} 
By \eqref{GrindEQ__4_17_} and \eqref{GrindEQ__4_36_} the morphism  $\left(\tilde{{\rm \Theta }}_{m+1} ,\tilde{{\rm \Theta }}_{m+1}^{'} \; \right)\; $ is well defined. On the other hand, by \eqref{GrindEQ__2_38_} for each $x\in X$ and $x'\in X'$ the homotopies $\tilde{{\rm \Theta }}_{m+1} $ and $\tilde{{\rm \Theta }}_{m+1}^{'} $ are constant on the subspaces $\left\{x\right\}\times \left\{0\right\}\times \left[\frac{1}{2} ;1\right]$, $\left\{x\right\}\times \left\{1\right\}\times \left[\frac{1}{2} ;1\right]$ and  $\left\{x'\right\}\times \left\{0\right\}\times \left[\frac{1}{2} ;1\right]$, $\left\{x\right\}\times \left\{0\right\}\times \left[\frac{1}{2} ;1\right]$, respectively. Consider the quotient space $Z$ of $I\times I$, where $\left\{0\right\}\times \left[\frac{1}{2} ;1\right]$ and $\left\{1\right\}\times \left[\frac{1}{2} ;1\right]$ are identified with the points $\left(0,\frac{1}{2} \right)$ and $\left(1,\frac{1}{2} \right)$, respectively. In this case, there exists a morphism $\left(\bar{{\rm \Theta }}_{m,m+1} ,\bar{{\rm \Theta }}_{m,m+1}^{'} \right):f\times 1_{Z} \to g_{m} $ such that 
\begin{equation} \label{GrindEQ__4_40_} 
\left(\tilde{{\rm \Theta }}_{m,m+1} ,\tilde{{\rm \Theta }}_{m,m+1}^{'} \; \right)=\left(\bar{{\rm \Theta }}_{m,m+1} ,\bar{{\rm \Theta }}_{m,m+1}^{'} \right)\circ \left(1_{X} \times q,\; 1_{X} \times q\right),                
\end{equation} 
where $q:I\times I\to Z$ is the quotient map. Let $k:I\times I\to Z$ be a homeomorphism such that
\begin{equation} \label{GrindEQ__4_41_} 
k\left(s,0\right)=q\left(\frac{s}{2} ,0\right),                                                        
\end{equation} 
\begin{equation} \label{GrindEQ__4_42_} 
k\left(s,1\right)=q\left(\frac{s}{2} ,1\right), 
\end{equation} 
\begin{equation} \label{GrindEQ__4_43_} 
k\left(s,1\right)=\left(0,t\right),                                                         
\end{equation} 
\begin{equation} \label{GrindEQ__4_44_} 
k\left(s,1\right)=\left(0,t\right). 
\end{equation} 
Let $\left({\rm \Theta }_{m,m+1} ,{\rm \Theta }_{m,m+1}^{'} \right):f\times 1_{Z} \to g_{m} $ be a morphism given by 
\begin{equation} \label{GrindEQ__4_45_} 
\left({\rm \Theta }_{m,m+1} \left(x,s,t\right),{\rm \Theta }_{m,m+1}^{'} \left(x',s,t\right)\right)=\left(\bar{{\rm \Theta }}_{m,m+1} \left(x,k\left(s,t\right)\right),\bar{{\rm \Theta }}_{m,m+1}^{'} \left(x',k\left(s,t\right)\right)\right). 
\end{equation} 
In this case, by \eqref{GrindEQ__4_45_}, \eqref{GrindEQ__4_41_}, \eqref{GrindEQ__4_42_}, \eqref{GrindEQ__4_40_}, \eqref{GrindEQ__4_39_} and \eqref{GrindEQ__4_11_},
\[\left({\rm \Theta }_{m,m+1} \left(x,s,0\right),{\rm \Theta }_{m,m+1}^{'} \left(x',s,0\right)\right)=\left(\bar{{\rm \Theta }}_{m,m+1} \left(x,k\left(s,0\right)\right),\bar{{\rm \Theta }}_{m,m+1}^{'} \left(x',k\left(s,0\right)\right)\right)=\] 
\[\left(\bar{{\rm \Theta }}_{m,m+1} \left(x,q\left(\frac{s}{2} ,0\right)\right),\bar{{\rm \Theta }}_{m,m+1}^{'} \left(x',q\left(\frac{s}{2} ,0\right)\right)\right)=\left(\tilde{{\rm \Theta }}_{m,m+1} \left(x,\frac{s}{2} ,0\right),\tilde{{\rm \Theta }}_{m,m+1}^{'} \left(x',\frac{s}{2} ,0\right)\; \right)=\] 
\begin{equation} \label{GrindEQ__4_46_} 
\left(\tilde{{\rm \Gamma }}_{m,m+1} \left(x,0,s\right),\tilde{{\rm \Gamma }}'_{m,m+1} \left(x',0,s\right)\right)=\left({\rm \Theta }_{m} \left(x,s\right),{\rm \Theta }_{m}^{'} \left(x',s\right)\right), 
\end{equation} 

\[\left({\rm \Theta }_{m,m+1} \left(x,s,1\right),{\rm \Theta }_{m,m+1}^{'} \left(x',s,1\right)\right)=\left(\bar{{\rm \Theta }}_{m,m+1} \left(x,k\left(s,1\right)\right),\bar{{\rm \Theta }}_{m,m+1}^{'} \left(x',k\left(s,1\right)\right)\right)=\] 
\[\left(\bar{{\rm \Theta }}_{m,m+1} \left(x,q\left(\frac{s}{2} ,1\right)\right),\bar{{\rm \Theta }}_{m,m+1}^{'} \left(x',q\left(\frac{s}{2} ,1\right)\right)\right)=\left(\tilde{{\rm \Theta }}_{m,m+1} \left(x,\frac{s}{2} ,1\right),\tilde{{\rm \Theta }}'_{m,m+1} \left(x',\frac{s}{2} ,1\right)\; \right)=\] 
\begin{equation} \label{GrindEQ__4_47_} 
\left(\tilde{{\rm \Gamma }}_{m,m+1} \left(x,1,s\right),\tilde{{\rm \Gamma }}'_{m,m+1} \left(x',1,s\right)\right)=\left(\left(q_{m,m_{1} } \cdot {\rm \Theta }_{m+1} \right)\left(x,s\right),\left(q_{m,m_{1} } \cdot {\rm \Theta '}_{m+1} \right)\left(x',s\right)\right).
\end{equation}

\noindent On the other hand, by \eqref{GrindEQ__4_45_}, \eqref{GrindEQ__4_43_}, \eqref{GrindEQ__4_44_}, \eqref{GrindEQ__4_40_}, \eqref{GrindEQ__4_12_} and \eqref{GrindEQ__4_37_},

\noindent 
\[\left({\rm \Theta }_{m,m+1} \left(x,0,t\right),{\rm \Theta }_{m,m+1}^{'} \left(x',0,t\right)\right)=\left(\bar{{\rm \Theta }}_{m,m+1} \left(x,k\left(0,t\right)\right),\bar{{\rm \Theta }}_{m,m+1}^{'} \left(x',k\left(0,t\right)\right)\right)=\] 
\[\left(\bar{{\rm \Theta }}_{m,m+1} \left(x,0,t\right),\bar{{\rm \Theta }}_{m,m+1}^{'} \left(x',0,t\right)\right)=\left(\tilde{{\rm \Theta }}_{m,m+1} \left(x,0,t\right),\tilde{{\rm \Theta }}'_{m,m+1} \left(x',0,t\right)\right)=\] 
\begin{equation} \label{GrindEQ__4_48_} 
\left(\tilde{{\rm \Gamma }}_{m,m+1} \left(x,t,0\right),\tilde{{\rm \Gamma }}'_{m,m+1} \left(x',t,0\right)\right)=\left(\varphi _{m,m+1} \left(x,s\right),\varphi '_{m,m+1} \left(x',s\right)\right), 
\end{equation}

\[\left({\rm \Theta }_{m,m+1} \left(x,1,t\right),{\rm \Theta }_{m,m+1}^{'} \left(x',1,t\right)\right)=\left(\bar{{\rm \Theta }}_{m,m+1} \left(x,k\left(1,t\right)\right),\bar{{\rm \Theta }}_{m,m+1}^{'} \left(x',k\left(1,t\right)\right)\right)=\] 
\[\left(\bar{{\rm \Theta }}_{m,m+1} \left(x,1,t\right),\bar{{\rm \Theta }}_{m,m+1}^{'} \left(x',1,t\right)\right)=\left(\tilde{{\rm \Theta }}_{m,m+1} \left(x,1,t\right),\tilde{{\rm \Theta }}'_{m,m+1} \left(x',1,t\right)\right)=\] 
\[\left({\rm \Delta }_{m,m+1} \left(x,t,1\right),{\rm \Delta }'_{m,m+1} \left(x',t,1\right)\right)=\left(\tilde{\psi }_{m,m+1} \left(p_{\tilde{\psi }\left(m,m+1\right)} \left(x\right),s\right),\tilde{\psi }_{m,m+1}^{'} \left(p_{\tilde{\psi }\left(m,m+1\right)}^{'} \left(x'\right),s\right)\right)=\] 
{\tiny \[ \label{GrindEQ__4_49_} 
\left(\left(\tilde{\psi }_{m,m+1} \circ \left(p_{\tilde{\psi }\left(m,m+1\right)\left( \right),\psi(m+1)} \times 1_I\right)\left(p_{\psi(m+1)}(x),s\right)\right),\left(\tilde{\psi }_{m,m+1}' \circ \left(p'_{\tilde{\psi }\left(m,m+1\right),\psi(m+1)} \times 1_I\right)\right) \left( p'_{\psi(m+1)}(x'),s\right)\right)=\]}
\begin{equation}\label{GrindEQ__4_48_}
\left(\psi_{m,m+1}\left( p_{\psi(m+1)}(x),s\right), \psi'_{m,m+1} \left( p'_{\psi(m+1)}(x'),s\right)\right).
\end{equation} 
Therefore,  $\boldsymbol{\left(\Theta ,\Theta '\right)}:\boldsymbol{\left(\varphi ,\varphi '\right)}\to \boldsymbol{\left(\psi ,\psi '\right)}\cdot \mathbf{\left(p,p'\right)}$. 
\end{proof}

\begin{theorem}
If $\left(\mathbf{p},\mathbf{p}'\right)=\left\{\left(p_{\lambda } ,p_{\lambda }^{'} \right)\; \right\}:f\to \mathbf{f}$\textbf{ }is a strong fiber $\mathbf{ANR}$-expansion of a continuous map  $f:X\to X'$ of compact metric spaces, $\boldsymbol{\left(\varphi ,\varphi '\right)}:f\to \mathbf{g}$\textbf{ }is a coherent morphism, where $\mathbf{g}\in \mathbf{ CH\left(tow-Mor_{CM} \right)}$\textbf{ }and\textbf{ }$\boldsymbol{\left(\Psi _{1} ,\Psi' _{1} \right)}, \boldsymbol{\left(\Psi _{2} ,\Psi' _{2} \right)}:\mathbf{f}\to \mathbf{g}$\textbf{ }are coherent morphisms such that $\boldsymbol{\left(\varphi ,\varphi '\right)}$ is coherent homotopic to $\boldsymbol{\left(\Psi _{1} ,\Psi' _{1} \right)} \circ \left(\mathbf{p},\mathbf{p}'\right)$ and $\boldsymbol{\left(\Psi _{2} ,\Psi' _{2} \right)} \circ \left(\mathbf{p},\mathbf{p}'\right)$, then the morphisms  $\boldsymbol{\left(\Psi _{1} ,\Psi' _{1} \right)}$ and $\boldsymbol{\left(\Psi _{2} ,\Psi' _{2} \right)}$ are coherent homotopic.
\end{theorem}

\begin{proof}
Let 
\begin{equation}\label{GrindEQ__4_48.1_}
\boldsymbol{\left(\Theta ^{1} ,{\Theta ^{'1}}\right)}=\left\{ \left(   \Theta_m ^{1} ,{\Theta_m ^{'1}}\right), \left(\Theta_{m,m+1} ^{1} ,{\Theta_{m, m+1} ^{'1}} \right), \Theta \right\},
\end{equation} 
and 
\begin{equation}\label{GrindEQ__4_48.2_}
\boldsymbol{\left(\Theta ^{2} ,{\Theta ^{'2}}\right)}=\left\{ \left(   \Theta_m ^{2} ,{\Theta_m ^{'2}}\right), \left(\Theta_{m,m+1} ^{2} ,{\Theta_{m, m+1} ^{'2}} \right), \Theta \right\},
\end{equation}
are the corresponding coherent homotopies. For each index $m\in \mathbb N$ consider the morphisms  $\left({\rm \Theta }_{m}^{1} ,{\rm \Theta }_{m}^{'1} \; \right),\left({\rm \Theta }_{m}^{2} ,{\rm \Theta }_{m}^{'2} \; \right):f\times 1_{I} \to g_{m} .$ By definition, in this case we have
\begin{equation} \label{GrindEQ__4_50_} 
\left({\rm \Theta }_{m}^{1} \left(x,0\right),{\rm \Theta }_{m}^{'1} \left(x',0\right)\right)=\left(\varphi _{m} \left(x\right),\varphi _{m}' \left(x'\right)\right), 
\end{equation} 
\begin{equation} \label{GrindEQ__4_51_} 
\left({\rm \Theta }_{m}^{1} \left(x,1\right),{\rm \Theta }_{m}^{'1} \left(x',1\right)\right)=\left(\psi _{m}^{1} \left(p_{\psi ^{1} \left(m\right)} \right)\left(x\right),\psi _{m}^{'1} \left(p_{\psi ^{1} \left(m\right)}' \right)\left(x'\right)\right), 
\end{equation} 
\begin{equation} \label{GrindEQ__4_52_} 
\left({\rm \Theta }_{m}^{2} \left(x,0\right),{\rm \Theta }_{m}^{'2} \left(x',0\right)\right)=\left(\varphi _{m} \left(x\right),\varphi _{m}' \left(x'\right)\right), 
\end{equation} 
\begin{equation} \label{GrindEQ__4_53_} 
\left({\rm \Theta }_{m}^{2} \left(x,1\right),{\rm \Theta }_{m}^{'2} \left(x',1\right)\right)=\left(\psi _{m}^{2} \left(p_{\psi ^{2} \left(m\right)} \right)\left(x\right),\psi _{m}^{2} \left(p_{\psi ^{2} \left(m\right)}' \right)\left(x'\right)\right). 
\end{equation} 
Let $n=\max \left(\psi ^{1} \left(m\right),\psi ^{2} \left(m\right)\right)$ and define $\left(\tilde{\psi }_{m}^{1} ,\tilde{\psi }_{m}^{'1} \right),\; \left(\tilde{\psi }_{m}^{2} ,\tilde{\psi }_{m}^{'2} \right):f_{n} \to g_{m} $ by
\begin{equation} \label{GrindEQ__4_54_} 
\left(\tilde{\psi }_{m}^{1} \left(x\right),\tilde{\psi }_{m}^{'1} \left(x'\right)\right)=\left(\psi _{m}^{1} \left(p_{\psi ^{1} \left(m\right),n} \right)\left(x\right),\psi _{m}^{'1} \left(p_{\psi ^{1} \left(m\right),n}' \right)\left(x'\right)\right), 
\end{equation} 
\begin{equation} \label{GrindEQ__4_55_} 
\left(\tilde{\psi }_{m}^{2} \left(x\right),\tilde{\psi }_{m}^{'2} \left(x'\right)\right)=\left(\psi _{m}^{2} \left(p_{\psi ^{2} \left(m\right),n} \right)\left(x\right),\psi _{m}^{'2} \left(p_{\psi ^{2} \left(m\right),n}' \right)\left(x'\right)\right). 
\end{equation} 
Let $\left(\tilde{\Theta }_{m} ,\tilde{\Theta }_{m}^{'} \right):f\times 1_{I} \to g_{m} $ be the fiber homotopy defined by
\begin{equation} \label{GrindEQ__4_56_} 
\left(\tilde{{\rm \Theta }}_{m} \left(x,t\right),\tilde{{\rm \Theta }}_{m}^{'} \left(x',t\right)\right)=
\begin{cases}
{\left({\rm \Theta }_{m}^{1} \left(x,1-2t\right),{\rm \Theta }_{m}^{'1} \left(x',1-2t\right)\right),\; \; \; if \; 0\le t\le \frac{1}{2} \; } \\ {\left({\rm \Theta }_{m}^{2} \left(x,2t-1\right),{\rm \Theta }_{m}^{'2} \left(x',2t-1\right)\right),\; \; \;  if \; \frac{1}{2} \le t\le 1.}
\end{cases}  
\end{equation} 
By \eqref{GrindEQ__4_50_} and \eqref{GrindEQ__4_52_} $\left(\tilde{\Theta }_{m} ,\tilde{\Theta }_{m}^{'} \right)$ is well-defined. On the other hand, by \eqref{GrindEQ__4_51_}, \eqref{GrindEQ__4_53_}, \eqref{GrindEQ__4_54_} and \eqref{GrindEQ__4_55_} it is a homotopy between the morphisms $\left(\tilde{\psi }_{m}^{1} ,\tilde{\psi }_{m}^{'1} \right)\circ \left(p_{n} ,p'_{n} \right)$ and $\left(\tilde{\psi }_{m}^{2} ,\tilde{\psi }_{m}^{'2} \right)\circ \left(p_{n} ,p_{n}^{'} \right)$. Indeed, 
\[\left(\tilde{\Theta }_{m} \left(x,0\right),\tilde{\Theta }_{m}^{'} \left(x',0\right)\right)=\left({\rm \Theta }_{m}^{1} \left(x,1\right),{\rm \Theta }_{m}^{'1} \left(x',1\right)\right)=\] 
\[\left(\psi _{m}^{1} \left(p_{\psi ^{1} \left(m\right)} \right)\left(x\right),\psi _{m}^{'1} \left(p_{\psi ^{1} \left(m\right)}^{'} \right)\left(x'\right)\right)=\] 
\[\left(\psi _{m}^{1} \left(p_{\psi ^{1} \left(m\right),n} \right)\left(p_{n} \left(x\right)\right),\psi _{m}^{'1} \left(p_{\psi ^{1} \left(m\right),n}^{'} \right)\left(p_{n}^{'} \left(x'\right)\right)\right)=\] 
\begin{equation} \label{GrindEQ__4_57_} 
\left(\tilde{\psi }_{m}^{1} \left(p_{n} \left(x\right)\right),\tilde{\psi }_{m}^{'1} \left(p_{n}^{'} \left(x'\right)\right)\right), 
\end{equation} 
\[\left(\tilde{\Theta }_{m} \left(x,1\right),\tilde{\Theta }_{m}^{'} \left(x',1\right)\right)=\left({\rm \Theta }_{m}^{2} \left(x,1\right),{\rm \Theta }_{m}^{'2} \left(x',1\right)\right)=\] 
\[\left(\psi _{m}^{2} \left(p_{\psi ^{2} \left(m\right)} \right)\left(x\right),\psi _{m}^{'2} \left(p_{\psi ^{2} \left(m\right)}^{'} \right)\left(x'\right)\right)=\] 
\[\left(\psi _{m}^{2} \left(p_{\psi ^{2} \left(m\right),n} \right)\left(p_{n} \left(x\right)\right),\psi _{m}^{'2} \left(p_{\psi ^{2} \left(m\right),n}^{'} \right)\left(p_{n}^{'} \left(x'\right)\right)\right)=\] 
\begin{equation} \label{GrindEQ__4_58_} 
\left(\tilde{\psi }_{m}^{2} \left(p_{n} \left(x\right)\right),\tilde{\psi }_{m}^{'2} \left(p_{n}^{'} \left(x'\right)\right)\right). 
\end{equation} 
Therefore, by the property SF2) of  $\mathbf{\left(p,p'\right)}:f\to \mathbf{f}$, there exist $\tilde{{\rm \Delta }}\left({\rm m}\right)\ge n$ and a morphism $\left(\tilde{{\rm \Delta }}_{m} ,\tilde{{\rm \Delta }}_{m}^{'} \right):f_{\tilde{{\rm \Delta }}\left(m\right)} \to g_{m} $ such that 
\[\left(\tilde{{\rm \Delta }}_{m} \left(x,0\right),\tilde{{\rm \Delta }}_{m}^{'} \left(x',0\right)\right)=\left(\tilde{\psi }_{m}^{1} \left(p_{n,\tilde{{\rm \Delta }}\left(m\right)} \right)\left(x\right),\tilde{\psi }_{m}^{'1} \left(p_{n,\tilde{{\rm \Delta }}\left(m\right)}' \right)\left(x'\right)\right)=\] 
\[\left(\left(\psi _{m}^{1} \circ p_{\psi ^{1} \left(m\right),n} \right)\left(p_{n,\tilde{{\rm \Delta }}\left(m\right)} \right)\left(x\right),\left(\psi _{m}^{'1} \circ p_{\psi ^{1} \left(m\right),n}' \right)\left(p_{n,\tilde{{\rm \Delta }}\left(m\right)}' \right)\left(x'\right)\right)=\] 
\begin{equation} \label{GrindEQ__4_59_} 
\left(\psi _{m}^{1} \left(p_{\psi ^{1} \left(m\right),\tilde{{\rm \Delta }}\left(m\right)} \right)\left(x\right),\psi _{m}^{'1} \left(p_{\psi ^{1} \left(m\right),\tilde{{\rm \Delta }}\left(m\right)}' \right)\left(x'\right)\right),                                 
\end{equation} 
\[\left(\tilde{{\rm \Delta }}_{m} \left(x,1\right),\tilde{{\rm \Delta }}_{m}^{'} \left(x',1\right)\right)=\left(\tilde{\psi }_{m}^{2} \left(p_{n,\tilde{{\rm \Delta }}\left(m\right)} \right)\left(x\right),\tilde{\psi }_{m}^{'2} \left(p_{n,\tilde{{\rm \Delta }}\left(m\right)}' \right)\left(x'\right)\right)=\] 
\[\left(\left(\psi _{m}^{2} \circ p_{\psi ^{2} \left(m\right),n} \right)\left(p_{n,\tilde{{\rm \Delta }}\left(m\right)} \right)\left(x\right),\left(\psi _{m}^{'2} \circ p_{\psi ^{2} \left(m\right),n}' \right)\left(p_{n,\tilde{{\rm \Delta }}\left(m\right)}' \right)\left(x'\right)\right)=\] 
\begin{equation} \label{GrindEQ__4_60_} 
\left(\psi _{m}^{2} \left(p_{\psi ^{2} \left(m\right),\tilde{{\rm \Delta }}\left(m\right)} \right)\left(x\right),\psi _{m}^{'2} \left(p_{\psi ^{2} \left(m\right),\tilde{{\rm \Delta }}\left(m\right)}' \right)\left(x'\right)\right).                                 
\end{equation} 
\begin{equation} \label{GrindEQ__4_61_} 
\left(\tilde{{\rm \Delta }}_{m} ,\tilde{{\rm \Delta }}_{m}^{'} \right)\circ \left(p_{\tilde{{\rm \Delta }}\left(m\right)} ,p_{\tilde{{\rm \Delta }}\left(m\right)}' \right)\cong \left(\tilde{\Theta }_{m} ,\tilde{\Theta }_{m}^{'} \right)\left(rel\left\{f\times 1_{\partial I} \right\}\right). 
\end{equation} 
Consider the fiber homotopies  
\begin{equation} \label{GrindEQ__4_62_} 
\left({\rm \Theta }_{m,m+1}^{1} ,{\rm \Theta }_{m,m+1}^{'1} \; \right),\left({\rm \Theta }_{m,m+1}^{2} ,{\rm \Theta }_{m,m+1}^{'2} \; \right):f\times 1_{I} \to g_{m} . 
\end{equation} 
By definition, in this case we have
\begin{equation} \label{GrindEQ__4_63_} 
\left({\rm \Theta }_{m,m+1}^{1} \left(x,s,0\right),{\rm \Theta }_{m,m+1}^{'1} \left(x,s,0\right)\right)=\left({\rm \Theta }_{m}^{1} \left(x,s\right),{\rm \Theta }_{m}^{'1} \left(x',s\right)\right),                          
\end{equation} 
\[\left({\rm \Theta }_{m,m+1}^{1} \left(x,s,1\right),{\rm \Theta }_{m,m+1}^{'1} \left(x,s,1\right)\right)=\]
\begin{equation} \label{GrindEQ__4_64_} 
\left(q_{m,m+1} \left({\rm \Theta }_{m+1}^{1} \left(x\right),s\right),q_{m,m+1}^{'} \left({\rm \Theta }_{m}^{'1} \left(x'\right),s\right)\right),      
\end{equation} 
\begin{equation} \label{GrindEQ__4_65_} 
\left({\rm \Theta }_{m,m+1}^{1} \left(x,0,t\right),{\rm \Theta }_{m,m+1}^{'1} \left(x,0,t\right)\right)=\left(\varphi _{m,m+1} \left(x,t\right),\varphi _{m,m+1}' \left(x',t\right)\right),                
\end{equation} 
\[\left({\rm \Theta }_{m,m+1}^{1} \left(x,1,t\right),{\rm \Theta }_{m,m+1}^{'1} \left(x,1,t\right)\right)=\]
\begin{equation} \label{GrindEQ__4_66_} 
\left(\psi _{m,m+1}^{1} \left(p_{\psi ^{1} \left(m+1\right)} \left(x\right),t\right),\psi _{m,m+1}^{'1} \left(p_{\psi ^{1} \left(m+1\right)}' \left(x'\right),t\right)\right),   
\end{equation} 
\begin{equation} \label{GrindEQ__4_67_} 
\left({\rm \Theta }_{m,m+1}^{2} \left(x,s,0\right),{\rm \Theta }_{m,m+1}^{'2} \left(x,s,0\right)\right)=\left({\rm \Theta }_{m}^{2} \left(x,s\right),{\rm \Theta }_{m}^{'2} \left(x',s\right)\right),                          
\end{equation} 
\[\left({\rm \Theta }_{m,m+1}^{2} \left(x,s,1\right),{\rm \Theta }_{m,m+1}^{'2} \left(x,s,1\right)\right)=\]
\begin{equation} \label{GrindEQ__4_68_} 
\left(q_{m,m+1} \left({\rm \Theta }_{m+1}^{2} \left(x\right),s\right),q_{m,m+1}' \left({\rm \Theta }_{m}^{'2} \left(x'\right),s\right)\right),      
\end{equation} 
\begin{equation} \label{GrindEQ__4_69_} 
\left({\rm \Theta }_{m,m+1}^{2} \left(x,0,t\right),{\rm \Theta }_{m,m+1}^{'2} \left(x,0,t\right)\right)=\left(\varphi _{m,m+1} \left(x,t\right),\varphi _{m,m+1}' \left(x',t\right)\right),                
\end{equation} 
\[\left({\rm \Theta }_{m,m+1}^{2} \left(x,1,t\right),{\rm \Theta }_{m,m+1}^{'2} \left(x,1,t\right)\right)=\]
\begin{equation} \label{GrindEQ__4_70_} 
\left(\psi _{m,m+1}^{2} \left(p_{\psi ^{2} \left(m+1\right)} \left(x\right),t\right),\psi _{m,m+1}^{'2} \left(p_{\psi ^{2} \left(m+1\right)}' \left(x'\right),t\right)\right).     
\end{equation} 
Let $\left(\tilde{\Theta }_{m,m+1} ,\tilde{\Theta }_{m,m+1}^{'} \right):f\times 1_{I} \times 1_{I} \to g_{m} $ be a fiber homotopy defined by
\[\left(\tilde{{\rm \Theta }}_{m,m+1} \left(x,s,t\right),\tilde{{\rm \Theta }}_{m.m+1}^{'} \left(x',s,t\right)\right)=\] 
\begin{equation} \label{GrindEQ__4_71_} 
=\begin{cases}{\left({\rm \Theta }_{m,m+1}^{1} \left(x,1-2s,t\right),{\rm \Theta }_{m,m+1}^{'1} \left(x',1-2s,t\right)\right),\; \; if\; 0\le s\le \frac{1}{2} \; } \\ {\left({\rm \Theta }_{m,m+1}^{2} \left(x,2s-1,t\right),{\rm \Theta }_{m,m+1}^{'2} \left(x',2s-1,t\right)\right),\; \; if\;  \frac{1}{2} \le s\le 1.} 
\end{cases}  
\end{equation} 
By \eqref{GrindEQ__4_65_} and \eqref{GrindEQ__4_69_}, it is well-defined. 
Let $\left({\rm \Lambda }_{m} ,{\rm \Lambda }_{m}^{'} \; \right):f\times 1_{I} \times 1_{I} \to g_{m} $ be a fiber homotopy which realizes \eqref{GrindEQ__4_61_}, i.e.
\begin{equation} \label{GrindEQ__4_72_} 
\left({\rm \Lambda }_{m} \left(x,s,0\right),{\rm \Lambda }_{m}^{'} \left(x',s,0\right)\right)=\left(\tilde{{\rm \Delta }}_{m} \left(p_{\tilde{{\rm \Delta }}\left(m\right)} \left(x\right),s\right),\tilde{{\rm \Delta }}_{m}^{'} \left(p_{\tilde{{\rm \Delta }}\left(m\right)}' \left(x'\right),s\right)\right), 
\end{equation} 
\begin{equation} \label{GrindEQ__4_73_} 
\left({\rm \Lambda }_{m} \left(x,s,1\right),{\rm \Lambda }_{m}^{'} \left(x',s,1\right)\right)=\left(\tilde{\Theta }_{m} \left(x,s\right),\tilde{\Theta }_{m}^{'} \left(x',s\right)\right), 
\end{equation} 
\begin{equation} \label{GrindEQ__4_74_} 
\left({\rm \Lambda }_{m} \left(x,s,t\right),{\rm \Lambda }_{m}^{'} \left(x',s,t\right)\right)=\left({\rm \Lambda }_{m} \left(x,s,0\right),{\rm \Lambda }_{m}^{'} \left(x',s,0\right)\right),\; if\; s\in \partial I. 
\end{equation} 
Define the fiber homotopy $\left(\tilde{{\rm \Gamma }}_{m,m+1} ,\tilde{{\rm \Gamma }}_{m,m+1}^{'} \right):f\times 1_{I} \times 1_{I} \to g_{m} $ by the formula
\[\left(\tilde{{\rm \Gamma }}_{m,m+1} \left(x,s,t\right),\tilde{{\rm \Gamma }}_{m,m+1}^{'} \left(x',s,t\right)\right)=\] 
\begin{equation} \label{GrindEQ__4_75_}
\begin{cases}
\left( \Lambda_m(x,s,3t),\Lambda'_m(x',s,3t)\right) \; \; \; \; \; \; \; \; \; \; \; \; \; \; \; \; \; \; \; \; \; \; \; \; \; \; \; \; \; \; \; \; \; \; \; \; if \; 0 \le s \le \frac{1}{3}\\
\left(\tilde{\Theta}_{,m+1}(x,s,3t-1),{\tilde{\Theta}}'_{m,m+1}(x',s,3t-1)\right) \; \; \; \; \; \; \; \; \; \; \; if \; \frac{1}{3} \le s \le \frac{2}{3}\\
\left(q_{m,m+1} \Lambda_m(x,s,3-3t),q'_{m,m+1}\Lambda'_m(x',s,3-3t)\right) \; \; \;  if \; \frac{2}{3} \le s \le 1.\\
\end{cases} 
\end{equation} 
By \eqref{GrindEQ__4_73_}, \eqref{GrindEQ__4_71_}, \eqref{GrindEQ__4_63_}, \eqref{GrindEQ__4_67_}, \eqref{GrindEQ__4_56_}, \eqref{GrindEQ__4_63_}, \eqref{GrindEQ__4_67_}  and \eqref{GrindEQ__4_74_} the  morphism  $\left(\tilde{{\rm \Gamma }}_{m,m+1} ,\tilde{{\rm \Gamma }}_{m,m+1}^{'} \right)\; $ is well defined. On the other hand, by \eqref{GrindEQ__4_74_}, for each $x\in X$ and $x'\in X'$ the homotopies $\tilde{{\rm \Gamma }}_{m,m+1} $ and $\tilde{{\rm \Gamma }}_{m,m+1}^{'} $ are constant on the subspaces $\left\{x\right\}\times \left\{0\right\}\times \left[0;\frac{1}{3} \right]$, $\left\{x\right\}\times \left\{1\right\}\times \left[0;\frac{1}{3} \right]$, $\left\{x\right\}\times \left\{0\right\}\times \left[\frac{2}{3} ;1\right]$, $\left\{x\right\}\times \left\{1\right\}\times \left[\frac{2}{3} ;1\right]$ and  $\left\{x'\right\}\times \left\{0\right\}\times \left[0;\frac{1}{3} \right]$, $\left\{x'\right\}\times \left\{1\right\}\times \left[0;\frac{1}{3} \right]$, $\left\{x'\right\}\times \left\{0\right\}\times \left[\frac{2}{3} ;1\right]$, $\left\{x'\right\}\times \left\{1\right\}\times \left[\frac{2}{3} ;1\right]$, respectively. Consider the quotient space $Z$ of $I\times I$, where $\left\{0\right\}\times \left[0;\frac{1}{3} \right]$, $\left\{1\right\}\times \left[0;\frac{1}{3} \right]$, $\left\{0\right\}\times \left[\frac{2}{3} ;1\right]$ and $\left\{1\right\}\times \left[\frac{2}{3} ;1\right]$ are identified with the points $\left(0,\frac{1}{3} \right)$, $\left(1,\frac{1}{3} \right)$, $\left(0,\frac{2}{3} \right)$ and $\left(1,\frac{2}{3} \right)$, respectively. In this case, there exists a morphism $\left(\bar{{\rm \Gamma }}_{m,m+1} ,\bar{{\rm \Gamma }}_{m,m+1}^{'} \right):f\times 1_{Z} \to g_{m} $ such that 
\begin{equation} \label{GrindEQ__4_76_} 
\left(\tilde{{\rm \Gamma }}_{m,m+1} ,\tilde{{\rm \Gamma }}_{m,m+1}^{'} \; \right)=\left(\bar{{\rm \Gamma }}_{m,m+1} ,\bar{{\rm \Gamma }}_{m,m+1}^{'} \right)\circ \left(1_{X} \times q,\; 1_{X} \times q\right),                 
\end{equation} 
where $q:I\times I\to Z$ is a quotient map. Let $k:I\times I\to Z$ be a homeomorphism such that
\begin{equation} \label{GrindEQ__4_77_} 
k\left(s,0\right)=q\left(s,0\right),                           \end{equation} 
\begin{equation} \label{GrindEQ__4_78_} 
k\left(s,1\right)=q\left(s,1\right), 
\end{equation} 
\begin{equation} \label{GrindEQ__4_79_} 
k\left(s,1\right)=q\left(0,\frac{1+t}{3} \right),                                            
\end{equation} 
\begin{equation} \label{GrindEQ__4_80_} 
k\left(s,1\right)=q\left(1,\frac{1+t}{3} \right). 
\end{equation} 
Let $\left({\rm \Gamma }_{m,m+1} ,{\rm \Gamma }_{m,m+1}^{'} \right):f\times 1_{Z} \to g_{m} $ be a morphism given by 
\begin{equation} \label{GrindEQ__4_81_} 
\left({\rm \Gamma }_{m,m+1} \left(x,s,t\right),{\rm \Gamma }_{m,m+1}^{'} \left(x',s,t\right)\right)=\left(\bar{{\rm \Gamma }}_{m,m+1} \left(x,k\left(s,t\right)\right),\bar{{\rm \Gamma }}_{m,m+1}^{'} \left(x',k\left(s,t\right)\right)\right). 
\end{equation} 
Let $k=\max \left(\tilde{{\rm \Delta }}\left(m\right),\; \tilde{{\rm \Delta }}\left(m+1\right)\psi ^{1} \left(m\right),\psi ^{2} \left(m\right)\right)$. Define 
\[\left(\sigma _{k} \left(x,s,t\right),\sigma _{k}' \left(x',s,t\right)\right)=\] 
\begin{equation} \label{GrindEQ__4_82_}
=\begin{cases}
\left(\tilde{\Delta}_m\left(p_{\tilde{\Delta}_m,l}(x),s\right),\tilde{\Delta}'_m\left(p'_{\tilde{\Delta}_m,l}(x'),s\right) \right), \; \; \; \; \; \; \; \; \; \; \; \; \; \; \; \; \; \; \; \; \;\; \; \; \; \; \; \; \; \;\; \; \; \;  if \; t=0\\
\left(q_{m,m+1}\tilde{\Delta}_{m+1}\left(p_{\tilde{\Delta}_{m+1},l}(x),s\right),q'_{m+1}\tilde{\Delta}'_{m+1}\left(p'_{\tilde{\Delta}_{m+1},l}(x'),s\right) \right), \; \; \; if \; t=1\\
\left(\psi^1_{m,m+1}\left(p_{\psi^1_{m,m+1},k}(x),s\right),\psi^{'1}_{m,m+1}\left(p'_{\psi^{'1}_{m,m+1},k}(x),s\right) \right), \; \; \; \; \; \; \; \; \;  if \; s=0\\
\left(\psi^2_{m,m+1}\left(p_{\psi^2_{m,m+1},k}(x),s\right),\psi^{'2}_{m,m+1}\left(p'_{\psi^{'2}_{m,m+1},k}(x),s\right) \right), \; \; \; \; \; \; \; \; \; if \; s=1.
\end{cases} 
\end{equation} 
In this case, we have
\begin{equation} \label{GrindEQ__4_83_} 
\left({\rm \Gamma }_{m,m+1|X \times \partial I^2} ,{\rm \Gamma }_{m,m+1 X' \times \partial I^2}^{'} \; \right)=\left(\sigma_k, \sigma'_k \right)\circ \left(p_k \times 1_{\partial I^2}, p'_k \times 1_{\partial I^2}\right). 
\end{equation} 
Therefore, by theorem 2.5 there exist $\tilde{{\rm \Delta }}\left(m,m+1\right)\ge k$ and a morphism $\left(\tilde{{\rm \Delta }}_{m,m+1} ,\tilde{{\rm \Delta }}_{m,m+1}^{'} \right):f_{\tilde{{\rm \Delta }}\left(m,m+1\right)} \times 1_{I} \times 1_{I} \to g_{m} $ such that 
\[\left({\tilde{\Delta} }_{m,m+1|X_{\tilde{\Delta}(m,m+1) } \times \partial I^2} ,{\tilde{\Delta} }'_{m,m+1|X'_{\tilde{\Delta} (m,m+1)} \times \partial I^2} \right)=\]
\begin{equation} \label{GrindEQ__4_84_} 
\left(\sigma_, \sigma'_k \right) \circ \left(p_{k,\tilde{{\rm \Delta }}\left(m,m+1\right)} \times 1_{\partial I^{2} } ,p_{k,\tilde{{\rm \Delta }}\left(m,m+1\right)}' \times 1_{\partial I^{2} } \; \right), 
\end{equation} 
\[\left(\tilde{{\rm \Delta }}_{m,m+1} ,\tilde{{\rm \Delta }}_{m,m+1}^{'} \right)\circ \left(p_{\tilde{{\rm \Delta }}\left(m,m+1\right)} \times 1_{I^{2} } ,p_{\tilde{{\rm \Delta }}\left(m,m+1\right)}' \times 1_{I^{2} } \right)\cong\]
\begin{equation} \label{GrindEQ__4_85_} 
\cong \left({\rm \Gamma }_{m,m+1} ,{\rm \Gamma }_{m,m+1}^{'} \right)\left(rel\left\{f\times 1_{\partial I^{2} } \right\}\right).\;  
\end{equation} 
Let ${\rm \Delta }\left(m+1\right)=\mathop{\max }\limits_{n\le m} \left\{\tilde{{\rm \Delta }}\left(n,n+1\right)\right\}$, then ${\rm \Delta }:\mathbb N\to \mathbb N$ will be an increasing function. Now define the morphisms
\begin{equation} \label{GrindEQ__4_86_} 
\left({\rm \Delta }_{m} ,{\rm \Delta }_{m}^{'} \right):f_{{\rm \Delta }(\left(m\right)} \times 1_{I} \to g_{m},                                            
\end{equation} 
\begin{equation} \label{GrindEQ__4_87_} 
\left({\rm \Delta }_{m,m+1} ,{\rm \Delta }_{m,m+1}^{'} \right):f_{{\rm \Delta }(\left(m\right)} \times 1_{I} \times 1_{I} \to g_{m},
\end{equation} 
by the following 
\begin{equation} \label{GrindEQ__4_88_} 
\left({\rm \Delta }_{m} ,{\rm \Delta }_{m}^{'} \right)=\left(\tilde{{\rm \Delta }}_{m} ,\tilde{{\rm \Delta }}_{m}^{'} \right)\circ \left(p_{\tilde{{\rm \Delta }}\left(m\right),{\rm \Delta }(\left(m\right)} \times 1_{I} ,p_{\tilde{{\rm \Delta }}\left(m\right),{\rm \Delta }(\left(m\right)}' \times 1_{I} \; \right), 
\end{equation} 
\[\left({\rm \Delta }_{m,m+1} ,{\rm \Delta }_{m,m+1}^{'} \right)=\]
\begin{equation} \label{GrindEQ__4_89_} 
\left(\tilde{{\rm \Delta }}_{m,m+1} ,\tilde{{\rm \Delta }}_{m,m+1}^{'} \right)\circ \left(p_{\tilde{{\rm \Delta }}\left(m+1\right),{\rm \Delta }(\left(m+1\right)} \times 1_{I^{2} } ,p_{\tilde{{\rm \Delta }}\left(m+1\right),{\rm \Delta }(\left(m+1\right)}' \times 1_{I^{2} } \; \right). 
\end{equation} 
In this case, we can show that $\boldsymbol{ \left(\Theta^1 ,\Theta^{'1}\right)}=\left\{ \left(\Theta^1_m ,\Theta^{'1}_m \right),  \left(\Theta^1_{m,m+1} ,\Theta^{'1}_{m,m+1}, \Theta \right)  \right\}$ is a coherent homotopy between the morphisms  $\boldsymbol{\left(\Psi _{1} ,\Psi _{1}^{'} \right)}$ and $\boldsymbol{\left( \Psi _{2} ,\Psi _{2}^{'} \right)}$. 
\end{proof}

By theorems 3.1 and 3.2, if $\mathbf{\left(p_{1} ,p'_{1} \right)}:f\to \mathbf{f_{1}} $ and $\mathbf{\left(p_{2} ,p'_{2} \right)}:f\to \mathbf{f_{2}}$\textbf{ }are two expansions of $f$, then there exists a unique isomorphism $\mathbf{\left(i,i'\right)}:\mathbf{f_{1}} \to \mathbf{f_{2}} $ such that 
\begin{equation} \label{GrindEQ__4_90_} 
\mathbf {\left[\left(i, i'\right)\right]} \circ \mathbf{\left(p_{1} ,p_{1}' \right)}=\mathbf{\left(p_{2} ,p_{2}' \right)} .    \end{equation} 
For each $f,g\in \mathbf{Mor_{CM}} $\textbf{ }consider the set of all triples 
\begin{equation} \label{GrindEQ__4_91_} 
\left(\mathbf{\left(p,{ p}'\right)},\mathbf{\left(q,{ q}'\right)},\left[\boldsymbol{\left(\Psi ,\Psi '\right)}\right]\right),                              
\end{equation} 
where $\mathbf{\left(p,p'\right)}:f\to \mathbf{f}$ and $\mathbf{\left(q,q'\right)}:g\to \mathbf{g}$ are strong fiber expansions and $\left[\boldsymbol{\left(\psi ,\psi '\right)}\right]$ is a coherent homotopy class of the coherent morphism $\boldsymbol{\left(\Psi ,\Psi '\right)}:\mathbf{f}\to \mathbf{g}$. Two such triples  $\left(\mathbf{\left(p_{1} ,{ p}_{1}' \right)},\mathbf{\left(q_{1} ,{\ q}_{1}' \right)},\left[\boldsymbol{\left(\Psi _{1} ,\Psi _{1}^{'} \right)}\right]\right)$ and $\left(\mathbf{\left(p_{2} ,{ p}_{2}' \right)},\mathbf{\left(q_{2} ,{\ q}_{2}' \right)},\left[\boldsymbol{\left(\Psi _{2} ,\Psi _{2}^{'} \right)}\right]\right)$ are called equivalent if 
\begin{equation} \label{GrindEQ__4_92_} 
\left[\mathbf{\left(j,{ j}'\right)}\right]\circ \left[\boldsymbol{\left(\Psi _{1} ,\Psi _{1}^{'} \right)}\right]=\left[\boldsymbol{\left(\Psi _{2} ,\Psi _{2}^{'} \right)}\right]\circ \left[\mathbf{\left(i,{ i}'\right)}\right] ,               
\end{equation} 
where $\left[\mathbf{\left(i,{ i}'\right)}\right]: \mathbf{f_{1}} \to \mathbf{f_{2}} $ and $\left[\mathbf{\left(j,{ j}'\right)}\right]; \mathbf{g_{1}} \to \mathbf{g_{2}}$ are isomorphisms. The equivalence class of the triple  $\left(\mathbf{\left(p, p'\right)},\mathbf{\left(q, q'\right)},\boldsymbol{\left[\left(\Psi ,\Psi '\right)\right]}\right)$ is denoted by $F:f\to g$ and is called a strong shape morphism from $f$ to $g$. Let the category $\mathbf{SSh\left(Mor_{CM} \right)}$ of all continuous maps of compact metric spaces and all strong shape morphisms be called the strong fiber shape category of  $\mathbf{Mor_{CM}} $.

$\mathbf{Remark:}$ Note that there exists a functorial relation between the strong fiber shape category and fiber shape category.

\end{document}